\documentclass[11pt, onecolumn]{article}

\usepackage{constants}

\usepackage{stmaryrd}
\usepackage{tikz}
\usetikzlibrary{shapes.misc}
\tikzset{cross/.style={cross out, draw=black, minimum size=2*(#1-\pgflinewidth), inner sep=0pt, outer sep=0pt},
cross/.default={1pt}}

\usepackage{pstricks}
\usepackage[utf8]{inputenc}
\usepackage[T1]{fontenc}
\usepackage{lmodern}
\usepackage{xr}
\usepackage{bold-extra}
\usepackage{dsfont}

\usepackage{amsmath}
\usepackage{amssymb}
\usepackage{mathrsfs}
\usepackage{mathtools}

\usepackage{amsfonts}
\usepackage{amsthm}

\usepackage{enumerate}
\usepackage{multirow}
\usepackage{graphicx}
\usepackage{hyperref}
\usepackage{xcolor}
\hypersetup{
    colorlinks,
    linkcolor={black},
    citecolor={black},
    urlcolor={black}
}

\makeatletter
\def\@seccntformat#1{\@ifundefined{#1@cntformat}%
   {\csname the#1\endcsname\quad}  
   {\csname #1@cntformat\endcsname}
}
\let\oldappendix\appendix 
\renewcommand\appendix{%
    \oldappendix
    \newcommand{\section@cntformat}{\appendixname: }
}
\makeatother

\usepackage[top=2.5cm,bottom=2.5cm,left=2cm,right=2cm,marginparwidth=2.5cm]{geometry}
\reversemarginpar

\makeatletter
\renewcommand{\paragraph}{%
  \@startsection{paragraph}{4}%
  {\z@}{1.5ex \@plus 1ex \@minus .2ex}{-1em}%
  {\normalfont\normalsize\bfseries}%
}
\makeatother

{\theoremstyle{plain}
\newtheorem{Proposition}{\textbf{Proposition}}[section]

\newtheorem{Lemma}[Proposition]{\textbf{Lemma}}

\newtheorem{Corollary}[Proposition]{\textbf{Corollary}}
\newtheorem{Assumption}{\textbf{Assumption}} }

{\theoremstyle{plain}
\newtheorem{Theorem}[Proposition]{Theorem}}

{\theoremstyle{definition}
\newtheorem{Definition}[Proposition]{Definition}}

{\theoremstyle{remark}

\newtheorem{Remark}[Proposition]{Remark}

}

\numberwithin{equation}{section}

\newcommand{\eps}{\varepsilon}

\newcommand{\bbC}{\mathbb{C}}

\newcommand{\bbG}{\mathbb{G}}

\newcommand{\bbK}{\mathbb{K}}

\newcommand{\bbR}{\mathbb{R}}

\newcommand{\N}{\mathbb{N}}
\newcommand{\Z}{\mathbb{Z}}
\newcommand{\R}{\mathbb{R}}

\renewcommand{\S}{\mathbb S} 

\newcommand{\cE}{\mathcal{E}}

\newcommand{\cX}{\mathcal{X}}

\newcommand{\rL}{\mathrm{L}}

\let\limsup\relax
\let\liminf\relax
\DeclareMathOperator* \limsup {\overline{lim}}
\DeclareMathOperator* \liminf {\underline{lim}}
\DeclareMathOperator*{\argmax}{arg\,max}
\DeclareMathOperator*{\argmin}{arg\,min}

\newcommand{\inclim}[1]{\lim_{#1}\!\!\uparrow\!}
\newcommand{\declim}[1]{\lim_{#1}\!\!\downarrow\!}

\let\originalleft\left
\let\originalright\right
\renewcommand{\left}{\mathopen{}\mathclose\bgroup\originalleft}
\renewcommand{\right}{\aftergroup\egroup\originalright}

\newcommand{\p}[1]{\left( #1 \right)}
\newcommand{\acc}[1]{\left\{ #1 \right\}}
\newcommand{\cro}[1]{\left[ #1 \right]}

\newcommand{\set}[2]{\acc{#1 \;\middle\vert\; #2 } }

\newcommand{\symdif}{\mathbin{\vartriangle}}

\newcommand{\ind}[1]{\mathds{1}_{#1}}

\newcommand{\dpe}{\coloneqq}

\newcommand{\eol}{\nonumber\\}

\def\restriction#1#2{\mathchoice
              {\setbox1\hbox{${\displaystyle #1}_{\scriptstyle #2}$}
              \restrictionaux{#1}{#2}}
              {\setbox1\hbox{${\textstyle #1}_{\scriptstyle #2}$}
              \restrictionaux{#1}{#2}}
              {\setbox1\hbox{${\scriptstyle #1}_{\scriptscriptstyle #2}$}
              \restrictionaux{#1}{#2}}
              {\setbox1\hbox{${\scriptscriptstyle #1}_{\scriptscriptstyle #2}$}
              \restrictionaux{#1}{#2}}}
\def\restrictionaux#1#2{{#1\,\smash{\vrule height .8\ht1 depth .85\dp1}}_{\,#2}} 

\newcommand{\module}[1]{\left\lvert #1 \right\rvert}
\newcommand{\norme}[2][]{\left\| #2 \right\|_{#1}}

\newcommand{\ceil}[1]{\!\left\lceil #1 \right\rceil\!}
\newcommand{\floor}[1]{\!\left\lfloor #1 \right\rfloor\!}

\newcommand{\ball}[2][]{\mathrm{B}_{#1}\p{#2}}
\newcommand{\clball}[2][]{\overline{\mathrm{B}}_{#1}\p{#2}}

\newcommand{\intervalle}[4]{#1#2\,,#3#4}

\newcommand{\intervalleff}[2]{\intervalle{\left[}{#1}{#2}{\right]}}
\newcommand{\intervalleof}[2]{\intervalle{\left(}{#1}{#2}{\right]}}
\newcommand{\intervallefo}[2]{\intervalle{\left[}{#1}{#2}{\right)}}
\newcommand{\intervalleoo}[2]{\intervalle{\left(}{#1}{#2}{\right)}}

\newcommand{\intint}[2]{\left\llbracket#1\,,#2\right\rrbracket}

\renewcommand{\d}{\mathrm{d}}

\newcommand{\base}[1]{\mathrm e_{#1}}

\newcommand\ps[2]{\left\langle #1, #2 \right\rangle}

\newcommand{\Borel}[1]{\mathcal{B}\left( #1 \right) }

\newcommand{\Leb}{\operatorname{Leb}}

\newcommand{\Dirac}[1]{\delta_{#1}}

\newcommand{\E}[2][]{\mathbb{E}_{#1} \left[ #2\right]}

\newcommand{\Pb}[2][]{\mathbb{P}_{#1}\left( #2\right)}

\newcommand{\Path}[1]{\overset{#1}{\rightsquigarrow}}

\newcommand{\concat}{\mathbin{*}}

\newcommand{\pro}{\mathcal{N}}
\newcommand{\projpro}{\mathcal{N^*}}
\newcommand{\ProSpace}[1][\R^d \times \intervalleoo0\infty]{\mathbf{N}\p{#1}}

\newcommand{\MM}[1][\pro]{\mathrm{M}_{ {#1} } }
\newcommand{\FMM}[2][\pro]{\mathrm{M}_{ {#1}^{(#2)} } }

\newcommand{\Trans}{\theta}

\newcommand{\AUXAnimal}{\mathcal{A}}
\newcommand{\AUXPath}{\mathcal{P}}

\newcommand{\AUXGeneric}{\mathcal{G}}
\newcommand{\AUXMassAnimal}{\mathrm{A}}
\newcommand{\AUXMassPath}{\mathrm{P}}

\newcommand{\AUXMassGeneric}{\mathrm{G}}
\newcommand{\AUXLimAnimal}{\mathbf{A}}
\newcommand{\AUXLimPath}{\mathbf{P}}

\newcommand{\AUXLimGeneric}{\mathbf{G}}
\newcommand{\AUXBiancre}[3]{\p{#1 \leftrightarrow #2, #3}}

\newcommand{\AUXDiamant}[2]{#1^{#2}}
\newcommand{\AUXFree}[1]{#1}

\newcommand{\Mass}[1]{\mathbf{m}\p{#1}}

\newcommand{\SetADF}[3]{\AUXFree{\AUXAnimal}\AUXBiancre{#1}{#2}{#3} }

\newcommand{\SetAUF}[1]{\AUXFree{\AUXAnimal}\p{#1} }
\newcommand{\SetADR}[4]{\AUXAnimal_{#1}\AUXBiancre{#2}{#3}{#4} }

\newcommand{\SetPDF}[3]{\AUXFree{\AUXPath}\AUXBiancre{#1}{#2}{#3} }

\newcommand{\SetPUF}[1]{\AUXFree{\AUXPath}\p{#1} }
\newcommand{\SetPDR}[4]{\AUXPath_{#1}\AUXBiancre{#2}{#3}{#4} }
\newcommand{\SetGDC}[4]{\AUXDiamant{\AUXGeneric}{#1}\AUXBiancre{#2}{#3}{#4} }
\newcommand{\SetGDF}[3]{\AUXFree{\AUXGeneric}\AUXBiancre{#1}{#2}{#3} }

\newcommand{\SetGUF}[1]{\AUXFree{\AUXGeneric}\p{#1} }
\newcommand{\SetGDR}[4]{\AUXGeneric_{#1}\AUXBiancre{#2}{#3}{#4} }
\newcommand{\SetADFalt}[3]{\AUXFree{\AUXAnimal}^{*}\AUXBiancre{#1}{#2}{#3} }
\newcommand{\SetAUFalt}[1]{\AUXFree{\AUXAnimal}^{*}\p{#1} }

\newcommand{\SetLADF}[3]{\AUXFree{\AUXAnimal}^{\mathrm L}\AUXBiancre{#1}{#2}{#3} }

\newcommand{\SetLPUF}[1]{\AUXFree{\AUXPath}^{\mathrm L}\p{#1} }

\newcommand{\MassADC}[4]{\AUXDiamant{\AUXMassAnimal}{#1}\AUXBiancre{#2}{#3}{#4} }
\newcommand{\MassADF}[3]{\AUXFree{\AUXMassAnimal}\AUXBiancre{#1}{#2}{#3} }

\newcommand{\MassAUF}[1]{\AUXFree{\AUXMassAnimal}\p{#1} }
\newcommand{\MassADR}[4]{\AUXMassAnimal_{#1}\AUXBiancre{#2}{#3}{#4} }

\newcommand{\MassPDF}[3]{\AUXFree{\AUXMassPath}\AUXBiancre{#1}{#2}{#3} }

\newcommand{\MassPUF}[1]{\AUXFree{\AUXMassPath}\p{#1} }
\newcommand{\MassPDR}[4]{\AUXMassPath_{#1}\AUXBiancre{#2}{#3}{#4} }
\newcommand{\MassGDC}[4]{\AUXDiamant{\AUXMassGeneric}{#1}\AUXBiancre{#2}{#3}{#4} }
\newcommand{\MassGDF}[3]{\AUXFree{\AUXMassGeneric}\AUXBiancre{#1}{#2}{#3} }

\newcommand{\MassGUF}[1]{\AUXFree{\AUXMassGeneric}\p{#1} }
\newcommand{\MassGDR}[4]{\AUXMassGeneric_{#1}\AUXBiancre{#2}{#3}{#4} }
\newcommand{\MassADFpen}[4]{\AUXFree{\AUXMassAnimal}^{(#4)}\AUXBiancre{#1}{#2}{#3} }
\newcommand{\MassAUFpen}[2]{\AUXFree{\AUXMassAnimal}^{(#2)}\p{#1} }

\newcommand{\MassLAUF}[1]{\mathrm{A^L}\p{#1} }
\newcommand{\MassLADF}[3]{\AUXFree{\AUXMassAnimal}^{\mathrm L}\AUXBiancre{#1}{#2}{#3} }
\newcommand{\MassLPUF}[1]{\mathrm{P^L}\p{#1} }

\newcommand{\MassLsaP}{\mathrm{P^L_{sa}}}
\newcommand{\MassLGUF}[1]{\mathrm{G^L}\p{#1} }
\newcommand{\MassLGDF}[3]{\AUXFree{\AUXMassGeneric}^{\mathrm L}\AUXBiancre{#1}{#2}{#3} }

\newcommand{\LimMassGDC}[1]{\AUXDiamant{\AUXLimGeneric}{#1} }
\newcommand{\LimMassGDCab}[2]{\AUXDiamant{\AUXLimGeneric}{#1}_{#2} }
\newcommand{\LimMassG}{\AUXLimGeneric }

\newcommand{\LimMassA}{\AUXLimAnimal }

\newcommand{\LimMassP}{\AUXLimPath}

\newcommand{\LimMassLA}{\AUXLimAnimal^{\mathrm L}}

\newcommand{\LimMassLG}{\AUXLimGeneric^{\mathrm L}}

\newcommand{\Diamant}[3]{\operatorname{Diam}^{#1}(#2 \leftrightarrow #3)}
\newcommand{\AntiDiamant}[3]{\operatorname{Diam}_{#1}(#2 \leftrightarrow #3)}
\newcommand{\Cone}[3]{\operatorname{Cone}^{#1}(#2, #3)}

\newcommand{\GeneralUB}{\mathbf{M}}
\newcommand{\Yinf}{\underline{Y}}
\newcommand{\Ysup}{\overline{Y}}

\newcommand{\TC}[1]{\gamma\p{#1} }

\newcommand{\Sym}[2][]{\mathfrak{S}_{#1}\p{#2} }

\makeatletter
\newcommand{\SymbolsX}[1]{%
  \ensuremath{%
    \ifcase#1
    \or 
      *%
    \or 
      \dagger
    \or 
      \ddagger
    \or 
      **
    \or 
      \dagger\dagger
    \or 
      \ddagger \ddagger
    \or 
      \diamond 
    \fi
  }%
}   
\makeatother

\newcounter{error}

\newcommand{\cError}{X^{\SymbolsX{\theerror} }}

\title{Law of large numbers for greedy animals and paths in an ergodic environment}
\author{Julien \textsc{Verges} \\ julien.verges@univ-tours.fr}
\date{\today}

\begin{document}

\maketitle
\abstract{Consider a family of random masses $\mathbf{m}(v)$ indexed by vertices of the lattice $\Z^d$. In the case where the masses are i.i.d.\ and satisfy a certain moment condition, it is known that there exists a deterministic $A\ge 0$ such that the maximal mass $A_n$ of an animal containing $0$ with cardinal $n$ satisfies $A_n/n \rightarrow A$ when $n\to \infty$, almost surely. The same also goes for self-avoiding paths. We extend this result to the case where the family of masses is an ergodic marked point process, with a suitable definition for animals in this context. Special cases include the initial model with ergodic instead of i.i.d.\ masses and marked Poisson point processes. We also discuss some sufficient or necessary conditions for integrability.}
\section{Introduction}
\label{sec : intro}
\subsection{Context}
In 1993-94, Cox, Gandolfi, Griffin and Kesten \cite{Cox93, Gan94} introduced the models of \emph{greedy lattice animals} and \emph{greedy lattice paths} as such: consider an integer $d\ge 2$ and the standard lattice $\Z^d$, i.e.\ the graph with vertex set $\Z^d$, in which two vertices are neighbors if and only if their Euclidean distance is $1$. A \emph{lattice animal} is a finite connected subset of $\Z^d$. The \emph{length} of a lattice animal $\xi$ is defined as its cardinal; given a family of i.i.d.\ nonnegative variables $(\Mass v)_{v\in \Z^d}$ with distribution $\nu$, the \emph{mass} $\Mass{\xi}$ of a lattice animal $\xi$ is defined as the sum of the $\Mass v$, for $v\in \xi$. For all $n\ge 1$, we define\footnote{The exponent L stands for \emph{lattice}.} $\MassLAUF{n}$ as the maximal mass of an animal of length $n$, containing the origin. Animals realizing this maximum are called \emph{greedy}. Cox, Gandolfi, Griffin and Kesten \cite{Cox93, Gan94} proved a law of large numbers for the process $\p{\MassLAUF{n}}_{n\ge 1}$. More precisely, if for some $\eps>0$,
\begin{equation}
	\label{eqn : intro/context/CGGK_ass}
	\E{\Mass 0^d \p{\log(\Mass0)^+}^{d+\eps}} <\infty,
\end{equation}
then there exists a deterministic constant $\LimMassLA(0) \in \intervallefo0\infty$, such that almost surely and in $\rL^1$,
\begin{equation}
	\label{eqn : intro/context/LLN}
	\lim_{n\to\infty} \frac{\MassLAUF{n}}{n} = \LimMassLA(0).
\end{equation}
They also proved an analogous result the maximal mass $\MassLsaP(n)$ of a self-avoiding lattice path of length $n$, starting at the origin. In 2002, Martin \cite{Mar02} showed the same results with the weaker assumption
\begin{equation}
	\label{eqn : intro/context/greedy_condition}
	\int_0^\infty \nu\p{\intervallefo t\infty}^{1/d}\d t< \infty.
\end{equation}
and simpler arguments. Although not stated, his proof still holds for the maximal mass $\MassLPUF{n}$ of a lattice path of length $n$, starting at the origin.

This article aims to
\begin{enumerate} 
	\item Extend \eqref{eqn : intro/context/LLN} to any stationary, ergodic family of variables, provided $\E{\MassLAUF{n}}/n$ is bounded.
	\item Show a continuous analogue of \eqref{eqn : intro/context/LLN}, with a marked Poisson point process on $\R^d$ instead of $\p{\Mass v}_{v\in \Z^d}$. Greedy \emph{continuous} paths were already introduced by Gouéré and Marchand~\cite{Gou08} as a tool for the study of a continuous model of first-passage percolation. They proved that under an assumption similar to~\eqref{eqn : intro/context/greedy_condition}, the mass of a greedy continous path grows linearly. Gouéré and Théret~\cite{Gou17} also used this fact in a subsequent study of the same model.
\end{enumerate}
Corollaries~\ref{cor : intro/special/Zd} and~\ref{cor : intro/special/Poisson} answer these questions. Both are stated in Section~\ref{subsec : intro/special}. We work with a stationary, ergodic marked point process, which encompasses both situations. Moreover, our main result, Theorem~\ref{thm : intro/main/MAIN}, implies the analogue of~\eqref{eqn : intro/context/LLN} for the maximal mass of an animal of length $n$, containing $0$ and $n u$, uniformally with respect to $u$ on certain subsets of $\R^d$.

\subsection{Framework}
Let $d\ge 2$ be a integer and $\norme\cdot$ a norm on $\R^d$. For all $x\in \R^d$ and $r>0$, let $\ball{x,r}$ and $\clball{x,r}$ respectively denote the open and closed balls of center $x$ and radius $r$, for the norm $\norme\cdot$. Let $\S$ denote the unit sphere for $\norme\cdot$. For all subsets $A,B\subseteq \R^d$, we define
\begin{equation}
 	\d(A,B) \dpe \inf \set{\norme{x-y} }{x\in A,\quad y\in B}.
\end{equation} Given $p\in \intervallefo1\infty$, the norm $\norme[p]\cdot$ on $\R^d$ is defined by
\begin{equation}
	\norme[p]{x} \dpe \p{ \sum_{i=1}^d \module{x_i}^p }^{1/p},
\end{equation}
and the associated balls are denoted by $\ball[p]{\cdot,\cdot}$ and $\clball[p]{\cdot,\cdot}$. The choice $\norme\cdot = \norme[1]\cdot$ will be useful to see the lattice model as a special case of the one developed here, while the choice $\norme\cdot = \norme[2]\cdot$ will make the Poissonian model rotation-invariant. Let $\Leb$ denote the Lebesgue measure on $\R^d$. We denote by $(\base i)_{1\le i \le d}$ the canonical basis of $\R^d$.

\paragraph{Point processes.} Given a locally compact, second countable and Hausdorff topological space $\bbG$ --- we will call such a space \emph{regular} --- let $\ProSpace[\bbG]$ denote the space of measures on $\bbG$ which take integer values on compact subsets, endowed with the $\sigma$-algebra generated by the maps $\eta \mapsto \eta(A)$, for all Borel sets $A\subseteq \bbG$. We call \emph{point process} on $\bbG$ a random variable with values in $\ProSpace[\bbG]$. See~\cite{Bac20} and~\cite{Dal06, Dal07} for the general theory of point processes. It is well known (see e.g.\ Lemma 1.6.8 in \cite{Bac20}) that a point process $\Phi$ on $\bbG$ may be written as the sum
\begin{equation}
	\label{eqn : intro/main/decomposition_pro}
	\Phi = \sum_{n=1}^{N} \Dirac{z_n},
\end{equation}
where $N$ and the $x_n$ for $1\le n \le N$ are random variables with values in $\N\cup\acc\infty$ and $\bbG$ respectively. Let $\pro$ be a simple marked point process on $\R^d \times \intervalleoo0\infty$, i.e.\ a point process on $\R^d\times \intervalleoo0\infty$ such that almost surely, for all $x\in \R^d$, $\pro\p{\acc{x}\times \intervalleoo{0}{\infty}}\le 1$. Equation~\eqref{eqn : intro/main/decomposition_pro} then takes the form
\begin{equation}
	\label{eqn : intro/main/decomposition_pro2}
	\pro = \sum_{n=1}^{N} \Dirac{x_n, \Mass{x_n}},
\end{equation}
where, for all $1\le n \le N$, $x_n \in \R^d$ and $\Mass{x_n}\in \intervalleoo0\infty$. Let 
\begin{equation}
\projpro \dpe \set{ x\in \R^d }{ \pro\p{\acc{x}\times \intervalleoo{0}{\infty}}>0 } = \set{x_n}{1\le n \le N}.
\end{equation}
For all $z\in \R^d$, let $T_z\pro$ denote the image of $\pro$ by the map
\begin{align*}
	\Trans_z : \R^d\times \intervalleoo0\infty &\longrightarrow \R^d\times \intervalleoo0\infty \\
	(x,t) &\longmapsto (x+z,t).
\end{align*}
We assume $\pro$ to be stationary, i.e.\ for all $z\in \R^d$, $\pro$ and $T_z \pro$ have the same distribution.
\begin{Definition}
	\label{def : intro/framework/mass_path_animal}
	For every subset $A\subseteq \R^d$, the \emph{mass} of $A$ (with respect to $\pro$) is defined as
	\begin{equation}
		\label{eqn : intro/framework/mass_path_animal}
		\Mass{A} \dpe \int_{A \times \intervalleoo0\infty} t \pro(\d x, \d t) = \sum_{n=1}^N \Mass{x_n} \ind{x_n \in A}.
	\end{equation}
\end{Definition}
For all $x\in \R^d$, we adopt the notation $\Mass x = \Mass{\acc x}$, which is consistent with~\eqref{eqn : intro/main/decomposition_pro2}. 

\paragraph{Moment measures of point processes.}
We call \emph{mean measure} of a point process $\Phi$ on a regular space $\bbG$ the measure defined on $\bbG$ by
\begin{equation}
	\label{eqn : intro/suff_condition/def_mm}
	\MM[\Phi](E) \dpe \E{ \Phi(E) },
\end{equation}
for all Borel subset $E \subseteq \bbG$. If $\MM[\pro]$ is locally finite, one shows by a straight adaptation of Lemma~6.1.17 (iii) in~\cite{Bac20} that there exists\footnote{Indeed consider the measure $\nu$ on $\intervalleoo0\infty$ by $\nu(B) \dpe \E{\intervallefo01^d \times B}$. By the mentioned lemma, for all compact subsets $A\subseteq \R^d$ and $B\subseteq \intervalleoo0\infty$, $\MM[\pro](A\times B) = \Leb(A)\nu(B)$. } a measure $\nu$ on $\intervalleoo0\infty$ such that
\begin{equation}
	\label{eqn : intro/framework/mm_as_a_product}
	\MM[\pro] = \Leb \otimes \nu.
\end{equation}
This will be the case under the framework of our main theorem (see Proposition~\ref{prop : intro/nec_condition}). In particular, every Lebesgue-negligible subset of $\R^d$ has almost surely no mass. We will regularly use this fact. 

For all $k\ge1$, we call $k$-\emph{th factorial power} of $\Phi$ the point process on $\bbG^k$ defined by
\begin{equation}
	\label{eqn : intro/suff_condition/def_fmm}
	\Phi^{(k)} \dpe \sum_{ \substack{1\le n_1, \dots, n_k \le N \\ \text{pairwise distinct}} } \Dirac{\p{ z_{n_1},\dots, z_{n_k} }  },
\end{equation}
with the notations of~\eqref{eqn : intro/main/decomposition_pro}. Note that $\Phi^{(1)}= \Phi$. We call $k$-th \emph{factorial moment measure} of $\Phi$ the measure $\MM[\Phi^{(k)}]$.

\paragraph{Continuous paths and animals.}

\begin{Definition}
	\label{def : intro/framework/path_animal}
	Following Gouéré and Marchand (2008) \cite{Gou08}, we call (continuous) \emph{path} a finite sequence of points of $\R^d$. For a given norm $\norme\cdot$, the \emph{length} of a path $\gamma = (x_0,\dots, x_r)$ is defined as
	\begin{equation}
	\label{eqn : intro/framework/length_path}
		\norme\gamma \dpe \sum_{i=0}^{r-1}\norme{x_i - x_{i+1}}.
	\end{equation}
	We call (continuous) \emph{animal} a finite connected graph whose vertices are points of $\R^d$. The \emph{length} of a animal $\xi = (V,E)$ with vertex set $V$ and edge set $E$ is defined as
	\begin{equation}
	\label{eqn : intro/framework/length_animal}
		\norme\xi \dpe \sum_{\acc{x,y}\in E} \norme{x-y}.
	\end{equation}
\end{Definition}
When there is no ambiguity, we will identify a path $\gamma = (x_0, \dots, x_r)$ with the animal with vertex set $\acc{x_i}_{0\le i \le r}$ and edge set $\acc{(x_{i-1},x_{i})}_{1\le i \le r}$. We will also identify a path or an animal with its vertex set, e.g.\ for any animal $\xi =(V,E)$, $\Mass{\xi}= \Mass{V}$ and for any path $\gamma=(x_0,\dots, x_r)$, $\Mass{\gamma}=\Mass{\acc{x_0,\dots, x_r}}$. The following families of paths will be of interest. For all $x,y\in \R^d$ and $\ell\ge0$, we define:
\begin{itemize}
	\item  $\SetPUF{\ell}$ as the set of paths of length at most $\ell$, starting at $0$.
	\item  $\SetPDF{x}{y}{\ell}$ as the set of paths of length at most $\ell$, starting at $x$ and ending at $y$.
\end{itemize}
Likewise, we define:
\begin{itemize}
	\item  $\SetAUF{\ell}$ as the set of animals of length at most $\ell$, containing $0$.
	\item  $\SetADF{x}{y}{\ell}$ as the set of animals of length at most $\ell$, containing $x$ and $y$.
\end{itemize}
For all $x,y\in \Z^d$, $\ell \in \N$, $\SetLPUF{\ell}$,\ldots, $\SetLADF{x}{y}{\ell}$ are defined the same way as their counterparts without the exponent $\mathrm L$, with \emph{lattice} paths and animals.

\paragraph{The processes.}
For any set of paths or animals denoted by a calligraphic font letter, we use the same letter in roman typestyle to denote the supremum of the mass of animals or paths in this set. For example, for all $\ell\ge 0$,
\begin{equation}
	\label{eqn : intro/framework/def_MassPUF}
	\MassPUF{\ell} \dpe \sup_{\gamma \in \SetPUF{\ell} } \Mass{\gamma}.
\end{equation}
We also use this convention for a generic $\AUXGeneric \in \acc{\AUXPath, \AUXAnimal}$ : for all $x,y\in \R^d$ and $\ell\ge 0$,
\begin{equation}
	\label{eqn : intro/framework/def_MassGUF_MassGDF}
	\MassGUF{\ell} \dpe \sup_{\gamma \in \SetGUF{\ell} } \Mass{\gamma}%
	\text{ and }
	\MassGDF{x}{y}{\ell} \dpe \sup_{\gamma \in \SetGDF{x}{y}{\ell}} \Mass{\gamma}.
\end{equation}
Another natural analogue of $\MassLAUF\cdot$ in a continuous context consists in restricting the supremum to animals which are included in $\projpro$. More precisely, for all $x,y\in\R^d$ and $\ell\ge0$, we define:
\begin{itemize}
	\item $\SetAUFalt{\ell}$ as the set of animals $\xi$ such that $\norme\xi + \d(0,\xi) \le \ell$, or $\xi$ is empty,
	\item $\SetADFalt{x}{y}{\ell}$ as the set of animals $\xi$ such that $\norme\xi + \d(x,\xi) + \d(y,\xi)\le \ell$, or $\xi$ is empty,
\end{itemize}
and the corresponding variables 
\begin{equation}
	\label{eqn : intro/framework/animaux_contraints}
	\MassAUFpen\ell{\infty} \dpe \sup_{\substack{ \xi \in \SetAUFalt{\ell} \\ \xi \subseteq \projpro}  } \Mass\xi%
	\text{ and }
 	\MassADFpen xy\ell{\infty} \dpe \sup_{\substack{ \xi \in \SetADFalt xy{\ell} \\ \xi \subseteq \projpro}  } \Mass\xi.
\end{equation}
It is pointless to introduce similar processes for paths, since by triangle inequality, skipping vertices outside $\projpro$ along a path produces a path with the same mass and smaller length. The notation $\MassAUFpen\cdot\infty$ is linked to the following third analogue of $\MassLAUF\cdot$, which is an interpolation of the preceding two. For all $q\in\intervalleff0\infty$, $x,y\in\R^d$ and $\ell \ge 0$, we define
\begin{align}
	\label{eqn : intro/framework/animaux_pen1}
	\MassAUFpen{\ell}{q} &\dpe \sup_{\xi \in \SetAUFalt{\ell}} \cro{ \Mass{\xi} - q\#\p{\xi \cap \projpro^\mathrm{c}} }\\%
	\text{and }
	\label{eqn : intro/framework/animaux_pen2}
	\MassADFpen{x}{y}{\ell}{q} &\dpe \sup_{\xi \in \SetADFalt{x}{y}{\ell}} \cro{ \Mass{\xi} - q\#\p{\xi \cap \projpro^\mathrm{c}} }, 
\end{align}
i.e.\ the analogues of $\MassAUF{\ell}$ and $\MassADF{x}{y}{\ell}$, with a penalization $-q$ for every vertex of $\xi$ not belonging to the point process. By adding one vertex and one edge, one shows that any animal in $\SetAUFalt{\ell}$ is included in an animal in $\SetAUF{\ell}$, thus $\MassAUFpen{\ell}{0} \le \MassAUF{\ell}$. The inclusion $\SetAUF{\ell} \subseteq \SetAUFalt{\ell}$ gives the converse inequality, hence
\begin{align*}
	\MassAUFpen{\ell}{0} &= \MassAUF{\ell}.
	\intertext{Likewise,}
	\MassADFpen{x}{y}{\ell}{0} &= \MassADF{x}{y}{\ell}.
\end{align*} 
Besides, \eqref{eqn : intro/framework/animaux_contraints} is compatible with \eqref{eqn : intro/framework/animaux_pen1} and \eqref{eqn : intro/framework/animaux_pen2}.

Note that for all $\ell>0$ and $q\in\intervalleff0\infty$,
\begin{equation}
	\label{eqn : intro/framework/easy_inequality}
	\MassPUF\ell \le \MassAUFpen\ell q \le \MassAUF\ell \le \MassPUF{2\ell},
\end{equation}
where we have used in the last inequality the fact that any animal may be covered by the path obtained by a depth-first search.
\subsection{Main results}
We work under the following assumptions.
\begin{Assumption}
\label{ass : intro/main/Ergodic_Stationary}
The process $\pro$ is ergodic with respect to the translations by elements of $\R^d$, i.e.\ for all measurable subsets $E\subseteq \ProSpace$ such that
	\begin{equation}
		\label{eqn : intro/main/invariant_event}
		\forall z\in\R^d,\quad \Pb{\acc{\pro \in E} \symdif \acc{T_z \pro \in E} } =0,
	\end{equation}
	$\Pb{\pro \in E} \in \acc{0,1}$. 
\end{Assumption}
(See Definition~8.4.1 in \cite{Bac20}.) Subsets satisfying~\eqref{eqn : intro/main/invariant_event} are called \emph{invariant}.
\begin{Assumption}
	\label{ass : intro/main/Moment}
	\[ \GeneralUB\dpe\sup_{\ell \ge 1} \frac{\E{ \MassAUF{\ell} } }{\ell} < \infty. \]
\end{Assumption}
Let $\cX$ denote the subset of $\clball{0,1}^2\times \intervalleof01$ consisting of triplets $(x,y,\ell)$ such that $x$ and $y$ are colinear, and $\norme{x-y}< \ell$.
\begin{Theorem}
	\label{thm : intro/main/MAIN}
	Let $\AUXGeneric \in \acc{\AUXPath, \AUXAnimal}$. Assume that $\pro$ satisfies Assumptions~\ref{ass : intro/main/Ergodic_Stationary} and~\ref{ass : intro/main/Moment}. Then there exists a deterministic, concave, symmetric with respect to $u\mapsto -u$ function $\LimMassG : \ball{0,1}\rightarrow \intervallefo0\infty$, such that for all compact subsets $K\subseteq \cX$,
	\begin{equation}
		\label{eqn : intro/main/MAIN/cvg}
		\sup\set{ \module{ \frac{\MassGDF{Lx}{Ly}{L\ell} }{L} - \ell\LimMassG\p{\frac{x-y}{\ell}}  } }%
		{(x,y,\ell) \in K}%
		\xrightarrow[L\to \infty]{\text{a.s. and }\rL^1} 0.
	\end{equation}
	Moreover,
	\begin{equation}
		\label{eqn : intro/main/MAIN/cvg_undirected}
		\frac{\MassGUF L}{L} \xrightarrow[L\to \infty]{\text{a.s. and }\rL^1} \LimMassG(0).
	\end{equation}
\end{Theorem}
The analogous result for penalized maximal masses of animals also holds.
\begin{Theorem}
	\label{thm : intro/main/MAIN_PENALIZED}
	Let $q\in \intervalleff0\infty$. Assume that $\pro$ satisfies Assumptions~\ref{ass : intro/main/Ergodic_Stationary} and~\ref{ass : intro/main/Moment}. Then there exists a deterministic, concave, symmetric with respect to $u\mapsto -u$ function $\LimMassA^{(q)} : \ball{0,1}\rightarrow \intervallefo0\infty$, such that for all compact subsets $K\subseteq \cX$,
	\begin{equation}
		\label{eqn : intro/main/MAIN_PENALIZED/cvg}
		\sup\set{ \module{ \frac{\MassADFpen{Lx}{Ly}{L\ell}q }{L} - \ell\LimMassA^{(q)}\p{\frac{x-y}{\ell}}  } }%
		{(x,y,\ell) \in K}%
		\xrightarrow[L\to \infty]{\text{a.s. and }\rL^1} 0.
	\end{equation}
	Moreover,
	\begin{equation}
		\label{eqn : intro/main/MAIN_PENALIZED/cvg_undirected}
		\frac{\MassAUFpen Lq}{L} \xrightarrow[L\to \infty]{\text{a.s. and }\rL^1} \LimMassA^{(q)}(0).
	\end{equation}
\end{Theorem}

\subsection{Special cases}
\label{subsec : intro/special}
Theorems~\ref{thm : intro/main/MAIN} and~\ref{thm : intro/main/MAIN_PENALIZED} applies for the original discrete model (up to a minor adjustment to ensure stationarity), provided the masses are ergodic and Assumption~\ref{ass : intro/main/Moment} holds, and for marked Poisson point processes, provided the distribution $\nu$ of the marks satisfies~\eqref{eqn : intro/context/greedy_condition}. The latter case may be extended to a certain class of determinantal point processes.

\paragraph{The discrete model.}
\begin{Corollary}
	\label{cor : intro/special/Zd}
	Let $\p{\Mass v}_{v\in\Z^d}$ be a stationary and ergodic family of nonnegative random variables, i.e.\ for all $i\in \intint1d$, $\p{\Mass{v+ \base i}}_{v\in\Z^d}$ has the same the distribution as $\p{\Mass v}_{v\in\Z^d}$, and for every event $\cE$ satisfying
	\begin{equation}
		\forall i\in\intint1d, \quad%
		\Pb{ \acc{\p{\Mass v}_{v\in\Z^d} \in \cE} \symdif %
			 \acc{\p{\Mass{v+\base i} }_{v\in\Z^d} \in \cE} }=0,
	\end{equation}
	$\Pb{\cE}\in\acc{0,1}$. Assume that
	\begin{equation}
		\label{eqn : intro/special/Zd/moment}
		\sup_{n\ge 1} \frac{ \E{\MassLAUF{n}} }{n} < \infty.
	\end{equation}
	Let $\AUXGeneric \in \acc{\AUXAnimal, \AUXPath}$. Fix $\norme\cdot = \norme[1]\cdot$. Then there exists a deterministic, concave, symmetric w.r.t. $u\mapsto -u$ function $\LimMassLG : \ball[1]{0,1}\rightarrow \intervallefo0\infty$ such that for all compact subsets $K\subseteq \cX$, 
	\begin{equation}
		\label{eqn : intro/special/Zd/cvg}
		\sup\set{\module{ \frac{ \MassLGDF{ \floor{Lx} }{ \floor{Ly} }{L\ell} }{L} - \ell\LimMassLG\p{\frac{x-y}{\ell}} }}%
		{(x,y,\ell) \in K}
		\xrightarrow[L\to\infty]{\text{a.s. and }\rL^1} 0,
	\end{equation}
	with the notation $\floor x \dpe \p{\floor{x_1},\dots, \floor{x_d}}$, for all $x= (x_1,\dots,x_d)\in \R^d$. Moreover,
	\begin{equation}
		\label{eqn : intro/special/Zd/cvg_undirected}
		\frac{\MassLGUF L}{L} \xrightarrow[L\to\infty]{\text{a.s. and }\rL^1}\LimMassG(0).
	\end{equation}
\end{Corollary}

\paragraph{Determinantal point processes.} Before stating the result, we need to recall a couple of definitions.

\begin{Definition}
	Let $\mu$ be a locally finite measure on a regular space $\bbG$ and $K : \bbG^2 \rightarrow \bbC$ be a measurable function. We say that $\Phi$ is a \emph{Determinantal point process} (DPP) with \emph{kernel} $K$ and \emph{background measure} $\mu$ if for all $k\ge 1$,
	\begin{equation}
		\FMM[\Phi]{k}\p{\d z_1,\dots,\d z_k} = \det \p{ K\p{z_i, z_j} }_{1\le i,j\le k}\mu(\d z_1) \dots \mu(\d z_k).
	\end{equation}
	In this article, we further say that $\Phi$ is a \emph{good DPP} if for $\mu^{\otimes k}$-almost all $(z_1,\dots,z_k)\in \bbG^k$, the matrix $ \p{ K( z_i, z_j )}_{1\le i,j\le k}$ is Hermitian nonnegative-definite (i.e.\ it is self-adjoint and its eigenvalues are nonnegative).
\end{Definition}

Chapter~5 of \cite{Bac20} provides a general study of DPPs.
\begin{Corollary}
	\label{cor : intro/special/DPP}
	Let $\pro$ be a stationary simple marked point process on $\R^d \times \intervalleoo0\infty$ with mean measure $\MM[\pro]=\Leb \otimes \nu$. Assume that 
	\begin{enumerate}[(i)]
		\item \label{item : intro/special/DPP/good_DPP} The point process $\pro$ is a good DPP with \emph{kernel} $K$ and \emph{background measure} $\Leb\otimes \nu$.
		\item \label{item : intro/special/DPP/vanish_infty} For all $s,t \in \intervalleoo0\infty$,
		\begin{equation}
			\label{eqn : intro/special/DPP/vanish_infty}
			K\p{(0,s), (z,t)} \xrightarrow[\norme z \to \infty]{} 0.
		\end{equation}
		\item \label{item : intro/special/DPP/greedy_condition} The distribution $\nu$ satisfies~\eqref{eqn : intro/context/greedy_condition}.
	\end{enumerate}
	Then Assumptions~\ref{ass : intro/main/Ergodic_Stationary} and~\ref{ass : intro/main/Moment}, and thus the conclusions of Theorems~\ref{thm : intro/main/MAIN} and~\ref{thm : intro/main/MAIN_PENALIZED} hold. Moreover,
	\begin{equation}
		\GeneralUB \le \Cr{cst : Dirac_bound} \int_0^\infty \nu\p{\intervallefo t\infty}^{1/d} \d t,
	\end{equation}
	where $\Cr{cst : Dirac_bound}$ is a constant introduced in Proposition~\ref{prop : intro/suff_condition/fmm}.
\end{Corollary}
Poisson point processes fall under~Corollary~\ref{cor : intro/special/DPP} (see e.g.\ Example~5.1.6 in \cite{Bac20}).
\begin{Corollary}
	\label{cor : intro/special/Poisson}
	Assume that
	\begin{enumerate}[(i)]
		\item The point process $\pro$ is Poisson with mean measure $\Leb\otimes \nu$.
		\item The distribution $\nu$ satisfies~\eqref{eqn : intro/context/greedy_condition}.
	\end{enumerate}
	Then Assumptions~\ref{ass : intro/main/Ergodic_Stationary} and~\ref{ass : intro/main/Moment}, and thus the conclusions of Theorems~\ref{thm : intro/main/MAIN} and~\ref{thm : intro/main/MAIN_PENALIZED} hold. Moreover,
	\begin{equation}
	 	\GeneralUB \le \Cr{cst : Dirac_bound} \int_0^\infty \nu\p{\intervallefo t\infty}^{1/d} \d t.
	 \end{equation} 
\end{Corollary}

\subsection{Around Assumption~\ref{ass : intro/main/Moment}}

We say that a point process $\Phi$ on the regular space $\bbG$ satisfies the \emph{moment property} with the constant $C>0$ if for all $k\ge 1$ and Borel subsets $B_1,\dots,B_k \subseteq \bbG$,
\begin{equation}
	\label{eqn : intro/suff_condition/fmm}
	\FMM[\Phi]{k}\p{\prod_{i=1}^k B_i} \le C^k \prod_{i=1}^k \MM[\Phi](B_i).
\end{equation} 

We have the following sufficient condition for Assumption~\ref{ass : intro/main/Moment}.
\begin{Proposition}
	\label{prop : intro/suff_condition/fmm}
	Let $\pro$ be a stationary simple marked point process on $\R^d \times \intervalleoo0\infty$, with mean measure $\Leb\otimes \nu$. Assume that
	\begin{enumerate}[(i)]
		\item The point process $\pro$ satisfies the moment property with the constant $C>0$.
		\item The distribution $\nu$ satisfies~\eqref{eqn : intro/context/greedy_condition}.
	\end{enumerate}
	Then
	\begin{equation}
		\label{eqn : intro/suff_condition/fmm/bound}
		\E{\sup_{\ell \ge 1} \frac{\MassAUF\ell}{\ell}} \le \Cl{cst : Dirac_bound} C^{1/d}\int_0^\infty \nu\p{\intervallefo t\infty }^{1/d}\d t,
	\end{equation}
	where $\Cr{cst : Dirac_bound}>0$ only depends on $d$ and $\norme\cdot$. In particular, Assumption~\ref{ass : intro/main/Moment} holds.
\end{Proposition}

Along the proof of Proposition~\ref{prop : intro/suff_condition/fmm}, we actually show the stronger version~\eqref{eqn : intro/nec_condition/layers} of Assumption~\ref{ass : intro/main/Moment}, which implies~\eqref{eqn : intro/context/greedy_condition}. Recall the decomposition~\eqref{eqn : intro/main/decomposition_pro2}. For all $t>0$, we consider the point process on $\R^d$
\begin{equation}
	\label{eqn : intro/main/truncated}
	\pro^{(t)} \dpe \sum_{n=1}^N \ind{\Mass{x_n}\ge t } \Dirac{x_n}.
\end{equation}
The measure $\pro^{(t)}\otimes \Dirac1$ is a stationary marked point process on $\R^d \times \intervalleoo0\infty$. Informally, it corresponds to giving mass $1$ to every point in $\pro^{(t)}$. The notation $\MassAUF\ell \cro{\pro^{(t)}\otimes \Dirac1 }$ in~\eqref{eqn : intro/nec_condition/layers} denotes the analogue of $\MassAUF\ell$, constructed from the process $\pro^{(t)}\otimes \Dirac1$ rather than $\pro$.
\begin{Proposition}
	\label{prop : intro/nec_condition}
	Let $\pro$ be a stationary simple point process on $\R^d \times \intervalleoo0\infty$ that satisfies Assumptions~\ref{ass : intro/main/Ergodic_Stationary} and~\ref{ass : intro/main/Moment}. Then $\MM[\pro]$ is locally finite, thus admits a decomposition of the form $\Leb \otimes \nu$. Moreover, if
	\begin{equation}
		\label{eqn : intro/nec_condition/layers}
		\int_0^\infty \sup_{\ell \ge 1} \frac{\E{\MassAUF\ell \cro{\pro^{(t)}\otimes \Dirac1 } }  }\ell \d t < \infty,
	\end{equation}
	then $\nu$ satisfies~\eqref{eqn : intro/context/greedy_condition}.
\end{Proposition}

Besides, if the masses of the atoms of $\pro$ are i.i.d, then Assumption~\ref{ass : intro/main/Moment} implies that $\nu$ has a $d$-th moment. More precisely, we introduce the notion of i.i.d.\  markings of point processes (see Definition~2.2.18 in \cite{Bac20}).
\begin{Definition}
	\label{def : intro/iid_marking}
	Let $\bbG$ and $\bbK$ be regular spaces , $\Phi$ be a point process on $\bbG$ and $\mu$ a probability distribution on $\bbK$. Consider a sequence of i.i.d.\  random variables $(t_n)_{n \ge 1}$. With the notations of~\eqref{eqn : intro/main/decomposition_pro},
	\begin{equation}
		\tilde\Phi \dpe \sum_{n=1}^N \Dirac{(z_n, t_n)}
	\end{equation}
	is a point process on $\bbG \times \bbK$, called an \emph{i.i.d.\  marking of $\Phi$, with mark distribution $\mu$}.
\end{Definition}
\begin{Proposition}
	\label{prop : intro/moment_d}
	Let $\Phi$ be an ergodic and stationary simple point process on $\R^d$ such that $\MM[\Phi] = \lambda \Leb$, with $0 < \lambda < \infty$. Let $\nu$ be a probability measure on $\intervalleoo0\infty$ and $\pro$ assume that is an i.i.d.\  marking of $\Phi$ with mark distribution $\nu$. If $\pro$ satisfies Assumption~\ref{ass : intro/main/Moment}, then
	\begin{equation}
		\label{eqn : intro/moment_d}
		\int_{\intervalleoo0\infty} t^d \nu(\d t) < \infty.
	\end{equation} 
\end{Proposition}

\subsection{Outline of the paper}

Section~\ref{sec : LLN} is devoted to the proof of Theorem~\ref{thm : intro/main/MAIN}. Our main tool is the following extension of Kingman's theorem, adapted from Akcoglu and Krengel (1981) \cite[Theorem 2.4]{AK81}. Since our version does not involve new ideas, we place it in the Appendix.
\begin{Theorem}
	\label{thm : intro/outline/AK81}
	Let $\p{ X(s,t) }_{s<t}$ be a random process indexed by ordered pairs of real numbers. Assume that $X(\cdot,\cdot)$ is 
	\begin{enumerate}[(i)]
		\item nonnegative,
		\item stationary, i.e.\ for all $u\in \R$, $\p{X(s+u,t+u)}_{s<t}$ has the same distribution as $\p{X(s,t)}_{s<t}$,
		\item superadditive, i.e.\ for all $s<t<u$, $X(s,u) \ge X(s,t) + X(t,u)$, 
		\item and satisfies\begin{equation}
	 	\sup_{t\ge 1} \frac1t\E{X(0,t)}<\infty.
	\end{equation} 
	\end{enumerate}
	Then for all $a,b \in \R$ such that $a\le b$, the limit
	\begin{equation}
		\label{eqn : intro/outline/AK81/convergence}
	 	Y(a,b)\dpe \lim_{\ell\to\infty} \frac1\ell X(a\ell ,b\ell)
	\end{equation}
	exists \textrm{a.s.}\ and in $\rL^1$. Moreover, for all sequences of rational numbers $(a_n)$ and $(b_n)$ such that\footnote{Actually for all $(a,b)$, $Y$ is almost surely continuous at $(a,b)$, but we only need the weaker version~\eqref{eqn : intro/outline/AK81/convergence_ab}, which appears along the proof of~Theorem~\ref{thm : intro/outline/AK81}.}
	\begin{equation*}
		\inclim{n\to \infty} a_n = a\text{ and } \declim{n\to \infty} b_n =b,
	\end{equation*}
	almost surely,
	\begin{equation}
		\label{eqn : intro/outline/AK81/convergence_ab}
		Y(a,b) = \lim_{n\to \infty} Y(a_n, b_n).
	\end{equation}
\end{Theorem}
Note that $Y(a,b)$ may not be a deterministic constant. However, we will use~\eqref{thm : intro/outline/AK81} in a context where it will be the case, by ergodicity.

To prove Theorem~\ref{thm : intro/main/MAIN}, the general ideas are somewhat similar to those of Gandolfi-Kesten~\cite{Gan94} and Martin~\cite{Mar02}. They defined auxiliary processes as the maximal mass of an animal with prescribed width and leftmost point (in the direction $\base 1$). The expectation of theses processes are superadditive, thus Fekete's lemma applies. The method of bounded differences (in \cite{Gan94}) or a concentration inequality due to Talagrand (1995) \cite[Theorem~8.1.1]{Tal95} (in \cite{Mar02}) then gives a sharp bound for the probability of $\frac{ \MassLAUF n} n$ taking values far from $\LimMassLA(0)$, yielding the LLN for $\MassLAUF{n}$ by Borel-Cantelli's lemma. 

Since we do not assume independence of the masses, we do not have access to bounded differences nor concentration inequalities. To circumvent this issue, we make the following changes in the strategy. First, the auxiliary processes we consider are similar to the ones defined by~\eqref{eqn : intro/framework/def_MassGUF_MassGDF}, by choosing special values of $(x,y,\ell)$ and restricting the supremum to the subset of $\SetGDF{x}{y}{\ell}$ consisting of paths or animals included in a certain diamond ($x$ and $y$ are extremal points of which). They are superadditive in the strong sense (not simply in expectation). Theorem~\ref{thm : intro/outline/AK81} gives the LLN for theses processes. Second, we use elementary concatenation arguments to compare $\MassGDF{x}{y}{\ell}$ to the auxiliary processes.

Since the proof of Theorem~\ref{thm : intro/main/MAIN_PENALIZED} only requires some minor adaptations, we leave it to the reader.

Section~\ref{sec : suff_condition} contains the proofs of Propositions~\ref{prop : intro/suff_condition/fmm},~\ref{prop : intro/nec_condition} and~\ref{prop : intro/moment_d}. The first one relies on a straightforward adaptation of Theorem~1.2 in Gouéré and Marchand (2008) \cite{Gou08}, which gives a bound for the mass of a path in the Poissonian case with unit masses. The second one is based on a classic upper bound for the Travelling salesman problem. The last one uses Borel-Cantelli's lemma.

Section~\ref{sec : app} is devoted to proving Corollaries~\ref{cor : intro/special/Zd} and~\ref{cor : intro/special/DPP}. For the second one, we use a void-probability-based criterion to show Assumption~\ref{ass : intro/main/Ergodic_Stationary} and Proposition~\ref{prop : intro/suff_condition/fmm} to show Assumption~\ref{ass : intro/main/Moment}.


\subsection{Related works and open questions}
\label{subsec : intro/open_questions}

\paragraph{Integrability.}
Consider the case where $\pro$ is a Poisson point process on $\R^d \times \intervalleoo0\infty$, with intensity $\Leb\otimes \nu$. Corollary~\ref{cor : intro/special/Poisson} and Proposition~\ref{prop : intro/moment_d} leave a gap in our understanding of the asymptotic behaviour of $\MassAUF\ell$ and similar processes, as in the case where $\nu$ has a finite $d$-th moment but does not satisfy~\eqref{eqn : intro/context/greedy_condition}, we do not know if Assumption~\ref{ass : intro/main/Moment} holds. In particular, we do not know if Assumption~\ref{ass : intro/main/Moment} and~\eqref{eqn : intro/nec_condition/layers} are equivalent. Note that the sharpest known necessary and sufficient conditions for the original discrete model, as stated by Martin \cite{Mar02}, are analogous to the ones for Poisson point processes.

In general, no moment condition on $\nu$ alone can guarantee Assumption~\ref{ass : intro/main/Moment}, even in the discrete model. Indeed, let $\nu$ be any probability measure on $\intervalleoo0\infty$ with unbounded support. Let $(X_v)_{v\in \Z^2}$ be a family of random variables with distribution $\nu$, such that
\begin{enumerate}
	\item For all $(v_1, v_2) \in \Z^2$, $X_{v_1, v_2} = X_{v_1,0}$.
	\item The variables $\p{X_{v_1, 0}}_{v_1 \in \Z }$ are independent. 
\end{enumerate}
Note that $(X_v)_{v\in \Z^2}$ is stationary and ergodic. For all $s>0$, almost surely, there exists $v_1 \in \N$ such that $X_{v_1,0} \ge s$. Thus by considering the path \[\gamma \dpe \p{ (0,0), (1,0), \dots, (v_1,0), (v_1, 1), \dots, (v_1, n) },\]
for $n \ge 1$, one shows that
\begin{align*}
	\limsup_{n\to\infty} \frac{ \MassLAUF n }{n} &\ge s,
\intertext{hence}
	\limsup_{n\to\infty} \frac{ \MassLAUF n }{n} &= \infty.
\end{align*}

\paragraph{Extension to possibly negative masses.} Dembo, Gandolfi and Kesten proved in 2001 \cite{Dem01} that~\eqref{eqn : intro/context/LLN} still holds when the masses $\p{\Mass v}_{v\in \Z^d}$ are not assumed to be necessarily negative, provided their positive parts satisfy~\eqref{eqn : intro/context/CGGK_ass}. They also study the maximal mass $G_n$ of an animal of any size, included in $\intint0n^d$. The order of $G_n$ is at most $n$ if $\LimMassLA(0)<0$ and $n^d$ if $\LimMassLA(0)>0$. In 2006, Hammond \cite{Ham06} pushed the study further by providing an estimate for $G_n$ in the critical case. He also proved that in the supercritical case the limit $\lim_{n \to \infty} \frac{G_n}{n^d}$ exists almost surely, and the animal realizing $G_n$ is dense, in the sense that it intersects all open sites of the largest cluster for a box-level percolation process on $\intint0n^d$ with arbitrarily high parameter.

In a recent article, Chang and Zheng \cite{Chang24} proved the law of large numbers for greedy lattice paths for possibly negative masses, provided their positive parts satisfy~\eqref{eqn : intro/context/CGGK_ass} and their negative parts have a finite fourth moment.

\paragraph{Large deviations.} In the article mentioned above \cite{Dem01}, Dembo, Gandolfi and Kesten proved a large deviation estimate for abnormally large values of $\MassLAUF n$, under an exponential moment condition. To our knowledge, the existence of the corresponding rate function remains to be shown. Besides, large deviations for abnormally small values of $\MassLAUF n$ seem not to have been studied.

\paragraph{About the limiting constant.} Lee showed in 1993 \cite{Lee93} that except in the case where the vertices have maximal mass with a probability greater than or equal to the site-percolation critical parameter, the limiting constant $\LimMassLA(0)$ for the greedy lattice animals is strictly greater than its analogue for greedy lattice paths. The same author showed in 1997 \cite{Lee97_Continuity} that under a domination assumption, it is continuous with respect to the distribution of $\Mass0$, and provide \cite{Lee97_PowerLaws} estimates for their behaviour near criticality for masses taking values in $\acc{0,1}$.

\subsection{Notations}
\label{subec : intro/notations}
\paragraph{$\pro$-mesurable random variables.}
In contexts where more than one point process is considered, we will indicate the dependence on the point process by square brackets, e.g.\ for any point process $\tilde \pro$ on $\R^d \times \intervalleoo0\infty$ and any subset $A\subseteq \R^d$,
\begin{equation}
	\label{eqn : intro/framework/mass_subset_other_process}
	\Mass{A}\cro{\tilde\pro} \dpe \int_{A \times \intervalleoo0\infty} t \tilde\pro(\d x, \d t).
\end{equation}
\paragraph{Animals and paths.} Given two paths $\gamma_1 = (x_0, \dots x_r)$ and $\gamma_2 = (y_0, \dots, y_s)$ such that $x_r=y_0$, we define the \emph{concatenation} of $\gamma_1$ and $\gamma_2$ as the path
\begin{equation}
	\gamma_1 \concat \gamma_2 \dpe (x_0, \dots, x_r, y_1, \dots, y_s).
\end{equation}
Similarly, given two animals $\xi_1 = (V_1, E_1)$ and $\xi_2 = (V_2, E_2)$ such that $V_1 \cap V_2 \neq \emptyset$, we define the \emph{concatenation} of $\xi_1$ and $\xi_2$ as the animal
\begin{equation}
	\xi_1 \concat \xi_2 \dpe (V_1\cup V_2, E_1\cup E_2).
\end{equation}

\paragraph{Vectors and subsets of $\R^d$.}
Fix $\Cl{cst : equiv_norms}>0$ (depending on $\norme\cdot$) such that
\begin{equation}
	\label{eqn : intro/notations/equiv_norms}
	\frac1{\Cr{cst : equiv_norms}}\norme[2]{\cdot} 	\le \norme{\cdot} \le \Cr{cst : equiv_norms}\norme[2]{\cdot}.
\end{equation}
We denote by $\S$ the unit sphere for $\norme\cdot$. For all $x,y\in \R^d$, we denote by $\intervalleff xy$ the segment between $x$ and $y$. For all $x\in \R^d$, $u\in \R^d\setminus \acc 0$ and $0<\delta<1$, we define the \emph{cone}
\begin{equation}
	\label{eqn : intro/notations/cone}
	\Cone{\delta}{x}{u} \dpe \set{z\in \R^d }{\ps{z-x}{\frac{u}{\norme[2]{u} }} \ge (1-\delta)\norme[2]{z-x} }.
\end{equation}
For all distinct points $x,y\in \R^d$ and $0<\delta < 1$, we define the \emph{diamond}
\begin{align}
	\label{eqn : intro/notations/diamond}
	\begin{split}
	\Diamant{\delta}{x}{y} &\dpe \Cone{\delta}{x}{y-x} \cap \Cone{\delta}{y}{x-y} 
	\end{split}
	\intertext{and the \emph{antidiamond} }
	\label{eqn : intro/notations/antidiamond}
	\begin{split}
	\AntiDiamant{\delta}{x}{y} &\dpe  \biggl( \R^d \setminus \p{\Cone{\delta}{x}{x-y}\cup \Cone{\delta}{y}{y-x}} \biggr)\cup\acc{x,y} 
	\end{split}
\end{align}
(see Figure~\ref{fig : intro/notations/diamond}).
\begin{figure}
\center
\def\svgwidth{0.6\linewidth}
\begingroup%
  \makeatletter%
  \providecommand\color[2][]{%
    \errmessage{(Inkscape) Color is used for the text in Inkscape, but the package 'color.sty' is not loaded}%
    \renewcommand\color[2][]{}%
  }%
  \providecommand\transparent[1]{%
    \errmessage{(Inkscape) Transparency is used (non-zero) for the text in Inkscape, but the package 'transparent.sty' is not loaded}%
    \renewcommand\transparent[1]{}%
  }%
  \providecommand\rotatebox[2]{#2}%
  \newcommand*\fsize{\dimexpr\f@size pt\relax}%
  \newcommand*\lineheight[1]{\fontsize{\fsize}{#1\fsize}\selectfont}%
  \ifx\svgwidth\undefined%
    \setlength{\unitlength}{154.98526626bp}%
    \ifx\svgscale\undefined%
      \relax%
    \else%
      \setlength{\unitlength}{\unitlength * \real{\svgscale}}%
    \fi%
  \else%
    \setlength{\unitlength}{\svgwidth}%
  \fi%
  \global\let\svgwidth\undefined%
  \global\let\svgscale\undefined%
  \makeatother%
  \begin{picture}(1,0.33700748)%
    \lineheight{1}%
    \setlength\tabcolsep{0pt}%
    \put(0,0){\includegraphics[width=\unitlength,page=1]{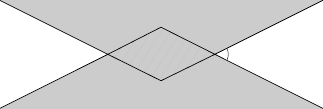}}%
    \put(0.71374595,0.16009468){\color[rgb]{0,0,0}\makebox(0,0)[lt]{\lineheight{1.25}\smash{\begin{tabular}[t]{l}$\delta$\end{tabular}}}}%
    \put(0.26094807,0.16009468){\color[rgb]{0,0,0}\makebox(0,0)[lt]{\lineheight{1.25}\smash{\begin{tabular}[t]{l}$\delta$\end{tabular}}}}%
    \put(0,0){\includegraphics[width=\unitlength,page=2]{FIG1_Diamond.pdf}}%
    \put(0.32258459,0.1309009){\color[rgb]{0,0,0}\makebox(0,0)[lt]{\lineheight{1.25}\smash{\begin{tabular}[t]{l}$x$\end{tabular}}}}%
    \put(0.65934327,0.13515766){\color[rgb]{0,0,0}\makebox(0,0)[lt]{\lineheight{1.25}\smash{\begin{tabular}[t]{l}$y$\end{tabular}}}}%
    \put(0,0){\includegraphics[width=\unitlength,page=3]{FIG1_Diamond.pdf}}%
  \end{picture}%
\endgroup%

\caption{Illustration of the definitions~\eqref{eqn : intro/notations/diamond} and~\eqref{eqn : intro/notations/antidiamond} in dimension $2$. The diamond $\Diamant{\delta}{x}{y}$ and the antidiamond $\AntiDiamant{\delta}{x}{y}$ are represented by the striped region and the shaded region respectively. On the figure, $1-\delta$ denote the cosine of the half-angle.}
\label{fig : intro/notations/diamond}
\end{figure}
For all $x=(x_1, \dots, x_d)\in \R^d$, we define
\begin{equation}
	\label{eqn : intro/notations/floor}
	\floor{x} = \p{\floor{x_1},\dots, \floor{x_d}}.
\end{equation}
We say that a function $f:\ball{0,1}\rightarrow \R$ is \emph{symmetric} if it is symmetric with respect to $u\mapsto -u$.
\paragraph{Positive and negative parts.} For every $a\in \R$, we use the notations $a^+ \dpe \max(a,0)$ and $a^- \dpe \max(-a,0)$.

\section{LLN for greedy animals and paths}
\label{sec : LLN}
In this section we prove Theorem~\ref{thm : intro/main/MAIN}. We fix $\AUXGeneric\in \acc{\AUXPath, \AUXAnimal}$.
Except when specified otherwise, the figures will relate to the case $\AUXGeneric = \AUXAnimal$ and $d=2$. For all $0< \delta < 1$, $\ell>0$ and distinct points $x,y\in \R^d$, we define
\begin{align}
	\SetGDC{\delta}{x}{y}{\ell} &\dpe \set{\xi \in \SetGDF{x}{y}{\ell} }{\xi \subseteq \Diamant{\delta}{x}{y}}\\
	\text{and}\quad%
	\SetGDR{\delta}{x}{y}{\ell} &\dpe \set{\xi \in \SetGDF{x}{y}{\ell} }{\xi \subseteq \AntiDiamant{\delta}{x}{y} \vphantom{\Diamant{\delta}{x}{y}} }.
\end{align}

\subsection{Pointwise convergence for animals and paths restricted to an antidiamond}
\label{subsec : LLN/simple}
This subsection aims to prove Proposition~\ref{prop : LLN/simple}, i.e.\  the pointwise analogue of Theorem~\ref{thm : intro/main/MAIN} for animals and paths restricted to an antidiamond. 
\begin{Proposition}
\label{prop : LLN/simple}
	There exists a concave, symmetric function $\LimMassG : \ball{0,1} \rightarrow \intervalleff0\GeneralUB$ such that for all $0<\delta<1$, $u\in \ball{0,1}\setminus\acc0$ and $-\infty < a < b < \infty$,
	\begin{equation}
		\label{eqn : LLN/simple}
		\frac{\MassGDR{\delta}{\ell a u}{\ell b u}{(b-a)\ell} }{(b-a)\ell} \xrightarrow[\ell \to \infty]{\text{a.s. and }\rL^1} \LimMassG(u).
	\end{equation}
	Moreover, for all $e\in \S$, $\beta \mapsto \LimMassG(\beta e)$ is nonincreasing on $\intervallefo01$ and uniformly continuous on $\intervalleoo{-1}{1}$.
\end{Proposition}
\begin{proof}
Equation~\eqref{eqn : LLN/simple} is a consequence of Lemmas~\ref{lem : LLN/simple/cvg},~\ref{lem : LLN/simple/cvg2},~\ref{lem : LLN/simple/inv_delta} and~\ref{lem : LLN/simple/concavity} below. Let $e\in \S$ and define
\begin{align*}
	f : \intervalleoo{-1}{1} &\longrightarrow \intervalleff0\GeneralUB \\ \beta &\longmapsto \LimMassG(\beta e).
\end{align*}
The function $f$ is even and concave, therefore for all $\beta \intervallefo01$,
\begin{equation*}
	f(\beta) = \frac{f(-\beta) + f(\beta)}{2} \le f(0).
\end{equation*}
In other words, $f$ has a maximum at $\beta=0$. Using concavity again, we deduce that $f$ is nonincreasing on $\intervallefo01$. Since $f$ is nonnegative, it has a finite limit at $-1$ and $1$. Moreover $f$ is continuous on $\intervalleoo{-1}{1}$, thus it has a continuous extension on $\intervalleff{-1}{1}$, thus it is uniformly continuous on $\intervalleoo{-1}1$.
\end{proof}
\begin{Lemma}
	\label{lem : LLN/simple/cvg}
	Let $u\in \ball{0,1}\setminus\acc0$ and $0<\delta <1$. Then there exists a constant $\LimMassGDC\delta(u) \in \intervalleff0\GeneralUB$ such that for all $a<b$,
	\begin{equation}
	\label{eqn : LLN/simple/cvg/diamond}
		\frac{ \MassGDC{\delta}{\ell a u}{\ell b u}{(b-a)\ell} }{(b-a)\ell} %
			\xrightarrow[\ell \to\infty]{\text{a.s. and }\rL^1} \LimMassGDC\delta(u).
	\end{equation}
\end{Lemma}
\begin{Lemma}
	\label{lem : LLN/simple/cvg2}
	Let $u\in \ball{0,1}\setminus\acc0$ and $0<\delta <1$. Then for all $a<b$,
	\begin{equation}
	\label{eqn : LLN/simple/cvg/antidiamond}
		\frac{ \MassGDR{\delta}{\ell a u}{\ell b u}{(b-a)\ell} }{(b-a)\ell} %
			\xrightarrow[\ell \to\infty]{\text{a.s. and }\rL^1} \LimMassGDC\delta(u).
	\end{equation}
\end{Lemma}
\begin{Lemma}
	\label{lem : LLN/simple/inv_delta}
	Let $u\in \ball{0,1}\setminus\acc0$. Then $\LimMassG(u) \dpe \LimMassGDC\delta(u)$ does not depend on $\delta$.
\end{Lemma}
\begin{Lemma}
	\label{lem : LLN/simple/concavity}
	The function $\LimMassG$ admits a concave, symmetric extension on $\ball{0,1}$. 
\end{Lemma}
\begin{proof}[Proof of Lemma~\ref{lem : LLN/simple/cvg}]
	\emph{Existence of the limit.} Let $u\in \ball{0,1}\setminus\acc0$, $0< \delta < 1$ and $a<b$. Consider the process defined for all $s< t$ by 
	\begin{equation}
		X(s,t) \dpe \MassGDC{\delta}{s u}{t u}{t-s}.
	\end{equation}
	The process $X$ is stationary since $\pro$ is stationary. We claim that on the almost sure event $\acc{\pro(\R u \times \intervalleoo0\infty) = 0}$, it is superadditive. Indeed let $s_1 < s_2 < s_3$. Let $\xi_1 \in \SetGDC{\delta}{s_1u}{s_2u}{s_2-s_1}$ and $\xi_2 \in \SetGDC{\delta}{s_2u}{s_3u}{s_3-s_2}$. Then
	\begin{equation}
		\label{eqn : LLN/simple/cvg/concat}
		\xi_1\concat\xi_2 \in \SetGDC{\delta}{s_1u}{s_3u}{s_3-s_1}.
	\end{equation}
	Moreover, $\xi_1\cap \xi_2 = \acc{s_2 u}$, therefore
	\begin{align}
		\Mass{\xi_1\concat\xi_2} &= \Mass{\xi_1} + \Mass{\xi_2}.\nonumber
		\intertext{By definition of $\MassGDC{\delta}{s_1 u}{s_3 u}{s_3-s_1}$,}
		\MassGDC{\delta}{s_1 u}{s_3 u}{s_3-s_1} &\ge \Mass{\xi_1} + \Mass{\xi_2}.\nonumber
	\end{align}
	Taking the supremum in $\xi_1$ and $\xi_2$, we get
	\begin{equation}
		\label{eqn : LLN/simple/cvg/superadd}
		\MassGDC{\delta}{s_1 u}{s_3 u}{s_3-s_1} \ge\MassGDC{\delta}{s_1 u}{s_2 u}{s_2-s_1}+ \MassGDC{\delta}{s_2 u}{s_3 u}{s_3-s_2},\nonumber
	\end{equation}
	i.e.\  $X$ is superadditive. Besides, Assumption~\ref{ass : intro/main/Moment} implies that $\sup_{t\ge 1}\frac{\E{X(0,t)}}{t}<\infty$. Theorem~\ref{thm : intro/outline/AK81} yields the existence of the limit
	\begin{equation}
		\label{eqn : LLN/simple/cvg/existence_limit}
		\LimMassGDCab{\delta}{a,b}(u) \dpe \lim_{\ell \to \infty} \frac{ \MassGDC{\delta}{\ell a u}{\ell b u}{(b-a)\ell} }{(b-a)\ell},
	\end{equation}
	a.s. and in $\rL^1$.

	\emph{Invariance of the limit.} We now prove that $\LimMassGDCab{\delta}{a,b}(u)$ is a deterministic constant and does not depend on $(a,b)$. Let $(a_n)$ and $(b_n)$ be two strictly monotone sequences of rational numbers such that
	\begin{equation*}
		\inclim{n\to\infty} a_n =a \text{ and } \declim{n\to \infty} b_n =b.
	\end{equation*}
	We claim that for all $z\in \R^d$ and $n\ge 1$, almost surely,
	\begin{equation}
		\label{eqn : LLN/simple/cvg/invariance/inequality_as}
		\LimMassGDCab{\delta}{a,b}(u)\cro{ T_{-z}\pro } \le \frac{b_n - a_n}{b-a}\LimMassGDCab{\delta}{a_n, b_n}(u).
	\end{equation}
	Indeed fix $z\in \R^d$ and $n \ge 1$. Consider an animal $\xi \in \SetGDC{\delta}{a\ell u+z}{b\ell u + z}{(b-a)\ell}$ and consider
	\begin{equation}
		\label{eqn : LLN/simple/cvg/inv_xi'}
		\xi' \dpe \p{ a_n\ell u , a\ell u +z}\concat \xi \concat \p{b\ell u +z , b_n\ell u }
	\end{equation}
	(see Figure~\ref{fig : LLN/simple/cvg/inv_xi'}). %
	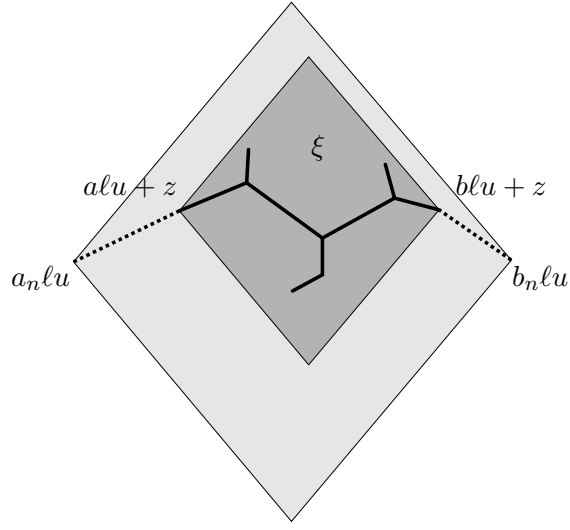
\begin{figure}
		\center
		\def\svgwidth{0.4\textwidth}
		\begingroup%
  \makeatletter%
  \providecommand\color[2][]{%
    \errmessage{(Inkscape) Color is used for the text in Inkscape, but the package 'color.sty' is not loaded}%
    \renewcommand\color[2][]{}%
  }%
  \providecommand\transparent[1]{%
    \errmessage{(Inkscape) Transparency is used (non-zero) for the text in Inkscape, but the package 'transparent.sty' is not loaded}%
    \renewcommand\transparent[1]{}%
  }%
  \providecommand\rotatebox[2]{#2}%
  \newcommand*\fsize{\dimexpr\f@size pt\relax}%
  \newcommand*\lineheight[1]{\fontsize{\fsize}{#1\fsize}\selectfont}%
  \ifx\svgwidth\undefined%
    \setlength{\unitlength}{210.33775029bp}%
    \ifx\svgscale\undefined%
      \relax%
    \else%
      \setlength{\unitlength}{\unitlength * \real{\svgscale}}%
    \fi%
  \else%
    \setlength{\unitlength}{\svgwidth}%
  \fi%
  \global\let\svgwidth\undefined%
  \global\let\svgscale\undefined%
  \makeatother%
  \begin{picture}(1,1.03054206)%
    \lineheight{1}%
    \setlength\tabcolsep{0pt}%
    \put(0,0){\includegraphics[width=\unitlength,page=1]{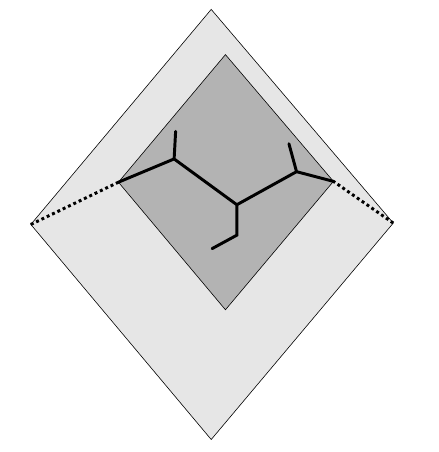}}%
    \put(-0.0474075,0.47233551){\color[rgb]{0,0,0}\makebox(0,0)[lt]{\lineheight{1.25}\smash{\begin{tabular}[t]{l}$a_n\ell u$\end{tabular}}}}%
    \put(0.89923294,0.47233551){\color[rgb]{0,0,0}\makebox(0,0)[lt]{\lineheight{1.25}\smash{\begin{tabular}[t]{l}$b_n\ell u$\end{tabular}}}}%
    \put(0.08820896,0.64530444){\color[rgb]{0,0,0}\makebox(0,0)[lt]{\lineheight{1.25}\smash{\begin{tabular}[t]{l}$a\ell u +z$\end{tabular}}}}%
    \put(0.79196985,0.64530444){\color[rgb]{0,0,0}\makebox(0,0)[lt]{\lineheight{1.25}\smash{\begin{tabular}[t]{l}$b\ell u +z$\end{tabular}}}}%
    \put(0.51890826,0.71957937){\color[rgb]{0,0,0}\makebox(0,0)[lt]{\lineheight{1.25}\smash{\begin{tabular}[t]{l}$\xi$\end{tabular}}}}%
  \end{picture}%
\endgroup%
		\caption{Construction of the animal $\xi'$ defined by~\eqref{eqn : LLN/simple/cvg/inv_xi'}. The lightly shaded region is $\Diamant{\delta}{ a_n\ell u}{ b_n \ell u}$. The shaded region is $\Diamant{\delta}{a \ell u + z }{ b \ell u + z}$. The animal $\xi'$ is the concatenation of $\xi$ (thick, solid lines) with the two segments represented by thick dotted lines. }
		\label{fig : LLN/simple/cvg/inv_xi'}
	\end{figure}%
	Its length satisfies
	\begin{align}
		\norme{\xi'} &= \norme{z + \ell u(a-a_n)}  + \norme\xi + \norme{ \ell u(b_n-b) -z }\eol
			&\le\norme{z + \ell u(a-a_n)} + (b-a)\ell + \norme{ \ell u(b_n-b) -z }\eol
			&\le \cro{ b-a + \norme u(b_n -b +a -a_n) + \frac{2\norme z}{\ell}}\ell.\nonumber
		\intertext{thus for large enough $\ell$,}
		\label{eqn : LLN/simple/cvg/invariance/length_xi'}
		\norme{\xi'} &\le \p{ b_n - a_n }\ell.
	\end{align}
	Besides, for large enough $\ell$, $\Diamant{\delta}{a\ell u+z}{b\ell u+z}\subseteq \Diamant{\delta}{a_n\ell u}{b_n\ell u}$, therefore
	\begin{equation*}
		\xi' \in \SetGDC{\delta}{ a_n\ell u}{b_n\ell u}{\p{ b_n - a_n }\ell}.
	\end{equation*}
	In particular,
	\begin{align}
		\Mass{\xi} \le \Mass{\xi'} &\le \MassGDC{\delta}{a_n\ell u}{b_n\ell u}{\p{ b_n - a_n }\ell}.\nonumber
		\intertext{Taking the supremum in $\xi$, we get}
		\MassGDC{\delta}{a\ell u+z}{b\ell u + z}{(b-a)\ell} &\le \MassGDC{\delta}{a_n\ell u}{b_n\ell u}{\p{ b_n - a_n }\ell}.\nonumber
	\end{align}
	Dividing by $(b-a)\ell$ and letting $\ell\to \infty$, we get~\eqref{eqn : LLN/simple/cvg/invariance/inequality_as}. 
	
	Letting $n\to \infty$ in~\eqref{eqn : LLN/simple/cvg/invariance/inequality_as} and applying~\eqref{eqn : intro/outline/AK81/convergence_ab}, we obtain

	\begin{equation}
		\label{eqn : LLN/simple/cvg/invariance/inequality_as2}
		\LimMassGDCab{\delta}{a,b}(u)\cro{ T_{-z} \pro} \le \LimMassGDCab{\delta}{a,b}(u).
	\end{equation}
	Since $z$ is any vector in $\R^d$, \eqref{eqn : LLN/simple/cvg/invariance/inequality_as2} is actually an equality. Besides, $\pro$ is ergodic therefore $\LimMassGDCab{\delta}{a,b}(u)$ is a.s. equal to its expectation, hence it does not depend on $a$ and $b$.
	\end{proof}

	\begin{proof}[Proof of Lemma~\ref{lem : LLN/simple/cvg2}]
	Let $u\in\ball{0,1}\setminus\acc0$, $0<\delta<1$ and $a<b$. Given~\eqref{eqn : LLN/simple/cvg/diamond} and the inequality
	\begin{equation*}
		\MassGDC{\delta}{a\ell u}{b\ell u}{(b-a)\ell} \le \MassGDR{\delta}{a\ell u}{b\ell u}{(b-a)\ell},
	\end{equation*}
	it is sufficient to prove the existence of $\Cl{cst : LLN/simple/cvg/big_diamond}>0$ such that for all $\ell>0$,
	\begin{equation}
	\label{eqn : LLN/antidiamond/cvg_simple/diamond_vs_antidiamond}
	\begin{split}
	\MassGDR{\delta}{a\ell u}{b\ell u}{(b-a)\ell}%
		&\le \MassGDC{\delta}{a\ell u - \Cr{cst : LLN/simple/cvg/big_diamond}(b-a)\ell u}{b\ell u + \Cr{cst : LLN/simple/cvg/big_diamond}(b-a)\ell u}{ (2\Cr{cst : LLN/simple/cvg/big_diamond}+1)(b-a)\ell}\\
		&\quad -\MassGDC{\delta}{a\ell u - \Cr{cst : LLN/simple/cvg/big_diamond}(b-a)\ell u}{a\ell u}{ \Cr{cst : LLN/simple/cvg/big_diamond}(b-a)\ell}\\
		&\quad -\MassGDC{\delta}{b\ell u}{b\ell u + \Cr{cst : LLN/simple/cvg/big_diamond}(b-a)\ell u}{\Cr{cst : LLN/simple/cvg/big_diamond}(b-a)\ell}.
	\end{split}
	\end{equation}
	Let $\Cr{cst : LLN/simple/cvg/big_diamond}>0$ be such that
	\begin{equation}
		\label{eqn : LLN/simple/cvg/diamond_vs_antidiamond/ball_in_diamond1}
		\clball{0,1} \subseteq \Diamant{\delta}{-\Cr{cst : LLN/simple/cvg/big_diamond}u}{(1+\Cr{cst : LLN/simple/cvg/big_diamond})u}.
	\end{equation}
	Consider three animals
	\begin{align}
		\xi &\in \SetGDR{\delta}{a\ell u}{b\ell u}{(b-a)\ell},\nonumber\\
		\xi_1 &\in \SetGDC{\delta}{a\ell u -\Cr{cst : LLN/simple/cvg/big_diamond}(b-a)\ell u }{a \ell u}{\Cr{cst : LLN/simple/cvg/big_diamond}(b-a)\ell},\nonumber\\
		\xi_2 &\in\SetGDC{\delta}{b\ell u}{b\ell u + \Cr{cst : LLN/simple/cvg/big_diamond}(b-a)\ell u}{\Cr{cst : LLN/simple/cvg/big_diamond}(b-a)\ell},\nonumber
	\end{align}
	and their concatenation	$\xi' \dpe \xi_1 \concat \xi \concat \xi_2$ (see Figure~\ref{fig : LLN/simple/cvg/diamond_vs_antidiamond/xi'}). We claim that
	\begin{equation}
		\label{eqn : LLN/simple/cvg/diamond_vs_antidiamond/xi'_good}
		\xi' \in \SetGDC{\delta}{a\ell u - \Cr{cst : LLN/simple/cvg/big_diamond}(b-a)\ell u}{b\ell u + \Cr{cst : LLN/simple/cvg/big_diamond}(b-a)\ell u}{(2\Cr{cst : LLN/simple/cvg/big_diamond} +1)(b-a)\ell},
	\end{equation}
	which is straightforward except for the inclusion $\xi' \subseteq \Diamant{\delta}{a\ell u - \Cr{cst : LLN/simple/cvg/big_diamond}(b-a)\ell u}{b\ell u + \Cr{cst : LLN/simple/cvg/big_diamond}(b-a)\ell u}$. By~\eqref{eqn : LLN/simple/cvg/diamond_vs_antidiamond/ball_in_diamond1},
	\begin{align}
		\xi &\subseteq \clball{a\ell u,(b-a)\ell} \subseteq \Diamant{\delta}{a\ell u -\Cr{cst : LLN/simple/cvg/big_diamond}(b-a)\ell u}{b\ell u  + \Cr{cst : LLN/simple/cvg/big_diamond}(b-a)\ell u}.\nonumber%
		\intertext{Moreover,}
		\xi_1 &\subseteq \Diamant{\delta}{a\ell u -\Cr{cst : LLN/simple/cvg/big_diamond}(b-a)\ell u }{a \ell u} \subseteq \Diamant{\delta}{a\ell u - \Cr{cst : LLN/simple/cvg/big_diamond}(b-a)\ell u}{b\ell u + \Cr{cst : LLN/simple/cvg/big_diamond}(b-a)\ell u}\nonumber\\
		\text{and }\xi_2 &\subseteq \Diamant{\delta}{b\ell u}{b\ell u + \Cr{cst : LLN/simple/cvg/big_diamond}(b-a)\ell u} \subseteq \Diamant{\delta}{a\ell u - \Cr{cst : LLN/simple/cvg/big_diamond}(b-a)\ell u}{b\ell u + \Cr{cst : LLN/simple/cvg/big_diamond}(b-a)\ell u},\nonumber%
	\end{align}
	thus~\eqref{eqn : LLN/simple/cvg/diamond_vs_antidiamond/xi'_good}.
	\begin{figure}
		\center
		\def\svgwidth{\textwidth}
		\begingroup%
  \makeatletter%
  \providecommand\color[2][]{%
    \errmessage{(Inkscape) Color is used for the text in Inkscape, but the package 'color.sty' is not loaded}%
    \renewcommand\color[2][]{}%
  }%
  \providecommand\transparent[1]{%
    \errmessage{(Inkscape) Transparency is used (non-zero) for the text in Inkscape, but the package 'transparent.sty' is not loaded}%
    \renewcommand\transparent[1]{}%
  }%
  \providecommand\rotatebox[2]{#2}%
  \newcommand*\fsize{\dimexpr\f@size pt\relax}%
  \newcommand*\lineheight[1]{\fontsize{\fsize}{#1\fsize}\selectfont}%
  \ifx\svgwidth\undefined%
    \setlength{\unitlength}{612.88455513bp}%
    \ifx\svgscale\undefined%
      \relax%
    \else%
      \setlength{\unitlength}{\unitlength * \real{\svgscale}}%
    \fi%
  \else%
    \setlength{\unitlength}{\svgwidth}%
  \fi%
  \global\let\svgwidth\undefined%
  \global\let\svgscale\undefined%
  \makeatother%
  \begin{picture}(1,0.29052706)%
    \lineheight{1}%
    \setlength\tabcolsep{0pt}%
    \put(0,0){\includegraphics[width=\unitlength,page=1]{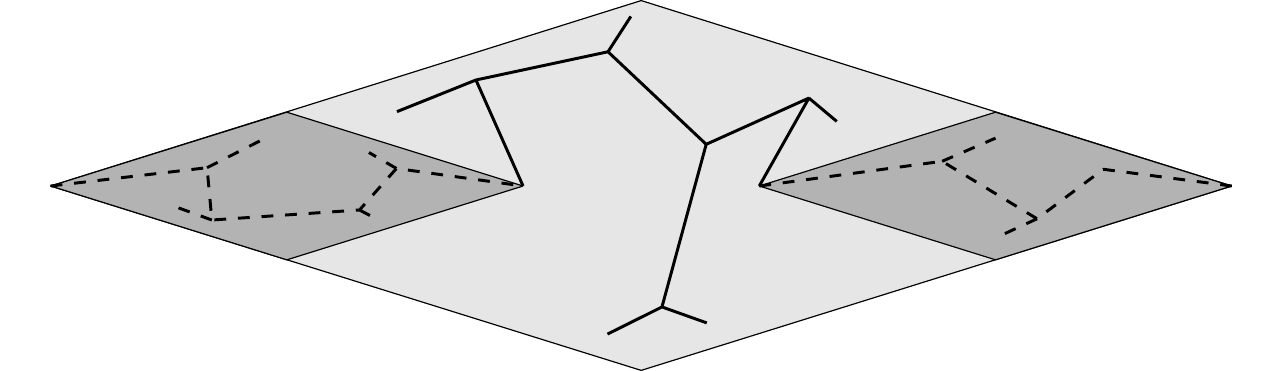}}%
    \put(0.40041114,0.1225915){\color[rgb]{0,0,0}\makebox(0,0)[lt]{\lineheight{1.25}\smash{\begin{tabular}[t]{l}$a\ell u$\\\end{tabular}}}}%
    \put(0.58667423,0.1225915){\color[rgb]{0,0,0}\makebox(0,0)[lt]{\lineheight{1.25}\smash{\begin{tabular}[t]{l}$b\ell u$\\\end{tabular}}}}%
    \put(0.91900551,0.10730292){\color[rgb]{0,0,0}\makebox(0,0)[lt]{\lineheight{1.25}\smash{\begin{tabular}[t]{l}$b\ell u + C_3(b-a)\ell u$\\\end{tabular}}}}%
    \put(-0.02634544,0.10301724){\color[rgb]{0,0,0}\makebox(0,0)[lt]{\lineheight{1.25}\smash{\begin{tabular}[t]{l}$a\ell u - C_3(b-a)\ell u$\\\end{tabular}}}}%
    \put(0.22279467,0.15804305){\color[rgb]{0,0,0}\makebox(0,0)[lt]{\lineheight{1.25}\smash{\begin{tabular}[t]{l}$\xi_1$\end{tabular}}}}%
    \put(0.79184487,0.16544223){\color[rgb]{0,0,0}\makebox(0,0)[lt]{\lineheight{1.25}\smash{\begin{tabular}[t]{l}$\xi_2$\end{tabular}}}}%
    \put(0.52754513,0.23173815){\color[rgb]{0,0,0}\makebox(0,0)[lt]{\lineheight{1.25}\smash{\begin{tabular}[t]{l}$\xi$\end{tabular}}}}%
  \end{picture}%
\endgroup%
		\caption{Construction of the animal $\xi'$ defined in the proof of Lemma~\ref{lem : LLN/simple/cvg2}. The lightly shaded region is $\Diamant{\delta}{a\ell u - C_3(b-a)\ell u}{b\ell u + C_3(b-a)\ell u}$.  The shaded regions are $\Diamant{\delta}{a\ell u -C_3(b-a)\ell u}{a\ell u}$ and $\Diamant{\delta}{b\ell u}{b\ell u + C_3(b-a)\ell u}$. The animal $\xi'$ is the concatenation of $\xi$ (thick, solid lines) with $\xi_1$ and $\xi_2$ (thick, dashed lines). }
		\label{fig : LLN/simple/cvg/diamond_vs_antidiamond/xi'}
	\end{figure}
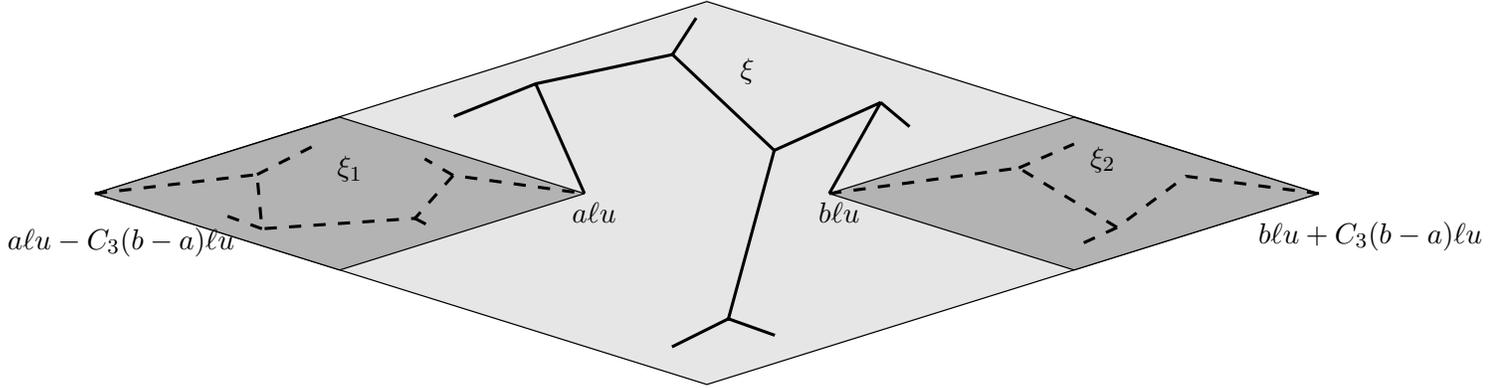%

	Besides, the intersection between any two animals among $\xi$, $\xi_1$ and $\xi_2$ is included in $\R u$, thus on the a.s.\ event $\acc{\pro\p{\R u \times \intervalleoo0\infty} = 0 }$, we have
	\begin{equation*}
		\Mass{\xi_1} + \Mass{\xi} + \Mass{\xi_2} = \Mass{\xi'}.
	\end{equation*}
	In particular, by~\eqref{eqn : LLN/simple/cvg/diamond_vs_antidiamond/xi'_good}, 
	\begin{equation*}
		\Mass{\xi_1} + \Mass{\xi} + \Mass{\xi_2} \le \MassGDC{\delta}{a\ell u - \Cr{cst : LLN/simple/cvg/big_diamond}(b-a)\ell u}{b\ell u + \Cr{cst : LLN/simple/cvg/big_diamond}(b-a)\ell u}{(2\Cr{cst : LLN/simple/cvg/big_diamond} +1)(b-a)\ell}  .
	\end{equation*}
	Taking the supremum with respect to $\xi_1$, $\xi$ and $\xi_2$ yields~\eqref{eqn : LLN/antidiamond/cvg_simple/diamond_vs_antidiamond}.
\end{proof}

\begin{proof}[Proof of Lemma~\ref{lem : LLN/simple/inv_delta}]
	By~\eqref{eqn : LLN/simple/cvg/diamond} and~\eqref{eqn : LLN/simple/cvg/antidiamond}, $\delta \mapsto \LimMassGDC\delta(u)$ is both nonincreasing and nondecreasing on $\intervalleoo01$, therefore it is constant.
\end{proof}

\begin{proof}[Proof of Lemma~\ref{lem : LLN/simple/concavity}]
	Define
	\begin{equation}
		\LimMassG(0) \dpe \limsup_{u\to 0} \LimMassG(u).
	\end{equation}
	The symmetry is a consequence of~\eqref{eqn : LLN/simple/cvg/diamond} and the stationnarity of $\pro$.

	We make the following claim, which is somewhat weaker than concavity: for all $u_1, u_2 \in \ball{0,1}\setminus\acc0$ such that $u_2\notin \R^-u_1$, for all $0 < \theta_1, \theta_2< 1$ such that $\theta_1+\theta_2 =1$,
	\begin{equation}
		\label{eqn : LLN/simple/concavity/weak_version}
	 	\theta_1 \LimMassG(u_1) + \theta_2\LimMassG(u_2) \le \LimMassG\p{\theta_1u_1 + \theta_2u_2}.
	\end{equation}
	Indeed let $u_1,u_2,\theta_1, \theta_2$ be as above and $0<\delta<1$ small enough so that for all $\ell>0$,
	\begin{align*}
		\Diamant{\delta}{0}{\theta_1 \ell u_1} &\subseteq \AntiDiamant{\delta}{0}{\theta_1\ell u_1 + \theta_2\ell u_2},\\
		\Diamant{\delta}{\theta_1 \ell u_1}{\theta_1 \ell u_1 + \theta_2\ell u_2 } &\subseteq \AntiDiamant{\delta}{0}{\theta_1\ell u_1 + \theta_2\ell u_2},
	\end{align*}
	\begin{equation*}
		\Diamant{\delta}{0}{\theta_1 \ell u_1}\cap \Diamant{\delta}{\theta_1 \ell u_1}{\theta_1 \ell u_1 + \theta_2\ell u_2 } = \acc{\theta_1 \ell u_1}.
	\end{equation*} Let $\ell>0$, $\xi_1\in \SetGDC{\delta}{0}{\theta_1 \ell u_1}{\theta_1 \ell}$ and $\xi_2\in \SetGDC{\delta}{\theta_1\ell u_1}{\theta_1\ell u_1+ \theta_2\ell u_2}{\theta_2\ell}$. Define $\xi\dpe \xi_1\concat \xi_2$ (see Figure~\ref{fig : LLN/simple/concavity}).
	\begin{figure}
		\center
		\def\svgwidth{0.3\linewidth}
		\begingroup%
  \makeatletter%
  \providecommand\color[2][]{%
    \errmessage{(Inkscape) Color is used for the text in Inkscape, but the package 'color.sty' is not loaded}%
    \renewcommand\color[2][]{}%
  }%
  \providecommand\transparent[1]{%
    \errmessage{(Inkscape) Transparency is used (non-zero) for the text in Inkscape, but the package 'transparent.sty' is not loaded}%
    \renewcommand\transparent[1]{}%
  }%
  \providecommand\rotatebox[2]{#2}%
  \newcommand*\fsize{\dimexpr\f@size pt\relax}%
  \newcommand*\lineheight[1]{\fontsize{\fsize}{#1\fsize}\selectfont}%
  \ifx\svgwidth\undefined%
    \setlength{\unitlength}{291.7072975bp}%
    \ifx\svgscale\undefined%
      \relax%
    \else%
      \setlength{\unitlength}{\unitlength * \real{\svgscale}}%
    \fi%
  \else%
    \setlength{\unitlength}{\svgwidth}%
  \fi%
  \global\let\svgwidth\undefined%
  \global\let\svgscale\undefined%
  \makeatother%
  \begin{picture}(1,0.52576036)%
    \lineheight{1}%
    \setlength\tabcolsep{0pt}%
    \put(0,0){\includegraphics[width=\unitlength,page=1]{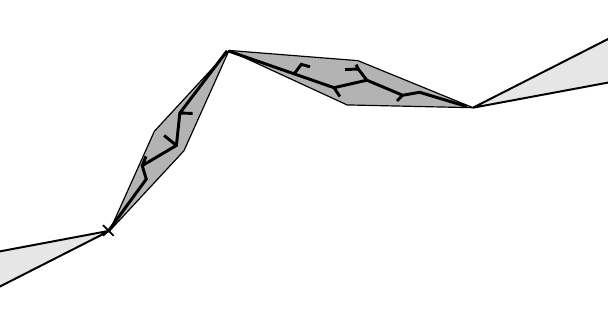}}%
    \put(0.31281967,0.26389554){\color[rgb]{0,0,0}\makebox(0,0)[lt]{\lineheight{1.25}\smash{\begin{tabular}[t]{l}$\xi_1$\end{tabular}}}}%
    \put(0.55574246,0.29183502){\color[rgb]{0,0,0}\makebox(0,0)[lt]{\lineheight{1.25}\smash{\begin{tabular}[t]{l}$\xi_2$\end{tabular}}}}%
    \put(0.15958492,0.0818664){\color[rgb]{0,0,0}\makebox(0,0)[lt]{\lineheight{1.25}\smash{\begin{tabular}[t]{l}$0$\end{tabular}}}}%
    \put(0.33630078,0.47821333){\color[rgb]{0,0,0}\makebox(0,0)[lt]{\lineheight{1.25}\smash{\begin{tabular}[t]{l}$\theta_1 \ell u_1$\end{tabular}}}}%
    \put(0.74907304,0.29696674){\color[rgb]{0,0,0}\makebox(0,0)[lt]{\lineheight{1.25}\smash{\begin{tabular}[t]{l}$\theta_1 \ell u_1 + \theta_2 \ell u_2$ \end{tabular}}}}%
  \end{picture}%
\endgroup%
		\caption{Illustration of the animal $\xi$ defined in the proof of Lemma~\ref{lem : LLN/simple/concavity}. The diamonds $\Diamant{\delta}{0}{\theta_1 \ell u_1}$ and $\Diamant{\delta}{\theta_1 \ell u_1}{\theta_1 \ell u_1 + \theta_2\ell u_2 }$ are represented by the shaded regions. The antidiamond $\AntiDiamant{\delta}{0}{\theta_1\ell u_1 + \theta_2\ell u_2}$ is represented by the complementary of the lightly shaded region. The animal $\xi$ is the concatenation of $\xi_1$ and $\xi_2$ (thick lines).}
		\label{fig : LLN/simple/concavity}
	\end{figure}
	Since $\xi_1\cap \xi_2 = \acc{\theta_1 \ell u_1}$ and $\xi\in \SetGDR{\delta}{0}{\ell(\theta_1 u_1 + \theta_2 u_2)}{\ell}$, on the a.s.\ event $\acc{\pro\p{\R u \times \intervalleoo0\infty}=0 }$,
	\begin{equation*}
		\Mass{\xi_1}  + \Mass{\xi_2} \le \MassGDR{\delta}{0}{\ell(\theta_1 u_1 + \theta_2 u_2)}{\ell}.
	\end{equation*}
	Taking the supremum in $\xi_1$ and $\xi_2$ leads to
	\begin{align}
		\MassGDC{\delta}{0}{\theta_1 \ell u_1}{\theta_1 \ell} + \MassGDC{\delta}{\theta_1\ell u_1}{\theta_1\ell u_1+ \theta_2\ell u_2}{\theta_2\ell}%
			&\le \MassGDR{\delta}{0}{\ell(\theta_1 u_1 + \theta_2 u_2)}{\ell},\nonumber%
	\intertext{thus}
		\frac{\E{\MassGDC{\delta}{0}{\theta_1 \ell u_1}{\theta_1 \ell}}}{\ell} + \frac{\E{\MassGDC{\delta}{\theta_1\ell u_1}{\theta_1\ell u_1+ \theta_2\ell u_2}{\theta_2\ell}}}{\ell}%
		&\le \frac{\E{\MassGDR{\delta}{0}{\ell(\theta_1 u_1 + \theta_2 u_2)}{\ell}}}{\ell}.\nonumber%
	\intertext{Applying stationarity, we get}
		\theta_1\cdot\frac{\E{\MassGDC{\delta}{0}{\theta_1 \ell u_1}{\theta_1 \ell}}}{\theta_1 \ell} + \theta_2\cdot\frac{\E{\MassGDC{\delta}{0}{\theta_2\ell u_2}{\theta_2\ell}}}{\theta_2\ell}%
			&\le \frac{\E{\MassGDR{\delta}{0}{\ell(\theta_1 u_1 + \theta_2 u_2)}{\ell}}}{\ell}.\nonumber
	\end{align}
	Letting $\ell\to \infty$ and using Lemmas~\ref{lem : LLN/simple/cvg} and~\ref{lem : LLN/simple/cvg2}, we obtain~\eqref{eqn : LLN/simple/concavity/weak_version}.

	We now prove~\eqref{eqn : LLN/simple/concavity/weak_version} in full generality, i.e.\  that $\LimMassG$ is concave on $\ball{0,1}$. The inequality~\eqref{eqn : LLN/simple/concavity/weak_version} implies that the restrictions of $\LimMassG$ on balls included in $\ball{0,1}\setminus\acc0$ are concave, thus continuous. Let $u_1, u_2 \in \ball{0,1}$ and $0< \theta_1, \theta_2 < 1$ such that $\theta_1 + \theta_2 = 1$. The only non trivial cases left to consider are
	\begin{enumerate}
		\item $u_1=0, u_2\neq 0$, 
		\item $u_1 = \lambda u_2$, with $\lambda<0$ and $u_1, u_2\in \ball{0,1}\setminus\acc0$.
	\end{enumerate}
	For the first one, let $(u_1^{(n)})_{n\ge1}$ be a sequence of elements of $\ball{0,1}\setminus\acc0$ converging to $u_1=0$ such that \[ \lim_{n\to\infty}\LimMassG\p{u_1^{(n)}} = \LimMassG(0), \] and $(u_2^{(n)})_{n\ge1}$ be a sequence converging to $u_2$ such that for all $n$, $u_1^{(n)}$ and $u_2^{(n)}$ are linearly independent. Thanks to the claim~\eqref{eqn : LLN/simple/concavity/weak_version} applied to $u_1^{(n)}$ and $u_2^{(n)}$,
	\begin{equation*}
		\theta_1 \LimMassG\p{u_1^{(n)}} +\theta_2 \LimMassG\p{u_2^{(n)}} \le \LimMassG\p{\theta_1 u_1^{(n)} + \theta_2 u_2^{(n)}}.
	\end{equation*}
	Since $\LimMassG$ is continuous at $\theta_2 u_2$ and $u_2$, letting $n\to\infty$ gives~\eqref{eqn : LLN/simple/concavity/weak_version} for $u_1$ and $u_2$. For the second one, let $(u_1^{(n)})_{n\ge1}$ and $(u_2^{(n)})_{n\ge1}$ be sequences converging to $u_1$ and $u_2$ respectively, such that for all $n$, $u_1^{(n)}$ and $u_2^{(n)}$ are linearly independent. The end of the argument is analogous to the first case.
\end{proof}
\subsection{Uniform upper bound for animals and paths restricted to an antidiamond}
\label{subsec : LLN/antidiamond}
The goal of this section is to prove Proposition~\ref{prop : LLN/antidiamond}.
\begin{Proposition}
	\label{prop : LLN/antidiamond}
	Let $e\in \S$, $0<\delta<1$ and $0<\alpha < 1$. Almost surely,
	\begin{equation}
	\label{eqn : LLN/antidiamond/mainUB}
	 \lim_{L \to \infty} \sup\set{ \p{\frac{\MassGDR{\delta}{Lx}{Ly}{L\ell} }{L} - \ell\LimMassG\p{\frac{x-y}{\ell} }  }^+}%
		{\begin{array}{c} x,y\in \intervalleff{-e}{e} \\ \alpha  \le \ell \le 1 \\ \ell > \norme{x-y} \end{array} } = 0.
	\end{equation}
\end{Proposition}
\begin{Remark}
	The analoguous result for $\AUXDiamant{\AUXMassGeneric}{\delta}$ is also true but we don't make it a proposition since it is not needed in the proof of Theorem~\ref{thm : intro/main/MAIN}. 
\end{Remark}
Fix $e$, $\delta$ and $\alpha$ as in Proposition~\ref{prop : LLN/antidiamond}. Given $x,y \in \R e$, we write $x\le y$ if $y-x \in \R^+ e$. We proceed as follows:
\begin{enumerate}
	\item Establish an upper bound for $\MassGDR{\delta}{Lx}{Ly}{L\ell}$ for $(x,y,\ell)$ taking values in a finite, $ \frac{1}{N}$-dense set of parameters thanks to Proposition~\ref{prop : LLN/simple}.
	\item Extend this bound to any values of $(x,y,\ell)$ with Lemma~\ref{lem : LLN/antidiamond/adjusting_format}.
\end{enumerate}
\begin{Lemma}
	\label{lem : LLN/antidiamond/adjusting_format}
	For all $x,y,x',y'\in \R e$ such that $x' \le x \le y \le y'$ and $\ell>\norme{x-y}$, for all $0<\delta<1$, for all $L >0$,
	\begin{equation}
		\label{eqn : LLN/antidiamond/adjusting_format}
		\MassGDR{\delta}{L x}{L y}{L \ell} %
			\le \MassGDR{\delta}{L x'}{ L y'}{ L \ell + L\norme{x-x'} + L\norme{y-y'} }.
	\end{equation}
\end{Lemma}
We first prove Proposition~\ref{prop : LLN/antidiamond}, assuming Lemma~\ref{lem : LLN/antidiamond/adjusting_format} is true.
\begin{proof}[Proof of Proposition~\ref{prop : LLN/antidiamond}]
	Let $N\ge 1$ be an integer. Consider the sets
	\begin{equation*}
		E_1 \dpe \set{ \frac{n}{N} e}{ n\in \intint{-N}{N} },%
		\quad E_2 \dpe \set{ \frac{n}{N} }{n \in \intint1{N+2} }.
	\end{equation*}
	Proposition~\ref{prop : LLN/simple} implies that
	\begin{equation}
		\label{eqn : LLN/antidiamond/errorX}
		\cError(L) \dpe \max\set{ \p{ \frac{\MassGDR{\delta}{L x}{L y}{L \ell} }{L} - \ell\LimMassG\p{\frac{x-y}{\ell} } }^+}{\begin{array}{c}x,y\in E_1\\ \ell \in E_2 \\ \ell > \norme{x-y} > 0 \end{array}} \xrightarrow[\ell \to \infty]{\text{a.s. and }\rL^1} 0.\\
	\end{equation}

	Let $L>0$, $x,y \in \intervalleff{- e}{ e}$ and $\alpha  \le \ell \le 1$ such that $\ell > \norme{x-y}$. Let $\omega$ be a modulus of continuity of $\beta \mapsto \LimMassG(\beta e)$ on $\intervalleoo{-1}{1}$. Without loss of generality, assume $x\le y$.  There exists distinct $x', y' \in E_1$ such that
	\begin{equation}
	\label{eqn : LLN/antidiamond/approx_a}
	x- \frac eN \le x'\le x \le y \le y' \le y+\frac eN.
	\end{equation} By \eqref{eqn : LLN/antidiamond/adjusting_format},
	\begin{align}
		\MassGDR{\delta}{Lx}{Ly}{L\ell} %
			&\le \MassGDR{\delta}{Lx'}{Ly'}{ L\ell + \frac{2L}N }. \nonumber%
		\intertext{Let $\ell'\dpe \frac{1}{N} \ceil{\ell N +2}  $. Note that $\ell'\in E_2$ and $\ell' \ge \ell +  \frac{2}{N} > \norme{x'-y'}$. Therefore, by definition of $\cError(L)$,}
		\frac{ \MassGDR{\delta}{L x}{L y}{L \ell} }L %
			&\le \frac{ \MassGDR{\delta}{L x'}{L y'}{L \ell'} }L\nonumber\\
			&\le \ell'\LimMassG\p{\frac{x'-y'}{\ell'} } + \cError(L).\nonumber
	\end{align}
	Since $\ell' \le \frac1N(\ell N + 3)$,
	\begin{equation*}
	 	\norme{x-y}\p{\frac1\ell - \frac{1}{\ell'}} \le \ell \p{\frac1\ell - \frac{1}{\ell'}}\le 1- \frac{\ell N}{\ell N +3} = \frac{3}{\ell N+3}\le \frac3{\ell N}.
	\end{equation*} Besides $\beta\mapsto \LimMassG(\beta e)$ is nonincreasing on $\intervallefo01$, therefore
	\begin{align}
		\frac{ \MassGDR{\delta}{L x}{L y}{L \ell} }L%
			&\le \ell'\LimMassG\p{\frac{x-y}{\ell'}} + \cError(L) \nonumber\\
			&\le \p{\ell + \frac{3}{N}} \cro{\LimMassG\p{\frac{x-y}{\ell}} + \omega\p{\frac{3}{\ell N}   }  } + \cError(L).\nonumber
	\end{align}
	Consequently,
	\begin{align}
		\p{ \frac{ \MassGDR{\delta}{Lx}{Ly}{L\ell} }L -  \ell\LimMassG\p{\frac{x-y}{\ell}} }^+%
			&\le  \frac{3}{ N}\cro{\LimMassG\p{\frac{x-y}{\ell}} + \omega\p{ \frac{3}{\ell N} } }%
				+ \ell \omega\p{\frac{3}{\ell N} } + \cError(L).\nonumber
		\intertext{Since $\ell \ge \alpha $,}
		\p{ \frac{ \MassGDR{\delta}{Lx}{Ly}{L\ell} }L -  \ell\LimMassG\p{\frac{x-y}{\ell}} }^+%
			&\le  \frac{3}{N}\cro{\LimMassG\p{\frac{x-y}{\ell}} + \omega\p{ \frac{3}{\alpha N} } }%
				+ \ell  \omega\p{ \frac{3}{\alpha N} }  + \cError(L),\nonumber
	\end{align}
	thus 
	\begin{equation}
		\label{eqn : LLN/antidiamond/UB_pre_limit}
		\begin{split}
		&\sup\set{ \p{\frac{\MassGDR{\delta}{L x}{L y}{L \ell} }{L} - \ell \LimMassG\p{\frac{x-y}{\ell} }  }^+}%
		{\begin{array}{c} x,y\in \intervalleff{- e}{e} \\ \alpha  \le \ell \le 1 \\ \ell > \norme{x-y} \end{array} } \\
			&\quad\le \frac{3}{ N}\cro{\LimMassG\p{0} + \omega\p{ \frac{3}{\alpha N} } }%
				+ \omega\p{ \frac{3}{\alpha N} }  + \cError(L).
		\end{split}
	\end{equation}
	Consequently, almost surely,
	\begin{equation*}
	\begin{split}
		&\limsup_{L\to \infty} \sup\set{ \p{\frac{\MassGDR{\delta}{L x}{L y}{L \ell} }{L} - \ell \LimMassG\p{\frac{x-y}{\ell} }  }^+}%
		{\begin{array}{c} x,y\in \intervalleff{- e}{e} \\ \alpha  \le \ell \le 1 \\ \ell > \norme{x-y} \end{array} } \\
			&\quad\le \frac{3}{  N}\cro{\LimMassG\p{0} + \omega\p{ \frac{3}{\alpha N} } }%
				+   \omega\p{ \frac{3}{\alpha N} }.
		\end{split}
	\end{equation*}
	Letting $N\to \infty$ gives~\eqref{eqn : LLN/antidiamond/mainUB}.
\end{proof}
\begin{proof}[Proof of Lemma~\ref{lem : LLN/antidiamond/adjusting_format}]
	Let $x,y,x',y',\delta, \ell$ and  $L$ as in the lemma. Let $\xi \in \SetGDR{\delta}{L x}{L y}{L \ell}$. Define the animal
	\begin{equation}
		\xi' \dpe (L x',L x) \concat \xi \concat (L y,L y').
	\end{equation}
	It is straightforward to check that \[ \xi' \subseteq \AntiDiamant{\delta}{Lx'}{Ly'}. \]
	Moreover,
	\begin{equation*}
		\norme{\xi'} \le L\norme{x-x'}  + \norme{\xi} + L\norme{y-y'}%
			\le L\norme{x-x'}  + L\ell + L\norme{y-y'}%
	\end{equation*}
	thus
	\begin{equation*}
	\xi' \in \SetGDR{\delta}{L x'}{L y'}{L \ell + L \norme{x-x'} + L \norme{y-y'} }.
	\end{equation*}
	Consequently,
	\begin{equation*}
		\Mass{\xi} \le \MassGDR{\delta}{Lx'}{Ly'}{L\ell + L\norme{x-x'} + L\norme{y-y'} }.
	\end{equation*}
	Taking the supremum with respect to $\xi$ concludes the proof.
\end{proof}
\subsection{Pointwise convergence for directed animals and paths}
\label{subsec : LLN/directed}

In this section we prove Proposition~\ref{prop : LLN/directed/cvg} which is a pointwise version of the almost sure part of~\eqref{eqn : intro/main/MAIN/cvg}.
\begin{Proposition}
\label{prop : LLN/directed/cvg}
	For all $e\in \S$ and $0\le \beta < 1$, for $a<b$, almost surely,
	\begin{equation}
	\label{eqn : LLN/directed/cvg}
		\lim_{\ell \to \infty }\frac{ \MassGDF{a\ell \beta e}{b \ell \beta e}{(b-a)\ell}}{(b-a)\ell}  = \LimMassG(\beta e).
	\end{equation}
\end{Proposition}
The hard part is to show an upper bound for $\MassGDF{a\ell \beta e}{b \ell \beta e}{(b-a)\ell}$. In the case $\AUXGeneric = \AUXAnimal$, we show that any animal $\xi\in \SetADF{a\ell \beta e}{b \ell \beta e}{(b-a)\ell}$ is included in a slightly larger animal $\xi'\in \SetADR{\delta}{x}{y}{\ell'}$, with $(x,y)$ depending on $\xi$, such that $\frac{\norme{x-y}}{\ell'}\gtrsim \beta$ : it is sufficient to link the leftmost and rightmost points of $\xi$ belonging to thin cones around $\R e$ to their projections on $\R e$. Proposition~\ref{prop : LLN/antidiamond} then provides $\Mass{\xi'}\lesssim \ell' \LimMassA\p{\frac{y-x}{\ell'}}$. Since $\hat \beta \mapsto \LimMassA\p{\hat \beta e}$ is nonincreasing on $\intervallefo01$, this gives a suitable upper bound for $\MassADF{a\ell \beta e}{b \ell \beta e}{(b-a)\ell}$. 

In the case $\AUXGeneric = \AUXPath$, the same argument does not allow to conclude, because paths may only be concatenated on their endpoints. A variant of this issue already arose in Gandolfi and Kesten \cite{Gan94}, and Martin \cite{Mar02}. We add a preliminary step consisting essentially in writing any path $\gamma \in \SetPDF{0}{\ell \beta e }{\ell}$ as the concatenation of subpaths whose endpoints are also their leftmost and rightmost points on thin cylinders around $\R e$. The mass of each subpath may be controlled with the argument used in the case $\AUXGeneric = \AUXAnimal$. We then apply the concavity of $\LimMassG$ to bound the total mass.    

In both cases, we use Lemma~\ref{lem : LLN/directed/cone_bound} to control the additional length introduced by our constructions.
\begin{Lemma}
	\label{lem : LLN/directed/cone_bound}
	Let $x\in \R^d$, $v\in \R^d\setminus\acc0$, $0<\delta<1$ and $y\in \Cone{\delta}{x}{v}$. Denote by $p(y)$ the orthogonal projection of $y$ on $x+\R v$. Then
	\begin{equation}
		\label{eqn : LLN/directed/cone_bound1}
		\norme{y-p(y)}\le \Cl{cst : cone_bound} \sqrt\delta \norme{y-x},\\
	\end{equation}
	where $\Cr{cst : cone_bound}$ only depends on $\norme{\cdot}$.
\end{Lemma}
\begin{proof}[Proof of Lemma~\ref{lem : LLN/directed/cone_bound}] 
 	Without loss of generality, we assume $x=0$ and $\norme[2]{v}=1$. Since $y\in \Cone{\delta}{x}{v}$,
	\begin{equation*}
		\norme[2]{p(y)} = \ps{y}{v} \ge (1-\delta) \norme[2]{y}.
	\end{equation*}
	Consequently,
	\begin{align*}
		\norme[2]{y-p(y)}^2 %
			&= \norme[2]{y}^2 - \norme[2]{p(y)}^2\\
			&\le \cro{1-(1-\delta)^2} \norme[2]{y}^2 \\
			&\le 2\delta \cdot \norme[2]{y}^2.
	\end{align*}
	Applying the norm equivalence~\eqref{eqn : intro/notations/equiv_norms} yields~\eqref{eqn : LLN/directed/cone_bound1}. 

\begin{proof}[Proof of Proposition~\ref{prop : LLN/directed/cvg}]
Fix $e\in \S$, $0\le \beta <1$ and $a<b$. By the straightforward lower bound
\begin{equation*}
	\MassGDF{a \ell \beta e}{b \ell \beta e}{(b-a)\ell} \ge \MassGDR{1/2}{a \ell \beta e}{b \ell \beta e}{(b-a)\ell}
\end{equation*}
and by Proposition~\ref{prop : LLN/simple}, almost surely,
\begin{equation}
	\liminf_{\ell \to \infty} \frac{ \MassGDF{a \ell \beta e}{b \ell \beta e}{(b-a)\ell} }{(b-a)\ell} \ge \LimMassG(\beta e).
\end{equation}

We now turn to the upper bound, i.e.\  we prove that almost surely,
\begin{equation}
	\label{eqn : LLN/directed/UB_generic}
	\limsup_{\ell \to \infty} \frac{ \MassGDF{a \ell \beta e}{b \ell \beta e}{(b-a)\ell} }{(b-a)\ell} \le \LimMassG(\beta e).
\end{equation}

\emph{Case 1: Assume that } $\AUXGeneric = \AUXAnimal$.%
\stepcounter{error}
Let $0<\delta <1$. Proposition~\ref{prop : LLN/antidiamond} implies that
\stepcounter{error}%
\begin{equation}
	\label{eqn : LLN/directed/animals/X}
	\begin{split}
	\cError(\ell)\dpe\sup &\Biggl\{ \p{\frac{\MassADR{\delta}{x}{y}{(1+2\Cr{cst : cone_bound}\sqrt\delta )(b-a)\ell } }{(1+2\Cr{cst : cone_bound}\sqrt\delta )(b-a)\ell } - \LimMassA\p{\frac{x-y }{(1+2\Cr{cst : cone_bound}\sqrt\delta )(b-a)\ell }  }}^+ \Biggm|\\
	&\quad\begin{array}{c} x,y\in \intervalleff{- 3 \p{ \module a\vee \module b}(1+\Cr{cst : cone_bound}\sqrt\delta )\ell e }{  3 \p{ \module a\vee \module b}(1+\Cr{cst : cone_bound}\sqrt\delta )\ell e} \\ \norme{x-y}< (1+2\Cr{cst : cone_bound}\sqrt\delta )(b-a)\ell \end{array} \Biggr\} \xrightarrow[\ell \to \infty]{\text{a.s.}} 0.
	\end{split}
\end{equation}
Let $\ell>0$ and $\xi\in \SetADF{a \ell \beta e}{b \ell \beta e}{(b-a)\ell}$. Consider two points
\begin{equation*}
	x \in \argmin_{z\in \xi \cap \Cone{\delta}{a \ell \beta e}{-e}} \ps{z}{e} \text{ and } y \in \argmax_{z\in \xi \cap \Cone{\delta}{b \ell \beta e}{e}} \ps{z}{e},
\end{equation*}
and denote by $p(x)$ and $p(y)$ their orthogonal projections on $\R e$. By Lemma~\ref{lem : LLN/directed/cone_bound},
\begin{equation}
	\label{eqn : LLN/directed/animals/projection1}
	\norme{x - p(x)} \le \Cr{cst : cone_bound}\sqrt\delta \norme{x - a\ell \beta e} \le \Cr{cst : cone_bound}\sqrt\delta (b-a)\ell.
\end{equation}
In particular, by triangle inequality,
\begin{equation*}
	\norme{p(x) - a\ell \beta e } \le \cro{ 1+ \Cr{cst : cone_bound}\sqrt\delta}(b-a)\ell,
\end{equation*}
thus
\begin{equation}
	\label{eqn : LLN/directed/animals/projection2}
	p(x) \in \intervalleff{-3(\module a \vee \module b)(1+\Cr{cst : cone_bound}\sqrt\delta)\ell  e}{3(\module a \vee \module b)(1+\Cr{cst : cone_bound}\sqrt\delta)\ell e}.
\end{equation}
Likewise,
\begin{equation}
	\label{eqn : LLN/directed/animals/projection3}
	\norme{y-p(y)} \le  \Cr{cst : cone_bound}\sqrt\delta (b-a)\ell
\end{equation}
and%
\begin{equation}
	\label{eqn : LLN/directed/animals/projection4}
	p(y) \in \intervalleff{-3(\module a \vee \module b)(1+\Cr{cst : cone_bound}\sqrt\delta)\ell  e}{3(\module a \vee \module b)(1+\Cr{cst : cone_bound}\sqrt\delta)\ell e}.
\end{equation}
Consider the animal
\begin{equation}
	\label{eqn : LLN/directed/xi'}
	\xi' \dpe \p{p(x) , x} \concat \xi \concat \p{y, p(y)}
\end{equation}
(see Figure~\ref{fig : LLN/cvg_free/construction}). By~\eqref{eqn : LLN/directed/animals/projection1} and~\eqref{eqn : LLN/directed/animals/projection3},
\begin{equation*}
	\xi' \in \SetADR{\delta}{p(x)}{p(y)}{(b-a)\ell (1+2\Cr{cst : cone_bound}\sqrt\delta )},
\end{equation*}
\begin{figure}
	\center
	\def\svgwidth{\textwidth}
	\begingroup%
  \makeatletter%
  \providecommand\color[2][]{%
    \errmessage{(Inkscape) Color is used for the text in Inkscape, but the package 'color.sty' is not loaded}%
    \renewcommand\color[2][]{}%
  }%
  \providecommand\transparent[1]{%
    \errmessage{(Inkscape) Transparency is used (non-zero) for the text in Inkscape, but the package 'transparent.sty' is not loaded}%
    \renewcommand\transparent[1]{}%
  }%
  \providecommand\rotatebox[2]{#2}%
  \newcommand*\fsize{\dimexpr\f@size pt\relax}%
  \newcommand*\lineheight[1]{\fontsize{\fsize}{#1\fsize}\selectfont}%
  \ifx\svgwidth\undefined%
    \setlength{\unitlength}{343.27355332bp}%
    \ifx\svgscale\undefined%
      \relax%
    \else%
      \setlength{\unitlength}{\unitlength * \real{\svgscale}}%
    \fi%
  \else%
    \setlength{\unitlength}{\svgwidth}%
  \fi%
  \global\let\svgwidth\undefined%
  \global\let\svgscale\undefined%
  \makeatother%
  \begin{picture}(1,0.33898908)%
    \lineheight{1}%
    \setlength\tabcolsep{0pt}%
    \put(0,0){\includegraphics[width=\unitlength,page=1]{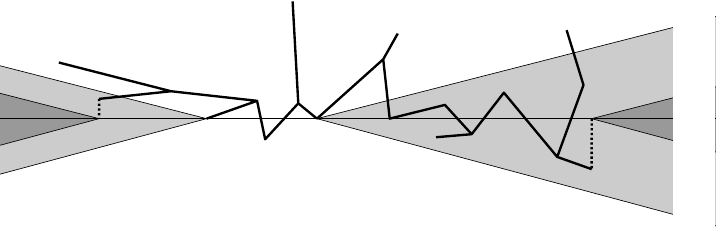}}%
    \put(0.36742121,0.27855413){\color[rgb]{0,0,0}\makebox(0,0)[lt]{\lineheight{1.25}\smash{\begin{tabular}[t]{l}$\xi$\end{tabular}}}}%
    \put(0.10341117,0.19682507){\color[rgb]{0,0,0}\makebox(0,0)[lt]{\lineheight{1.25}\smash{\begin{tabular}[t]{l}$x$\end{tabular}}}}%
    \put(0.09572154,0.15152668){\color[rgb]{0,0,0}\makebox(0,0)[lt]{\lineheight{1.25}\smash{\begin{tabular}[t]{l}$p(x)$\end{tabular}}}}%
    \put(0.81345383,0.19760456){\color[rgb]{0,0,0}\makebox(0,0)[lt]{\lineheight{1.25}\smash{\begin{tabular}[t]{l}$p(y)$\end{tabular}}}}%
    \put(0.83825138,0.09935136){\color[rgb]{0,0,0}\makebox(0,0)[lt]{\lineheight{1.25}\smash{\begin{tabular}[t]{l}$y$\end{tabular}}}}%
    \put(0.27978053,0.14511362){\color[rgb]{0,0,0}\makebox(0,0)[lt]{\lineheight{1.25}\smash{\begin{tabular}[t]{l}$a \ell \beta e$\end{tabular}}}}%
    \put(0.95398664,0.17031412){\color[rgb]{0,0,0}\makebox(0,0)[lt]{\lineheight{1.25}\smash{\begin{tabular}[t]{l}$\bbR e$\end{tabular}}}}%
    \put(0.43160796,0.1466726){\color[rgb]{0,0,0}\makebox(0,0)[lt]{\lineheight{1.25}\smash{\begin{tabular}[t]{l}$b \ell \beta e$\end{tabular}}}}%
  \end{picture}%
\endgroup%

	\caption{Construction of the animal $\xi'$ defined by~\eqref{eqn : LLN/directed/xi'}. The lightly shaded regions are $\Cone{\delta}{a \ell \beta e }{-e}$ and $\Cone{\delta}{b \ell \beta e}{e}$. The shaded regions are $\Cone{\delta}{p(x)}{- e}$ and $\Cone{\delta}{p(y)}{e}$. The animal $\xi'$ is the concatenation of $\xi$ (thick, solid lines), with the paths $(x, p(x))$ and $(y, p(y))$ (thick, dotted lines).}
	\label{fig : LLN/cvg_free/construction}
\end{figure}
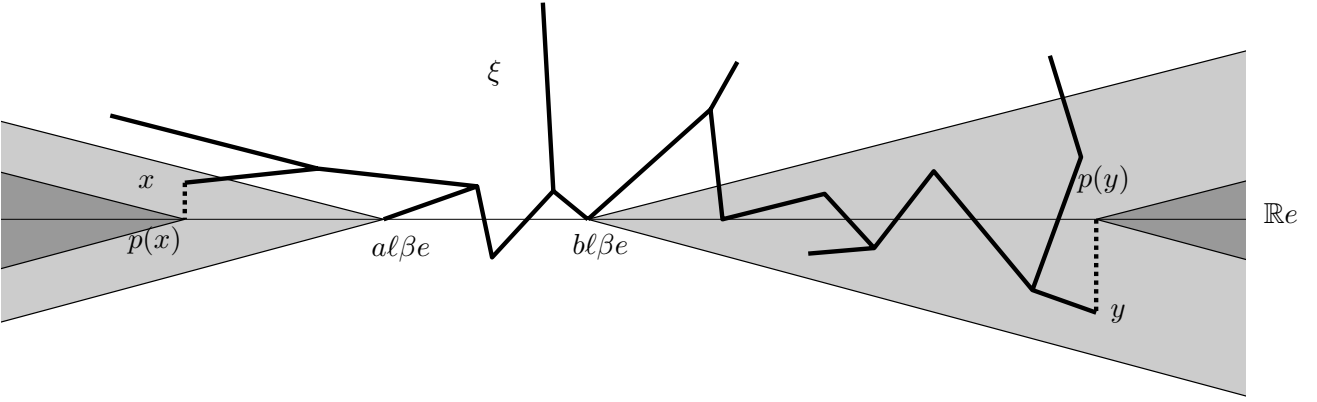%
thus
\begin{equation}
	\label{eqn : LLN/directed/animals/UB_MassXi_1}
	\Mass{\xi} \le \Mass{\xi'} \le \MassADR{\delta}{p(x)}{p(y)}{(b-a)\ell (1+2\Cr{cst : cone_bound}\sqrt\delta )}.
\end{equation}

By the definition of $\cError(\ell)$,~\eqref{eqn : LLN/directed/animals/projection2} and~\eqref{eqn : LLN/directed/animals/projection4},
\begin{align}
	\Mass{\xi} &\le (b-a)\ell\p{1+2\Cr{cst : cone_bound}\sqrt\delta }\cro{\LimMassA\p{\frac{p(x)-p(y)}{(b-a)\ell\p{1+2\Cr{cst : cone_bound}\sqrt\delta }}}  + \cError(\ell)}.\nonumber
	\intertext{Since $\hat\beta\mapsto\LimMassA(\hat\beta e)$ is nonincreasing on $\intervallefo01$ and $\norme{p(x) - p(y)} \ge (b-a)\ell \beta$,}
	\Mass{\xi}%
		&\le (b-a)\ell\p{1+2\Cr{cst : cone_bound}\sqrt\delta }\cro{\LimMassA\p{\frac{\beta e}{1+2\Cr{cst : cone_bound}\sqrt\delta}}  + \cError(\ell)}.\nonumber
	\intertext{Taking the supremum in $\xi$ and applying~\eqref{eqn : LLN/directed/animals/X}, we deduce that almost surely,}
	\limsup_{\ell \to \infty}\frac{ \MassADF{a \ell \beta e}{ b \ell \beta e}{(b-a)\ell} }{(b-a)\ell}%
		&\le \p{1+2\Cr{cst : cone_bound}\sqrt\delta }\LimMassA\p{\frac{\beta e}{1+2\Cr{cst : cone_bound}\sqrt\delta}}.
\end{align}
Since $\LimMassA$ is continuous, letting $\delta\to 0$ yields~\eqref{eqn : LLN/directed/UB_generic}. 

\emph{Case 2: Assume that }$\AUXGeneric = \AUXPath$.
For $h>0$, we say that a path $x\Path\gamma y$ is a $h$-\emph{cylinder path} if \begin{equation}
	x \in \argmin \set{ \ps{z}{e} }{z\in \gamma, \norme{z- p(z)} \le h}%
	\text{ and }%
	y \in \argmax \set{ \ps{z}{e}  }{z\in \gamma, \norme{z- p(z)} \le h},
\end{equation}
or vice versa. Note that in particular, this implies that $\norme{x- p(x)} \le h$ and $\norme{y- p(y)} \le h$. It is a variant of the notion introduced by Martin above (7.4) in \cite{Mar02}. Lemma~\ref{lem : LLN/directed/paths/decomposition}, proven at the end of the section, is analogous to Lemma~7 there.
\begin{Lemma}
	\label{lem : LLN/directed/paths/decomposition}
	Let $\ell>0$, $0<\delta < 1/4$ and a path $x\Path\gamma y$ of length at most $(b-a)\ell$ whose endpoints lie on $\R e$. Consider the path
	\begin{equation}
		\gamma' \dpe (x- \delta \ell e, x) \concat \gamma \concat (y, y+ \delta\ell e).	
	\end{equation} 
	Then there exist $r \le \frac{2(b-a)}{\delta}+3$ and a sequence $(\gamma_i')_{1\le i \le r}$ of $\delta^2\ell$-cylinder paths such that 
	\begin{equation}
		\gamma' = \gamma_1' \concat \dots \concat \gamma_r'.
	\end{equation}
\end{Lemma}
Lemma~\ref{lem : LLN/directed/cone_bound} implies that for small enough $0< \delta < 1/4$, for all $z\in \Cone{\delta^5}{0}{e} \cap \clball{0,(b-a +\delta^2)\ell}$,
\begin{equation}
	\label{eqn : LLN/directed/paths/projection}
	\norme{z-p(z)}\le \delta^2 \ell.	
\end{equation}
Fix $0< \delta < 1/4$ satisfying this property. We claim that for all $\delta^2 \ell$-cylinder paths $\gamma$ with endpoints $x$ and $y$ and length at most $(b-a)\ell$,
\begin{equation}
	\label{eqn : LLN/directed/paths/cylinder_path_subset_antidiamond}
	\gamma \subseteq \AntiDiamant{\delta^5}{p(x)}{p(y)}.
\end{equation}
Indeed let $\gamma$ be such a path. Then $\gamma \subseteq \clball{p(x) ,(b-a +\delta^2)\ell}$ therefore for all $z \in \gamma \cap \Cone{\delta^5}{p(x)}{-e}$, by~\eqref{eqn : LLN/directed/paths/projection} we have
\begin{equation*}
	\norme{z - p(z)} \le \delta^2 \ell,
\end{equation*}
thus $\ps z e \ge \ps x e$, which gives $z= p(x)$. Similarly, if $z \in \gamma \cap \Cone{\delta^5}{p(y)}{e}$, then $z=p(y)$, hence~\eqref{eqn : LLN/directed/paths/cylinder_path_subset_antidiamond}.
Proposition~\ref{prop : LLN/antidiamond} implies that
\stepcounter{error}
\begin{equation}
	\label{eqn : LLN/directed/paths/X}
	\begin{split}
	\cError(\ell)\dpe\sup &\Biggl\{ \p{\frac{\MassPDR{\delta^5}{x}{y}{\ell'} }{\ell'} - \LimMassP\p{\frac{x-y }{\ell'}  }}^+ \\%
		&\quad \Biggl| \begin{array}{c}%
			x,y\in \intervalleff{-\p{ 3(\module a \vee \module b)+2\delta^2 + \delta} \ell e }{ \p{ 3(\module a \vee \module b)+2\delta^2 + \delta} \ell e } \\%
			\norme{x-y}< \ell' \\%
			2\delta^2 \ell \le \ell' \le \p{b-a +2\delta^2 + \delta}\ell%
		\end{array} \Biggr \}%
		\xrightarrow[\ell \to \infty]{\text{a.s.}} 0.
	\end{split}
\end{equation}
Let $\ell>0$ and $\gamma \in \SetPDF{a \ell \beta e}{b \ell \beta e}{(b-a)\ell}$. With the notations of Lemma~\ref{lem : LLN/directed/paths/decomposition}, for all $i\in\intint1r$, we denote by $x_{i-1}'$ and $x_i'$ the endpoints of $\gamma_i'$. For all $i\in \intint1r$, we consider the path
\begin{equation}
	\gamma_i'' \dpe (p(x_{i-1}'), x_{i-1}') \concat \gamma_i' \concat (x_i', p(x_i')).
\end{equation}
Then for all $i\in\intint1r$, by~\eqref{eqn : LLN/directed/paths/cylinder_path_subset_antidiamond},
\begin{equation}
	\gamma_i''\in \SetPDR{\delta^5}{p(x_{i-1}')}{p(x_i')}{  \norme{\gamma_i'} + 2\delta^2\ell }.
\end{equation}
Moreover, reasoning like for~\eqref{eqn : LLN/directed/animals/projection2} and~\eqref{eqn : LLN/directed/animals/projection4}, we get that for all $i\in \intint0r$,
\begin{equation*}
	p(x_i') \in \intervalleff{-\p{ 3(\module a \vee \module b)+2\delta^2 + \delta} \ell e }{ \p{ 3(\module a \vee \module b)+2\delta^2 + \delta} \ell e }.
\end{equation*}
Consequently, by definition of $\cError(\ell)$, for all $i\in\intint1r$,
\begin{align}
	\Mass{\gamma_i''} &\le \p{\norme{\gamma_i'} + 2\delta^2\ell }\cro{ \LimMassP\p{ \frac{p(x_{i-1}')-p(x_i')}{\norme{\gamma_i'} + 2\delta^2\ell} } + \cError(\ell) }.\nonumber
	\intertext{Summing over $i$ and applying the concavity of $\LimMassP$, we get }
	\Mass{\gamma}%
	&\le \sum_{i=1}^r\Mass{\gamma_i''}\eol
	&\le \sum_{i=1}^r\p{\norme{\gamma_i'} + 2\delta^2\ell }\cro{ \LimMassP\p{ \frac{p(x_{i-1}')-p(x_i')}{\norme{\gamma_i'} + 2\delta^2\ell} } + \cError(\ell) }\eol
	&\le \sum_{i=1}^r \p{\norme{\gamma_i'} + 2\delta^2\ell }\LimMassP\p{ \frac{p(x_{i-1}')-p(x_i')}{\norme{\gamma_i'} + 2\delta^2\ell} } + \sum_{i=1}^r \p{ \norme{\gamma_i'} + 2\delta^2 \ell } \cError(\ell)\eol
	&\le \p{ \sum_{i=1}^r \p{ \norme{\gamma_i'} + 2\delta^2 \ell } } \LimMassP\p{ \frac{ \sum_{i=1}^r \p{p(x_{i-1}')-p(x_i')} }{ \sum_{i=1}^r \p{\norme{\gamma_i'} + 2\delta^2\ell }   } } +  \sum_{i=1}^r \p{ \norme{\gamma_i'} + 2\delta^2 \ell  } \cError(\ell) \eol
	&= \p{ \norme{\gamma'}  +2r\delta^2 \ell}\LimMassP\p{ \frac{ x_0' - x_r' }{ \norme{\gamma'}  +2r\delta^2 \ell } } +\p{ \norme{\gamma'} + 2r\delta^2 \ell  } \cError(\ell).
\end{align}
Applying $r\le \frac{2(b-a)}{\delta} + 3$, $\norme{\gamma'} \le (b-a+2\delta)\ell$ and the monotonicity of $\beta \mapsto \LimMassP(\beta e)$ on $\intervallefo01$ yields
\begin{align}
	\frac{\Mass\gamma}{(b-a)\ell}%
		&\le \p{ 1 + \frac{2\delta + 6 \delta^2}{b-a} + 4\delta  }\LimMassP\p{ \frac{ \beta e }{ 1 + \frac{2\delta + 6 \delta^2}{b-a} + 4\delta } } + \p{ 1 + \frac{2\delta + 6 \delta^2}{b-a} + 4\delta   } \cError(\ell).\nonumber
	\intertext{Taking the supremum in $\gamma$ gives}
	\frac{\MassPDF{a \ell \beta e}{b \ell \beta e}{(b-a)\ell} }{(b-a)\ell}%
		&\le \p{ 1 + \frac{2\delta + 6 \delta^2}{b-a} + 4\delta  }\LimMassP\p{ \frac{ \beta e }{ 1 + \frac{2\delta + 6 \delta^2}{b-a} + 4\delta  } } + \p{ 1 + \frac{2\delta + 6 \delta^2}{b-a} + 4\delta   } \cError(\ell).\nonumber
	\intertext{Considering $\ell\to \infty$ we deduce that almost surely,}
	\label{eqn : LLN/directed/paths/Final_UB}
	\limsup_{\ell \to \infty} \frac{\MassPDF{a \ell \beta e}{b \ell \beta e}{(b-a)\ell} }{(b-a)\ell}%
		&\le \p{ 1 + \frac{2\delta + 6 \delta^2}{b-a} + 4\delta }\LimMassP\p{ \frac{ \beta e }{ 1 + \frac{2\delta + 6 \delta^2}{b-a} + 4\delta } }.
\end{align}
Since $\LimMassP$ is continuous, letting $\delta\to 0$ gives~\eqref{eqn : LLN/directed/UB_generic} and concludes the proof of Proposition~\ref{prop : LLN/directed/cvg}.
\end{proof}

\begin{proof}[Proof of Lemma~\ref{lem : LLN/directed/paths/decomposition}]
Denote $\gamma = (x=x_0, \dots, x_I=y)$. Let us consider
\begin{align}
	j(0) &\dpe \min \argmin \set{ \ps{x_j}{e} }{ j\in\intint0I, \norme{x_j - p(x_j)} \le \delta^2\ell}\\%
	\text{ and }i(0) &\dpe \max \argmax \set{ \ps{x_i}{e} }{ i\in\intint0I, \norme{x_i - p(x_i)} \le \delta^2\ell}.
\end{align}
We only treat the case where $j(0) \le i(0)$, the other case is similar. We then define (see Figure~\ref{fig : Decomposition}) a increasing then constant sequence by the recursion
\begin{equation}
	\label{eqn : LLN/directed/paths/decomposition/i}
	i(n) \dpe %
	\begin{cases}
		I &\text{if $i(n-1)=I$,}\\
		\max \argmin \set{ \ps{x_i}{e} }{ i\in\intint{i(n-1)+1}I, \norme{x_i - p(x_i)} \le \delta^2\ell} \qquad &\text{if $n$ is odd,}\\
		\max \argmax \set{ \ps{x_i}{e} }{ i\in\intint{i(n-1)+1}I, \norme{x_i - p(x_i)} \le \delta^2\ell} \qquad &\text{if $n$ is even.}
	\end{cases}
\end{equation}
Similarly, we define a decreasing then stationary sequence by the recursion
\begin{equation}
	\label{eqn : LLN/directed/paths/decomposition/j}
	j(n) \dpe %
	\begin{cases}
		0 &\text{if $j(n-1)=0$,}\\
		\min \argmax \set{ \ps{x_j}{e} }{ j\in\intint0{j(n-1)-1}, \norme{x_j - p(x_j)} \le \delta^2\ell} \qquad &\text{if $n$ is odd,}\\
		\min \argmin \set{ \ps{x_j}{e} }{ j\in\intint0{j(n-1)-1}, \norme{x_j - p(x_j)} \le \delta^2\ell} \qquad &\text{if $n$ is even.}
	\end{cases}
\end{equation}%
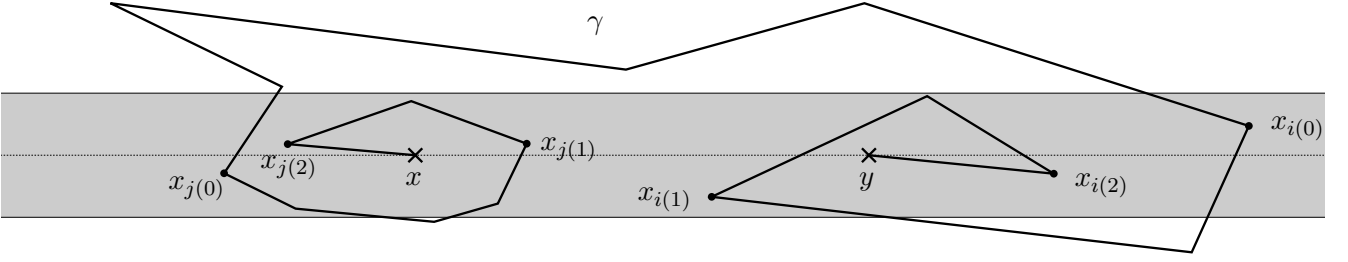
\begin{figure}
	\center
	\def\svgwidth{\textwidth}
	\begingroup%
  \makeatletter%
  \providecommand\color[2][]{%
    \errmessage{(Inkscape) Color is used for the text in Inkscape, but the package 'color.sty' is not loaded}%
    \renewcommand\color[2][]{}%
  }%
  \providecommand\transparent[1]{%
    \errmessage{(Inkscape) Transparency is used (non-zero) for the text in Inkscape, but the package 'transparent.sty' is not loaded}%
    \renewcommand\transparent[1]{}%
  }%
  \providecommand\rotatebox[2]{#2}%
  \newcommand*\fsize{\dimexpr\f@size pt\relax}%
  \newcommand*\lineheight[1]{\fontsize{\fsize}{#1\fsize}\selectfont}%
  \ifx\svgwidth\undefined%
    \setlength{\unitlength}{485.75275571bp}%
    \ifx\svgscale\undefined%
      \relax%
    \else%
      \setlength{\unitlength}{\unitlength * \real{\svgscale}}%
    \fi%
  \else%
    \setlength{\unitlength}{\svgwidth}%
  \fi%
  \global\let\svgwidth\undefined%
  \global\let\svgscale\undefined%
  \makeatother%
  \begin{picture}(1,0.21456276)%
    \lineheight{1}%
    \setlength\tabcolsep{0pt}%
    \put(0,0){\includegraphics[width=\unitlength,page=1]{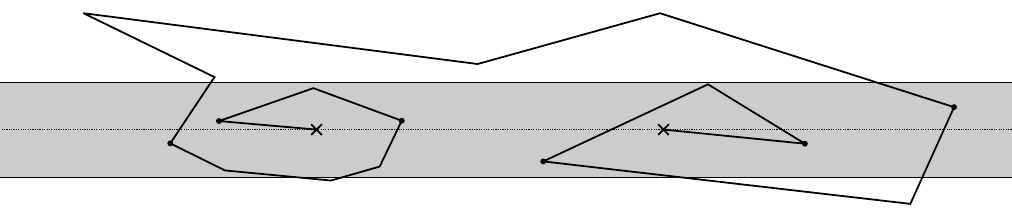}}%
    \put(0.30549617,0.06431301){\color[rgb]{0,0,0}\makebox(0,0)[lt]{\lineheight{1.25}\smash{\begin{tabular}[t]{l}$x$\end{tabular}}}}%
    \put(0.64831182,0.06431301){\color[rgb]{0,0,0}\makebox(0,0)[lt]{\lineheight{1.25}\smash{\begin{tabular}[t]{l}$y$\end{tabular}}}}%
    \put(0.12619971,0.0609189){\color[rgb]{0,0,0}\makebox(0,0)[lt]{\lineheight{1.25}\smash{\begin{tabular}[t]{l}$x_{j(0)}$\end{tabular}}}}%
    \put(0.40739987,0.09016425){\color[rgb]{0,0,0}\makebox(0,0)[lt]{\lineheight{1.25}\smash{\begin{tabular}[t]{l}$x_{j(1)}$\end{tabular}}}}%
    \put(0.19590594,0.07704029){\color[rgb]{0,0,0}\makebox(0,0)[lt]{\lineheight{1.25}\smash{\begin{tabular}[t]{l}$x_{j(2)}$\end{tabular}}}}%
    \put(0.81115465,0.06290664){\color[rgb]{0,0,0}\makebox(0,0)[lt]{\lineheight{1.25}\smash{\begin{tabular}[t]{l}$x_{i(2)}$\end{tabular}}}}%
    \put(0.48089472,0.05287067){\color[rgb]{0,0,0}\makebox(0,0)[lt]{\lineheight{1.25}\smash{\begin{tabular}[t]{l}$x_{i(1)}$\end{tabular}}}}%
    \put(0.95926908,0.10648538){\color[rgb]{0,0,0}\makebox(0,0)[lt]{\lineheight{1.25}\smash{\begin{tabular}[t]{l}$x_{i(0)}$\end{tabular}}}}%
    \put(0.44274395,0.18179428){\color[rgb]{0,0,0}\makebox(0,0)[lt]{\lineheight{1.25}\smash{\begin{tabular}[t]{l}$\gamma$\end{tabular}}}}%
  \end{picture}%
\endgroup%
	\caption{Illustration of $i(n)$ and $j(n)$ defined by~\eqref{eqn : LLN/directed/paths/decomposition/i} and~\eqref{eqn : LLN/directed/paths/decomposition/j}. The shaded strip represents the set of points $z\in \R^d$ such that $\norme{z-p(z)}\le \delta^2\ell$.}
	\label{fig : Decomposition}
\end{figure}%
Note that the sequences $\p{ \norme{ p \p{ x_{i(n)} } - p \p{ x_{i(n+1)} } } }$ and $\p{ \norme{ p(x_{j(n)}) - p(x_{j(n+1)}) } }$ are nonincreasing and that every subpath of the form $\gamma_0\dpe \bigl( x_{j(0)}, \dots, x_{i(0)} \bigr) $, $ \gamma_n^+ \dpe \bigl( x_{i(n)},  x_{i(n)+1}, \dots, x_{i(n+1)} \bigr) $ or $\gamma_n^- \dpe \bigl( x_{j(n+1)},  x_{j(n+1)+1}, \dots, x_{j(n)} \bigr)$ a is $\delta^2\ell$-cylinder path. Define
\begin{align}
	N^+ &\dpe \min\set{n\ge 0 }{ \text{$n$ is odd and }\norme{ p(x_{i(n)}) - p(x_{i(n+1)}) } < \delta \ell }
	\intertext{and}
	N^- &\dpe \min\set{n\ge 0 }{ \text{$n$ is odd and }\norme{ p(x_{j(n)}) - p(x_{j(n+1)}) } < \delta \ell }.
\end{align}
Consider the paths
\begin{align}
	\hat \gamma_{N^+}^+ &\dpe \p{ x_{i(N^+)}, \dots, x_I = y} \concat \p{y, y+\delta \ell e} 
	\intertext{and}
	\hat \gamma_{N^-}^- &\dpe \p{x - \delta\ell e, x} \concat \p{ x = x_0 , \dots, x_{j(N^-)} }.
\end{align}
We claim that they are $\delta^2 \ell$-cylinder paths. Indeed, $\norme{x_{i(N^+)}- p\p{x_{i(N^+)}} } \le \delta^2 \ell$ and $\norme{y+\delta \ell e - p(y+\delta \ell e)}=0 \le \delta^2 \ell$. Moreover, by definition of $i(N^+)$,
\begin{equation}
	\label{eqn : LLN/directed/paths/decomposition/last_bit_cylinder1}
	\ps{x_{i(N^+)} }{e} = \min \set{\ps{z}{e} }{ z \in \hat \gamma_{N^+}^+, \norme{z - p(z)}\le \delta^2\ell }.
\end{equation}
By definition of $i(N^+ +1)$,
\begin{equation*}
	\ps{x_{i(N^++1)} }{e} = \max \set{\ps{x_i}{e} }{ i \in \intint{i(N^+)}{I}, \norme{x_i - p(x_i)}\le \delta^2\ell }.
\end{equation*}
Besides, $\norme{ p(x_{i(N^+)}) - p(x_{i(N^++1)}) }< \delta \ell$ and $y-p(x_{i(N^+)}) \in \R^+ e$ therefore
\begin{equation}
	\label{eqn : LLN/directed/paths/decomposition/last_bit_cylinder2}
	\ps{y+\ell \delta e }{e} = \max \set{\ps{z}{e} }{ z \in \hat \gamma_{N^+}^+, \norme{z - p(z)}\le \delta^2\ell }.
\end{equation}
Consequently, $\hat \gamma_{N^+}^+$ is a $\delta^2\ell$-cylinder path. The same goes for $\hat \gamma_{N^-}^-$ Our candidate for $(\gamma'_i)_{1\le i \le r}$ is the sequence $\p{\hat \gamma_{N^-}^-, \gamma_{N^--1}^-, \dots, \gamma_0^-, \gamma_0, \gamma_0^+, \dots, \gamma_{N^+-1}^+, \hat \gamma_{N^+}^+  }$.

For all $n\in \intint{0}{N^+ -1}$, by triangle inequality,
\begin{align*}
	\norme{ x_{i(n)} - x_{i(n+1)} } %
		&\ge \norme{ p(x_{i(n)}) - p(x_{i(n+1)}) } - \norme{ p(x_{i(n)}) - x_{i(n)} } - \norme{ p(x_{i(n+1)}) - x_{i(n+1)} }\\
		&\ge \delta \ell - 2\delta^2 \ell \ge \frac{\delta\ell}{2}.
	\intertext{Likewise, for all $n\in \intint{0}{N^--1}$,}
	\norme{ x_{j(n)} - x_{j(n+1)} }%
		&\ge \frac{\delta\ell}{2}.
\end{align*}
Consequently,
\begin{equation*}
	N^+ + N^- \le \frac{(b-a)\ell}{\frac{\delta\ell}{2}}  = \frac{2(b-a)}{\delta},
\end{equation*}
thus $r \le \frac{2(b-a)}{\delta}+3$.
\end{proof}

\subsection{Proof of Theorem~\ref{thm : intro/main/MAIN}}

Let $K$ be a compact subset of $\cX$. Consider
\begin{equation}
	0< \delta < 1 - \sup_{(x,y,\ell)\in K} \frac{\norme{x-y} }{\ell}. 
\end{equation}
Let $0< \eps \le \frac{\delta^3}9$. Consider an integer $M\ge \eps^{-1}$ and a modulus of continuity $\omega$ of $\LimMassG$ on $\clball{0,1-\delta+\eps}$. By compactness there exists a finite family $\p{(x_n, y_n)}_{1\le n \le N}$ of pairs of colinear points in $\clball{0,1}$, such that for all pair $(x,y)$ of colinear points in $\clball{0,1}$,
\begin{equation*}
	\norme{x-x_n} + \norme{y-y_n} \le \eps
\end{equation*}
for some $1\le n \le N$. For all $L>0$, define
\stepcounter{error}%
\begin{equation}
	\cError(L) \dpe%
	\max\set%
		{
		\module{ \frac{ \MassGDF{L x_n}{L y_n}{\frac{mL}{M} }}L - \frac mM\LimMassG\p{ \frac{M(y_n - x_n)}{m} } }
		}
		{
		\begin{array}{c}
			 n \in\intint1N,\\
			 m \in\intint1{2M},\\
			\norme{x_n - y_n} < \frac{m}{M}
		\end{array}
		}.
\end{equation}
Let $L>0$ and $(x,y,\ell) \in K$. We claim that
\begin{equation}
	\label{LLN/conclusion/bound_neg_part}
	\cro{\frac{\MassGDF{Lx}{Ly}{L\ell} }{L} - \ell\LimMassG\p{ \frac{y-x}{\ell} }  }^- \le \delta\LimMassG(0) +  \omega\p{\frac{3\eps}{\delta(\delta- 2\eps)}} + \cError(L) .
\end{equation}
Indeed the inequality is clear if $\ell < \delta$. Assume that $\ell \ge \delta$. There exists $1\le n \le N$ such that
\begin{equation}
	\label{LLN/conclusion/bound_neg_part/approx_xy}
	\norme{x-x_n} + \norme{y- y_n} \le \eps.
\end{equation}
Let $m\dpe \floor{ M(\ell - \eps) }$. Note that 
\begin{equation}
	\label{LLN/conclusion/bound_neg_part/LB_k}
	m \ge  M(\ell - \eps) -1 = M\p{\ell - \eps -\frac1M} \ge M(\ell-2\eps).
\end{equation}
We first estimate $\frac{M(y_n - x_n)}{m}$. By two iterations of the triangle inequality we have
\begin{align}
	\norme{\frac{M(y_n - x_n)}{m} - \frac{y-x}{\ell} }%
		&\le \norme{\frac{M(y_n - x_n)}{m} - \frac{M(y-x)}{m} }  + \norme{\frac{M(y-x)}{m} - \frac{y-x}{\ell} }\eol
		&\le \frac{M}{m}\cro{ \norme{y_n -y} + \norme{x_n - x} } + \module{ \frac{M}{m} - \frac1\ell } \cdot \norme{y-x}\eol
		&= \frac{M}{m}\cro{ \norme{y_n - y} + \norme{x_n -  x} } + \frac{M\ell - m}{\ell m}\cdot\norme{y-x}.\nonumber
	\intertext{Applying~\eqref{LLN/conclusion/bound_neg_part/approx_xy},~\eqref{LLN/conclusion/bound_neg_part/LB_k} and $\norme{y-x} \le 1$ yields}
	\norme{\frac{M(y_n - x_n)}{m} - \frac{y-x}{\ell} }%
		&\le \frac{\eps}{ \ell - 2 \eps} + \frac{ 2\eps  }{\ell(\ell - 2\eps)}.\nonumber
	\intertext{Since $\ell \ge \delta$, we get}
	\label{LLN/conclusion/bound_neg_part/erreur_format}
	\norme{\frac{M(y_n - x_n)}{m} - \frac{y-x}{\ell} }%
		&\le \frac{\eps}{ \delta - 2 \eps} + \frac{ 2\eps }{\delta(\delta- 2\eps)} \le \frac{3\eps}{\delta(\delta- 2\eps)}.
\end{align}
In particular, $\frac mM > \norme{x_n - y_n}$. Let $\xi \in \SetGDF{Lx_n}{Ly_n}{ \frac{mL}{M} }$. Then the animal
\begin{equation*}
	\xi' \dpe (Lx,Lx_n) \concat \xi \concat (L y_n, Ly)
\end{equation*}
satisfies $\norme{\xi'} \le \frac{mL}{M} + \eps L \le L\ell$, thus $\xi' \in \SetGDF{Lx}{Ly}{L\ell}$. Consequently,
\begin{equation}
	\Mass\xi \le \Mass{\xi'} \le \MassGDF{Lx}{Ly}{L\ell}.
\end{equation}
Taking the supremum with respect to $\xi$ leads to
\begin{equation}
	\label{LLN/conclusion/bound_neg_part/pre_LB1}
	\MassGDF{Lx}{Ly}{L\ell} \ge \MassGDF{Lx_n}{Ly_n}{ \frac{mL}{M} }.
\end{equation}
By definition of $\cError(L)$, we get
\begin{equation}
	\label{LLN/conclusion/bound_neg_part/LB1}
	\frac{ \MassGDF{Lx}{Ly}{L\ell}}L \ge \frac{m}{M} \LimMassG\p{\frac{M(y_n - x_n)}{m}}   - \cError(L).
\end{equation}
Plugging~\eqref{LLN/conclusion/bound_neg_part/LB_k} and~\eqref{LLN/conclusion/bound_neg_part/erreur_format} into~\eqref{LLN/conclusion/bound_neg_part/LB1} gives
\begin{align}
	\frac{ \MassGDF{Lx}{Ly}{L\ell}  }L%
		&\ge \p{\ell - 2\eps }\cdot \cro{\LimMassG\p{\frac{y-x}{\ell}} - \omega\p{\frac{3\eps}{\delta(\delta- 2\eps)}} } - \cError(L) \eol
		&\ge \ell \LimMassG\p{\frac{y-x}{\ell}} - \ell  \omega\p{\frac{3\eps}{\delta(\delta- 2\eps)} } -2\eps \LimMassG(0) - \cError(L) ,\nonumber
\end{align}
thus
\begin{equation}
	\label{LLN/conclusion/bound_neg_part2}
	\cro{ \frac{ \MassGDF{Lx}{Ly}{L\ell}  }L - \ell \LimMassG\p{\frac{y-x}{\ell}} }^- \le \omega\p{\frac{3\eps}{\delta(\delta- 2\eps)}} +2\eps \LimMassG(0) + \cError(L),
\end{equation}
which gives~\eqref{LLN/conclusion/bound_neg_part}.

In particular, by Proposition~\ref{prop : LLN/directed/cvg},
\begin{equation}
	\label{LLN/conclusion/bound_neg_part/LB}
\sup\set{\cro{ \frac{ \MassGDF{Lx}{Ly}{L\ell} }{L} - \ell\LimMassG\p{\frac{y-x}{\ell} } }^-}%
	{ (x,y,\ell)\in K}
		\xrightarrow[L \to\infty]{\text{a.s.}} 0.
\end{equation}
Similarly, letting $m' \dpe \ceil{M(\ell + \eps)} \le M(\ell + 2\eps) \le 2M$ leads to
\begin{equation}
 	\label{LLN/conclusion/bound_neg_part/pre_UB1}
 	\MassGDF{Lx}{Ly}{L\ell} \le \MassGDF{Lx_n}{Ly_n}{ \frac{m'L}{M} },
\end{equation} 
thus the analogue of~\eqref{LLN/conclusion/bound_neg_part/LB} with $\cro{ \frac{ \MassGDF{Lx}{Ly}{L\ell} }{L} - \LimMassG\p{\frac{y-x}{\ell} } }^+$, thus the almost sure part of~\eqref{eqn : intro/main/MAIN/cvg}. To prove the $\rL^1$ convergence, let $\alpha >0$ be such that $ \clball{0,2} \subseteq \Diamant{1/2}{-\alpha\base 1}{\alpha\base 1}$. For all $L>0$ and $(x,y,\ell)\in K$, a straightforward concatenation argument yields
\begin{equation}
	\label{eqn : LLN/conclusion/domination}
	\frac{ \MassGDF{Lx}{Ly}{\ell L}}{L} \le \frac{ \MassADC{1/2}{-\alpha L \base 1}{\alpha L \base 1}{(1+4\alpha)L} }{L}.
\end{equation}
Lemma~\ref{lem : LLN/simple/cvg} and the domination~\eqref{eqn : LLN/conclusion/domination} imply the $\rL^1$ part of \eqref{eqn : intro/main/MAIN/cvg}.

In the case $\AUXGeneric = \AUXAnimal$, \eqref{eqn : intro/main/MAIN/cvg_undirected} is a direct consequence of~\eqref{eqn : intro/main/MAIN/cvg}, as for all $L>0$, $\MassAUF L = \MassADF00L$. Let us turn to the case $\AUXGeneric = \AUXPath$. Let $L>0$. We have
\begin{align}
	\MassPUF L &\ge \MassPDF00L, \nonumber \intertext{therefore, almost surely } \liminf_{L \to \infty} \frac{\MassPUF L}{L} &\ge \LimMassP(0).\label{eqn : LLN/conclusion/LB_lim_PUF}
\end{align}
Besides, for all $u\in \clball{0,1}$,
\begin{align}
	\MassPDF{0}{L u}{L} %
		&\le \MassPDF{0}{\frac{L}{1-\delta} \cdot (1-\delta)u}{\frac{L}{1-\delta}},\nonumber
	\intertext{thus}
	\frac{\MassPUF L}L%
		&= \frac{1}{L}\sup_{u\in \clball{0,1} } \MassPDF{0}{L u}{L}\nonumber\\
		&\le \frac{1}{1-\delta}\sup_{v\in \clball{0,1-\delta}} \frac{ \MassPDF{0}{\frac{L}{1-\delta} \cdot v}{\frac{L}{1-\delta}} }{\frac{L}{1-\delta}}. \label{eqn : LLN/conclusion/UB_PUF}
\end{align}
By~\eqref{eqn : intro/main/MAIN/cvg} and $\LimMassP(0) = \sup_{v \in \ball{0,1}} \LimMassP(v)$, almost surely,
\begin{equation}
	\limsup_{L \to \infty} \frac{\MassPUF L}L%
		\le \frac{\LimMassP(0)}{1-\delta}\label{eqn : LLN/conclusion/UB_lim_PUF}
\end{equation}
The almost sure part of \eqref{eqn : intro/main/MAIN/cvg_undirected} is a consequence of \eqref{eqn : LLN/conclusion/LB_lim_PUF} and \eqref{eqn : LLN/conclusion/UB_lim_PUF}. The domination \eqref{eqn : LLN/conclusion/UB_PUF} gives the convergence in $\rL^1$. \qed
\section{A sufficient condition and a necessary condition for integrability}
\label{sec : suff_condition}

Let $\pro$ be a simple marked stationary point process on $\R^d \times \intervalleoo0\infty$, with mean measure $\Leb \otimes \nu$ (see~\eqref{eqn : intro/framework/mm_as_a_product}).
\subsection{A sufficient condition for integrability: proof of Proposition~\ref{prop : intro/suff_condition/fmm}}
\label{subsec : suff_condition}

We rely on Lemmas~\ref{lem : suff_condition/fmm_pty_mapping},~\ref{lem : suff_condition/Dirac_case}, and~\ref{lem : suff_condition/projection} which respectively states that:
\begin{enumerate}
	\item The moment property~\eqref{eqn : intro/suff_condition/fmm} is stable by mapping and restriction.
	\item If $\pro$ has the form $\Phi \otimes \Dirac1$ (i.e.\ all the marks are equal to $1$) and satisfies the moment property (see~\eqref{eqn : intro/suff_condition/fmm}) with the constant $C$, then $\sup_{\ell \ge 1} \frac{\MassAUF \ell}{\ell}$ is bounded by a quantity the only depends on its mean measure, $C$ and $d$.
	\item The variable $\sup_{\ell \ge 1} \frac{\MassAUF \ell\cro\pro}{\ell}$ is bounded by the integral on $\intervalleoo0\infty$ of the corresponding quantity with $\pro^{(t)}$, defined by~\eqref{eqn : intro/main/truncated} (see~\eqref{eqn : suff_condition/projection}), which is of the form $\Phi \otimes \Dirac1$.
\end{enumerate}
Their proofs are postponed to the end of the section. By equivalence of the norms on $\R^d$ there exists a constant $\Cl{EN}>0$ such that for all animals $\xi$,
\begin{equation*}
	\frac{1}{\Cr{EN}}\norme[1]{\xi} \le \norme{\xi} \le \Cr{EN}\norme[1]{\xi}, 
\end{equation*}
thus the truth value of $\E{ \sup_{\ell \ge 1} \frac{\MassAUF \ell}{\ell} }<\infty$ does not depend on $\norme\cdot$. We may thus assume $\norme\cdot = \norme[1]\cdot$.

\begin{Lemma}
	\label{lem : suff_condition/fmm_pty_mapping}
	Let $\Phi$ be a point process on a regular space $\bbG$. Assume that $\Phi$ satisfies the moment property with the constant $C$.
	\begin{enumerate}[(i)]
		\item \label{item : suff_condition/fmm_pty_mapping/restriction} For all Borel subset $\bbG' \subseteq \bbG$, the point process $\restriction\Phi{\bbG'}$ satisfies the moment property with the constant $C$. 
		\item \label{item : suff_condition/fmm_pty_mapping/mapping} For all regular space $\bbG'$ and measurable map $f: \bbG \rightarrow \bbG'$, the point process $\Phi\circ f^{-1}$ satisfies the moment property with the constant $C$.
	\end{enumerate}
\end{Lemma}
\begin{Lemma}
	\label{lem : suff_condition/Dirac_case}
	Let $\Phi$ be a stationary point process on $\R^d$ with mean measure (see the definition~\eqref{eqn : intro/suff_condition/def_mm}) $\lambda \Leb$, where $\lambda < \infty$. Assume that $\Phi$ satisfies the moment property with the constant $C$. Then there exists $\Cr{cst : Dirac_bound}>0$, depending only on $d$ and $\norme\cdot$, such that
	\begin{equation}
		\label{eqn : suff_condition/Dirac_case}
		\E{ \sup_{\ell \ge 1}\frac{\MassAUF\ell\cro{\Phi \otimes \Dirac1} }{\ell} } \le \Cr{cst : Dirac_bound} \cdot (C\lambda)^{1/d}.
	\end{equation}
\end{Lemma}
\begin{Lemma}
	\label{lem : suff_condition/projection}
	Almost surely,
	\begin{equation}
		\label{eqn : suff_condition/projection}
		\sup_{\ell \ge 1} \frac{\MassAUF\ell\cro\pro}{\ell} \le \int_0^\infty \sup_{\ell \ge 1} \frac{\MassAUF\ell \cro{\pro^{(t)}} }{\ell} \d t.
	\end{equation}
\end{Lemma}
Assume the hypotheses of Proposition~\ref{prop : intro/suff_condition/fmm}. Let $t>0$. By the first part of Lemma~\ref{lem : suff_condition/fmm_pty_mapping}, the restriction of $\pro$ on $\R^d \times \intervallefo t\infty$ satisfies the moment property with the constant $C$. By the second part of this lemma, applied with the mapping
\begin{align*}
	f : \R^d \times \intervallefo t\infty &\longrightarrow \R^d \\ (x,s) &\longmapsto x,
\end{align*}
the process $\pro^{(t)}$ also satisfies the moment property with the constant $C$.

Besides, we claim that
\begin{equation}
	\label{eqn : suff_condition/intesity_prot}
	\MM[ \pro^{(t)} ] = \nu\p{ \intervallefo t\infty } \Leb.
\end{equation}
Indeed, let $B \subseteq \R^d$ be a Borel set. We have
\begin{align*}
	\MM[ \pro^{(t)} ](B) &= \E{ \int_{B} \pro^{(t)}(\d x)    }\\
		&= \E{ \int_{B \times \intervallefo t\infty } \pro(\d x, \d s)  }\\
		&= \MM[\pro]\p{ B \times \intervallefo t\infty }\\
		&= \Leb(B) \nu \p{\intervallefo t\infty},
\end{align*}
i.e.~\eqref{eqn : suff_condition/intesity_prot}.

Lemma~\ref{lem : suff_condition/Dirac_case} then yields
\begin{equation*}
	\E{ \sup_{\ell \ge 1}\frac{\MassAUF\ell\cro{\pro^{(t)} \otimes \Dirac1 } }{\ell} } \le \Cr{cst : Dirac_bound} C^{1/d} \nu\p{\intervallefo t\infty}^{1/d}.
\end{equation*}
Combining this inequality with~\eqref{eqn : suff_condition/projection} leads to~\eqref{eqn : intro/suff_condition/fmm/bound}. \qed
\begin{proof}[Proof of Lemma~\ref{lem : suff_condition/fmm_pty_mapping}]
\emph{Proof of~\eqref{item : suff_condition/fmm_pty_mapping/restriction}. } Note that for all $k\ge$, for all Borel sets $B_1, \dots, B_k \subseteq \bbG'$, 
\begin{equation*}
	\p{ \restriction\Phi{\bbG'} }^{(k)}\p{\prod_{i=1}^k B_i } = \Phi^{(k)}\p{\prod_{i=1}^k B_i }.
\end{equation*}

\emph{Proof of~\eqref{item : suff_condition/fmm_pty_mapping/mapping}. }
	Let $k\ge 1$ and $B_1,\dots,B_k$ be Borel subsets of $\bbG'$. Recall~\eqref{eqn : intro/main/decomposition_pro}. Then
	\begin{align}
		\p{ \Phi \circ f^{-1} }^{(k)} \p{ \prod_{i=1}^k B_i }%
			&= \sum_{ \substack{1\le n_1, \dots, n_k \le N \\ \text{pairwise distinct} } } \prod_{i=1}^k \ind{B_i}\p{f(x_{n_i})}\nonumber\\
			&= \sum_{ \substack{1\le n_1, \dots, n_k \le N \\ \text{pairwise distinct} } } \prod_{i=1}^k \ind{f^{-1}(B_i)}\p{x_{n_i}}\nonumber\\
			&= \Phi^{(k)}\p{ \prod_{i=1}^k f^{-1}(B_i) }.\nonumber
	\end{align}
	Taking the expectancy and applying~\eqref{eqn : intro/suff_condition/fmm} concludes the proof.
\end{proof}
\begin{proof}[Proof of Lemma~\ref{lem : suff_condition/Dirac_case}]
	First assume that $\lambda =1$. We follow Lemma~2.1 in Gouéré, Marchand (2008) \cite{Gou08}. Let
	\begin{equation}
		\label{eqn : suff_condition/Dirac_case/alpha0}
		\alpha_0 \dpe \p{ 2^{d+1}e C}^{1/d}.
	\end{equation}
	Given an integer $k\ge 1$ and $\alpha \ge \alpha_0$, let $\Pi(k,\alpha)$ denote the set of $k$-uples $(x_1,\dots, x_k) \in (\R^d)^k$ of pairwise distinct atoms of $\Phi$ such that
	\begin{equation*}
	  	\norme{(0, x_1, \dots, x_k)} = \sum_{i=1}^k \norme{x_i - x_{i-1} } \le \frac{k}{\alpha},
	\end{equation*}  
	with the convention $x_0=0$. We have: 
	\begin{align}
		\Pb{ \Pi(k, \alpha) \neq \emptyset}%
			&\le \E{ \#\Pi(k, \alpha) }\eol %
			&= \int_{(\R^d)^k} \ind{k \ge \alpha\sum_{i=1}^k \norme{x_i - x_{i-1} } } \FMM[\Phi]{k}(\d x_1, \dots, \d x_k)\eol
			&\le \int_{(\R^d)^k} \exp\p{k - \alpha\sum_{i=1}^k \norme{x_i - x_{i-1} } } \FMM[\Phi]{k}(\d x_1, \dots, \d x_k).\nonumber
		\intertext{Applying~\eqref{eqn : intro/suff_condition/fmm} gives}
		\Pb{ \Pi(k, \alpha) \neq \emptyset}%
			&\le C^k \int_{(\R^d)^k} \exp\p{k - \alpha\sum_{i=1}^k \norme{x_i - x_{i-1} } } \d x_1 \dots \d x_k\eol
			&= \p{C\int_{\R^d} \exp\p{1-\alpha \norme x} \d x }^k \eol
			&= \p{C\alpha^{-d}\cdot e2^d }^k .
	\end{align}
	By union bound,
	\begin{align}
		\Pb{ \sup_{\ell\ge 1} \frac{\MassPUF\ell\cro{ \Phi \otimes \Dirac 1  } }{\ell} \ge \alpha }%
			&\le \sum_{k\ge 1} \Pb{ \Pi(k, \alpha) \neq \emptyset}\eol
			&\le \sum_{k\ge 1} \p{C\alpha^{-d}\cdot e2^d  }^k. \eol
		\intertext{By definition of $\alpha_0$ (see~\eqref{eqn : suff_condition/Dirac_case/alpha0}), the series converges, and}
		\Pb{ \sup_{\ell \ge 1} \frac{\MassPUF\ell\cro{ \Phi \otimes \Dirac 1  }}{\ell} \ge \alpha }%
			&\le C\alpha^{-d}\cdot e2^d  \cdot \p{1- C\alpha^{-d}\cdot e2^d  }^{-1}\eol
			&\le 2^{d+1} e C \alpha^{-d}.\nonumber
	\end{align}
	Consequently,
	\begin{align}
		\E{ \sup_{\ell \ge 1} \frac{\MassPUF\ell\cro{ \Phi \otimes \Dirac 1  }}{\ell} }%
			&= \int_0^\infty \Pb{ \sup_{\ell \ge 1} \frac{\MassPUF\ell}{\ell} \ge \alpha } \d \alpha \eol
			&\le \alpha_0 + 2^{d+1}e C \cdot \p{\int_{\alpha_0}^\infty \alpha^{-d} \d \alpha}  \eol
			&= \alpha_0 + \frac{2^{d+1} e C \alpha_0^{1-d} }{d-1},\nonumber
		\intertext{thus}
		\E{ \sup_{\ell \ge 1} \frac{\MassPUF\ell\cro{ \Phi \otimes \Dirac 1  }}{\ell} }%
			&\le \frac{ d(2^{d+1} e C)^{1/d} }{ d-1 }.
	\end{align}

	In the general case, consider the homothety $f : x \mapsto \lambda^{1/d}x$ on $\R^d$. By Lemma~\ref{lem : suff_condition/fmm_pty_mapping}, $\Phi\circ f^{-1}$ satisfies~\eqref{eqn : intro/suff_condition/fmm} with the constant $C$. Besides, since $\Phi\circ f^{-1}$ has mean measure $\Leb$, by the previous case,
	\begin{equation*}
		\E{ \sup_{\ell \ge 1} \frac{\MassPUF\ell \cro{(\Phi \circ f^{-1})\otimes \Dirac1 }}{\ell} }%
			\le \frac{ d(2^{d+1} e C)^{1/d} }{ d-1 }.
	\end{equation*}
	Moreover, for all continuous paths $\gamma$,
	\begin{align*}
		\Mass{\gamma}\cro{\Phi \otimes \Dirac1} &= \Mass{f(\gamma)} \cro{(\Phi \circ f^{-1})\otimes \Dirac1 }%
		\intertext{and}
		\norme{f(\gamma)} &= \lambda^{1/d} \norme\gamma,
	\end{align*}
	hence
	\begin{equation}
		\E{ \sup_{\ell \ge 1} \frac{\MassPUF\ell \cro{\Phi \otimes \Dirac1 }}{\ell} }%
			\le \frac{ d(2^{d+1} e C \lambda)^{1/d} }{ d-1 },
	\end{equation}
	which concludes the proof with~\eqref{eqn : intro/framework/easy_inequality}.
\end{proof}
\begin{proof}[Proof of Lemma~\ref{lem : suff_condition/projection}]
	We follow the proof of Theorem~2.3 in Martin (2002) \cite{Mar02}. Let $\xi$ be an animal. Then
	\begin{align*}
		\Mass{\xi}\cro\pro %
			&= \sum_{x\in\xi} \Mass{x}\cro{\pro}\\
			&= \sum_{x\in\xi} \int_0^\infty \ind{t\le \Mass{x}\cro\pro}\d t\\
			&=  \int_0^\infty \p{\sum_{x\in\xi} \ind{t\le \Mass{x}\cro\pro} }\d t\\
			&=  \int_0^\infty \p{\Mass\xi\cro{\pro^{(t)}}   } \d t.
	\end{align*}
	Taking the supremum with respect to $\xi \in \SetAUF\ell$, then $\ell>0$ yields~\eqref{eqn : suff_condition/projection}.
\end{proof}

\subsection{Necessary conditions for integrability: proofs of Propositions~\ref{prop : intro/nec_condition} and~\ref{prop : intro/moment_d} }
\label{subsec : nec_condition}

We first state a minor adaptation of the ergodic theorem for point processes (see e.g.\ Theorem~12.2.IV in \cite{Dal07}), in which we assume ergodicity but not locally finiteness of the mean measure.
\begin{Theorem}
	\label{thm : nec_condition/ergodic_thm}
	Let $\Phi$ be a stationary ergodic simple point process on $\R^d$ and $K\subseteq \R^d$ be a compact, convex subset with positive Lebesgue measure. Almost surely,
	\begin{equation}
		\frac{ \Phi(\ell K) }{\ell^d \Leb(K)} \xrightarrow[\ell \to \infty]{} \E{\Phi\p{\intervalleff01^d}}.
	\end{equation}
\end{Theorem}
\begin{proof}
	We only need to treat the case where $ \E{\Phi\p{\intervalleff01^d}} = \infty$. Fix an integer $k\ge 1$ and denote by $\Phi_k$ the point process obtained from $\Phi$ by removing every atom of $\Phi$ with distance to the nearest . Note that $\Phi_k$ is stationary, ergodic, has a locally finite mean measure and almost surely, for all $B \in \Borel{\R^d}$,
	\begin{equation*}
	  	\inclim{k\to \infty} \Phi_k(B) = \Phi(B).
	\end{equation*} 
	By Theorem~12.2.IV in \cite{Dal07}, almost surely, for all $k\ge 1$,
	\begin{equation*}
		\lim_{n \to \infty} \frac{ \Phi_k(n K)}{n^d \Leb(K)} = \E{\Phi_k\p{\intervalleff01^d} }.
	\end{equation*}
	Consequently, almost surely, for all $k\ge1$,
	\begin{equation*}
		\liminf_{n \to \infty} \frac{ \Phi(n K) }{n^d \Leb(K)} %
			\ge \E{\Phi_k\p{\intervalleff01^d} } .
	\end{equation*}
	Letting $k\to \infty$, we get that almost surely,
	\begin{equation*}
		\lim_{n \to \infty} \frac{ \Phi(n K) }{n^d \Leb(K)} %
			= \infty .
	\end{equation*}
	An elementary inclusion argument provides the analogous result with a limit along $\ell \in \intervalleoo0\infty$.
\end{proof}

The proofs of Propositions~\ref{prop : intro/nec_condition} and~\ref{prop : intro/moment_d} rely on Lemmas~\ref{lem : nec_condition/ergodic_thm} and~\ref{lem : nec_condition/TSP}. They are consequences of Theorem~\ref{thm : nec_condition/ergodic_thm} above and Theorem~3 in Few (1955) \cite{Few55}, which gives asymptotics for the minimal time in the so-called \emph{Travelling salesman problem} for the Euclidean distance in a unit cube.
\begin{Lemma}
	\label{lem : nec_condition/ergodic_thm}
	Let $\Phi$ be a stationary ergodic simple point process on $\R^d$ and $\lambda < \E{\Phi\p{\intervalleff01^d}}$. Then there exists $L>0$ such that
	\begin{equation}
		\Pb{ \Phi\p{\intervalleff0L^d} \ge \lambda L^d } \ge 1/2.
	\end{equation}
\end{Lemma}
\begin{Lemma}
	\label{lem : nec_condition/TSP}
	There exists a constant $\Cl{cst : TSP}$, depending only on $d$ and $\norme\cdot$, such that for all $n \in \N^*$ and $x_1, \dots, x_n \in \intervalleff0L^d$, there exists a path $\gamma$ starting at the origin, such that for all $i\in \intint1n, x_i \in \gamma$, and
	\begin{equation}
		\norme{\gamma} \le \Cr{cst : TSP}n^{\frac{d-1}{d}}L.
	\end{equation}
\end{Lemma}
\begin{proof}[Proof of Proposition~\ref{prop : intro/nec_condition}]
Let $t>0$ and $\lambda <  \E{\pro^{(t)}\p{\intervalleff01^d}  }$. For all $L>0$, consider the event
\begin{equation}
	E_t(L) \dpe \acc{ \pro^{(t)} \p{\intervalleff0L^d} \ge \lambda L^d }.
\end{equation}
Lemma~\ref{lem : nec_condition/ergodic_thm} implies the existence of $L>0$ such that
\begin{equation}
	\Pb{ E_t(L) } \ge 1/2.
\end{equation}
Moreover, by Lemma~\ref{lem : nec_condition/TSP},
\begin{equation}
	E_t(L) \subseteq \acc{ \MassPUF{ \Cr{cst : TSP} \p{\lambda L^d}^{\frac{d-1}{d}}L }\cro{\pro^{(t)}\otimes \Dirac1 }  \ge \lambda L^d }.
\end{equation}
Consequently,
\begin{align}
	\E{ \MassPUF{ \Cr{cst : TSP}\lambda^{\frac{d-1}{d}} L^d }\cro{\pro^{(t)}\otimes \Dirac1 }  }%
		&\ge \E{ \MassPUF{ \Cr{cst : TSP}\lambda^{\frac{d-1}{d}} L^d }\cro{\pro^{(t)}\otimes \Dirac1 }  \ind{E_t(L)} }\eol
		&\ge \E{ \lambda L^d \ind{E_t(L)} }  \eol
		&\ge \frac{ \lambda L^d}2.\nonumber
\end{align}
In particular,
\begin{equation}
	\sup_{\ell \ge 1} \frac{ \E{\MassPUF\ell\cro{\pro^{(t)}\otimes \Dirac1 } } }{\ell} \ge \frac{ \lambda^{1/d} }{2\Cr{cst : TSP}},
\end{equation}
thus
\begin{equation}
	\label{eqn : nec_condition/final_inequality}
	\sup_{\ell \ge 1} \frac{ \E{\MassPUF\ell\cro{\pro^{(t)}\otimes \Dirac1 } } }{\ell} \ge \frac{ \E{\pro^{(t)}\p{\intervalleff01^d}  }^{1/d} }{2\Cr{cst : TSP}}.
\end{equation}

To prove the first part of Proposition~\ref{prop : intro/nec_condition}, note that
\begin{equation*}
	\sup_{\ell \ge 1}\frac{ \E{\MassAUF\ell\cro\pro } }{\ell} \ge  t\sup_{\ell \ge 1} \frac{ \E{\MassPUF\ell\cro{\pro^{(t)}\otimes \Dirac1 } } }{\ell}.
\end{equation*}
Applying~\eqref{eqn : nec_condition/final_inequality} gives 
\begin{equation*}
	\E{\pro^{(t)}\p{\intervalleff01^d}} < \infty,
\end{equation*}
thus $\MM[\pro]$ is locally finite.

Integrating~\eqref{eqn : nec_condition/final_inequality} with respect to $t$ on $\intervalleoo0\infty$, applying~\eqref{eqn : intro/nec_condition/layers} and noting that $\E{\pro^{(t)}\p{\intervalleff01^d}  } = \nu\p{\intervallefo t\infty}$ gives the second part of the proposition. 
\end{proof}
\begin{proof}[Proof of Proposition~\ref{prop : intro/moment_d}]
	Let $\Phi$, $\nu$ and $\pro$ be as in Proposition~\ref{prop : intro/moment_d}. Assume that~\eqref{eqn : intro/moment_d} fails. We prove that Assumption~\ref{ass : intro/main/Moment} also fails. In the decomposition $\Phi = \sum_{n\ge1} \Dirac{z_n}$ (see~\eqref{eqn : intro/main/decomposition_pro}), we can assume that $\p{\norme{z_n} }_{n \ge 1}$ is nondecreasing. We claim that there exists a constant $c>0$, depending only on $\norme\cdot$, such that that almost surely, for large enough $n$,
	\begin{equation}
		\label{eqn : nec_condition/norme_xn}
		\norme{z_n} \le \frac{c n^{1/d}}{\lambda^{1/d} }.
	\end{equation}
	Indeed by Theorem~\ref{thm : nec_condition/ergodic_thm}, almost surely, there exists $\ell_0>0$ such that for all $\ell\ge \ell_0$,
	\begin{equation*}
		\Phi\p{\clball{0,\ell} } \ge \frac \lambda2\Leb\p{\clball{0, \ell}}.
	\end{equation*}
	Taking $\ell_n \dpe \p{ \frac{2n}{\lambda \Leb\p{\clball{0, 1} } } }^{1/d}$ gives
	\begin{equation*}
		\Phi\p{\clball{0,\ell_n} } \ge n
	\end{equation*}
	almost surely, for large enough $n$, thus~\eqref{eqn : nec_condition/norme_xn}.

	Let $s>0$. Since $\nu$ has an infinite $d$-th moment,
	\begin{equation*}
		\sum_{n\ge 1} \Pb{\Mass{z_n} \ge sn^{1/d} }%
			= \sum_{n\ge 1} \Pb{ \Mass{z_1}^d \ge s^dn } =\infty.
	\end{equation*}
	Besides, the $\Mass{z_n}$ are independent, thus by the Borel-Cantelli lemma, almost surely, for infinitely many $n\ge 1$,
	\begin{equation}
		\label{eqn : nec_condition/good_n}
		\MassAUF{\norme{z_n}} \ge \Mass{z_n} \ge sn^{1/d}.
	\end{equation}
	Equations~\eqref{eqn : nec_condition/norme_xn} and~\eqref{eqn : nec_condition/good_n} yield
	\begin{align*}
		\limsup_{\ell \to \infty} \frac{\MassAUF\ell}{\ell} &\ge \frac{s \lambda^{1/d} }c.
		\intertext{This inequality holds for all $s>0$ therefore}
		\limsup_{\ell \to \infty} \frac{\MassAUF\ell}{\ell} &= \infty.
	\end{align*}
	In particular, the conclusion of Theorem~\ref{thm : intro/main/MAIN} fails, hence Assumption~\ref{ass : intro/main/Moment} also fails.
\end{proof}
\section{Proof of the corollaries}
\label{sec : app}

\subsection{Lattice paths and animals}
In this section we prove Corollary~\ref{cor : intro/special/Zd}. We only treat the harder case $\AUXGeneric =\AUXAnimal$. Let $\norme\cdot = \norme[1]\cdot$ and $\p{\Mass v}_{v\in \Z^d}$ be as in the corollary. We consider the process
\begin{equation}
	\pro \dpe \sum_{v\in \Z^d} \Dirac{v, \Mass v}. 	
\end{equation}
Let $U$ be a uniform random variable on $\intervalleff01^d$, such that $U$ and $\p{\Mass v}_{v\in \Z^d}$ are independent. We define the point process
\begin{equation}
	\pro' \dpe T_U \pro.
\end{equation}
Our main arguments are Lemmas~\ref{lem : app/lattice/reduction} and~\ref{lem : app/lattice/ergodicity}, which imply that $\MassLADF xyn$ can be approximated by $\MassADFpen{x}{y}{n}\infty\cro{\pro'}$, and $\pro'$ satisfies the hypotheses of Theorem~\ref{thm : intro/main/MAIN_PENALIZED}.
\begin{Lemma}
	\label{lem : app/lattice/reduction}
	For all $x,y\in\Z^d$ and $n\in \N$,
	\begin{equation}
		\label{eqn : app/lattice/reduction}
		\MassLADF xy{n+1} = \MassADFpen{x}{y}{n}\infty\cro\pro.
	\end{equation}
\end{Lemma}
\begin{Lemma}
	\label{lem : app/lattice/ergodicity}
	The process $\pro'$ is stationary and satisfies Assumptions~\ref{ass : intro/main/Ergodic_Stationary} and~\ref{ass : intro/main/Moment}. Moreover, for all $x,y\in \R^d$ and $\ell >2d$,
	\begin{equation}
		\label{eqn : app/lattice/ergodicity/encadrement}
		\MassADFpen{x}{y}{\ell-2d}\infty\cro{\pro'} \le \MassADFpen{x}{y}{\ell}\infty\cro{\pro} \le \MassADFpen{x}{y}{\ell+2d}\infty\cro{\pro'}.
	\end{equation}
\end{Lemma}
\begin{proof}[Proof of Corollary~\ref{cor : intro/special/Zd}]
	Let $K \subseteq \cX$ be a compact set. It is sufficient to show that
	\begin{equation}
		\label{eqn : app/lattice/win_condition}
		\sup\set{\module{ \frac{ \MassLADF{ \floor{Lx} }{ \floor{Ly} }{L\ell} }{L} - \ell\LimMassA[\pro']\p{\frac{x-y}{\ell}} }}%
		{(x,y,\ell) \in K}
		\xrightarrow[L\to\infty]{\text{a.s. and }\rL^1} 0.
	\end{equation}
	Define
	\begin{equation}
		\eps \dpe 1- \max \set{ \frac{\norme{x-y}}{\ell} }{ (x,y,\ell)\in K } >0.
	\end{equation}
	Let $0<\delta \le \frac\eps2$. Consider the set
	\begin{equation}
		K_\delta \dpe \set{ \p{ (1+\delta)x,  (1+\delta)y, \ell } }{(x,y,\ell)\in K} \cup \set{ \p{ (1-\delta)x,  (1-\delta)y, \ell } }{(x,y,\ell)\in K}.
	\end{equation}
	Note that $K_\delta$ is compact, $K_\delta \subseteq \cX$ and for all $(x,y,\ell) \in K_\delta$,
	\begin{equation}
		\label{eqn : app/lattice/K_delta}
	 	\frac{\norme{x-y} }{\ell} \le 1- \frac\eps2.	
	\end{equation} 
	Let $\omega$ be a modulus of continuity of the restriction of $\LimMassA[\pro']$ on $\clball{0, 1-\frac\eps2}$.

	By Lemma~\ref{lem : app/lattice/ergodicity}, the process $\pro'$ satisfies the hypotheses of Theorem~\ref{thm : intro/main/MAIN_PENALIZED}, thus
	\stepcounter{error}%
	\begin{equation}
		\label{eqn : app/lattice/input_thm}
		\cError(L) \dpe \sup\set{ \module{ \frac{\MassADFpen{Lx}{Ly}{L\ell}\infty\cro{\pro'} }{L} - \ell\LimMassA^{(\infty)}\p{\frac{x-y}{\ell}}\cro{\pro'}  } }%
		{(x,y,\ell) \in K_\delta}%
		\xrightarrow[L\to \infty]{\text{a.s. and }\rL^1} 0.
	\end{equation}
	By~\eqref{eqn : app/lattice/reduction} and~\eqref{eqn : app/lattice/ergodicity/encadrement}, for large enough $L>0$, for all $(x,y,\ell)\in K$,
	\begin{equation*}
		\MassLADF{ \floor{Lx} }{ \floor{Ly} }{L\ell} \le \MassADFpen{Lx}{Ly}{\frac{ L\ell}{1-\delta} }\infty\cro{\pro'},
	\end{equation*}
	thus by triangle inequality
	\begin{align*}
		&\sup\set{\cro{ \frac{ \MassLADF{ \floor{Lx} }{ \floor{Ly} }{L\ell} }{L} - \ell\LimMassA\p{\frac{x-y}{\ell}}[\pro'] }^+ }%
		{(x,y,\ell) \in K} \\%
			&\quad \le  \sup\set{\cro{ \frac{ \MassADFpen{ Lx }{ Ly }{ \frac{L\ell}{1-\delta} }\infty[\pro'] }{L} - \ell\LimMassA\p{\frac{x-y}{\ell}}[\pro'] }^+ }%
		{(x,y,\ell) \in K}\\
			&\quad \le \sup\set{\module{ \frac{ \MassADFpen{ Lx }{ Ly }{ \frac{L\ell}{1-\delta} }\infty[\pro'] }{L} - \frac{\ell}{1-\delta}\LimMassA\p{\frac{(1-\delta)(x-y)}{\ell} }[\pro'] } }%
		{(x,y,\ell) \in K}\\%
			&\qquad + \sup\set{\module{ \frac{\ell}{1-\delta}\LimMassA\p{\frac{(1-\delta)(x-y)}{\ell}}[\pro'] - \ell\LimMassA\p{\frac{x-y}{\ell}}[\pro'] } }%
		{(x,y,\ell) \in K}.
	\end{align*}
	Consequently, by~\eqref{eqn : app/lattice/K_delta}, the definition of $\cError(L)$ and the definition of $\omega$,
	\begin{equation}
		\label{eqn : app/lattice/before_limit1}
		\begin{split}
		&\sup\set{\cro{ \frac{ \MassLADF{ \floor{Lx} }{ \floor{Ly} }{L\ell} }{L} - \ell\LimMassA\p{\frac{x-y}{\ell}}[\pro'] }^+ }%
		{(x,y,\ell) \in K} \\ &\quad\le \frac{1}{1-\delta} \cError\p{\frac L{1-\delta}} + \p{\frac{1}{1-\delta} - 1 }\LimMassA(0)[\pro']  + \omega\p{\delta}.
		\end{split}
	\end{equation}
	Similarly,
	\begin{equation}
		\label{eqn : app/lattice/before_limit2}
		\begin{split}
		&\sup\set{\cro{ \frac{ \MassLADF{ \floor{Lx} }{ \floor{Ly} }{L\ell} }{L} - \ell\LimMassA\p{\frac{x-y}{\ell}}[\pro'] }^- }%
		{(x,y,\ell) \in K} \\ &\quad \le \frac{1}{1+\delta} \cError\p{\frac L{1+\delta}} + \p{1-\frac{1}{1+\delta} }\LimMassA(0)[\pro']  + \omega\p{\delta}.
		\end{split}
	\end{equation}
	Taking $L\to \infty$ then $\delta\to 0$ gives the almost sure convergence in~\eqref{eqn : app/lattice/win_condition}. The inequality~\eqref{eqn : app/lattice/before_limit1} also gives the $\rL^1$ convergence by domination. 
\end{proof}
\begin{proof}[Proof of Lemma~\ref{lem : app/lattice/reduction}]
	Let $x,y\in\Z^d$ and $n\in \N$. Any lattice animal $\xi \in \SetLADF xy{n+1}$ may be covered by a tree belonging to $\SetADFalt{x}{y}{n}$ and included in $\Z^d$, therefore
	\begin{equation}
		\label{eqn : app/lattice/reduction/UB}
		\MassLADF xy{n+1} \le \MassADFpen{x}{y}{n}\infty\cro\pro.
	\end{equation}

	Conversely, let $\xi=(V,E) \in \SetADFalt{x}{y}{n}$. Assume that $\xi \subseteq \Z^d$. For all edges $\acc{z,z'} \in E$, there exists a lattice path %
	\[ \gamma(z,z') = (z=z_0,\dots, z_r = z'), \]%
	with $r=\norme[1]{z-z'}$. Let \[ \xi' \dpe \bigcup_{\acc{z,z'}\in E} \gamma(z,z'). \] Then $V\subseteq \xi'$ and $\xi' \in \SetLADF xy{n+1}$, thus
	\begin{equation}
		\label{eqn : app/lattice/reduction/LB}
		\MassLADF xy{n+1} \ge \MassADFpen{x}{y}{n}\infty\cro\pro.
	\end{equation}
\end{proof}
\begin{proof}[Proof of Lemma~\ref{lem : app/lattice/ergodicity}]
	Stationarity is a straightforward extension of Example~6.1.4 in \cite{Bac20} to marked processes. To show ergodicity, let $E\subset \ProSpace$ be an invariant subset, i.e.\ $E$ satisfies~\eqref{eqn : intro/main/invariant_event}. Then by independence of $\p{\Mass v}_{v\in \Z^d}$ and $U$,
	\begin{align}
		\Pb{\pro' \in E}%
			&= \int_{\intervalleff01^d} \Pb{ T_u\pro \in E } \d u.\nonumber%
		\intertext{Since $E$ is invariant,}
		\Pb{\pro' \in E}%
			&= \int_{\intervalleff01^d} \Pb{ \pro \in E } \d u = \Pb{ \pro \in E }. \nonumber
	\end{align}
	Using the ergodicity of the family $\p{\Mass v}_{v\in \Z^d}$, we obtain $\Pb{\pro'\in E} \in \acc{0,1}$, thus Assumption~\ref{ass : intro/main/Ergodic_Stationary} is proven.

	Let $x,y\in \R^d$ and $\ell > 2d$. Let $\xi' \in \SetADFalt{x}{y}{\ell-2d}$ such that $\xi' \subseteq (\pro')^*$. Then the translated animal \begin{equation*}
		\xi \dpe \xi' - U
	\end{equation*}
	satisfies $\xi \subseteq \projpro$, $\Mass{\xi}\cro\pro = \Mass{\xi'}\cro{\pro'}$ and by the triangle inequality,
	\begin{align*}
		\norme[1]{\xi} + \d(x,\xi) + \d(y,\xi)%
			&\le \norme[1]{\xi'} + \d(x,\xi') + \d(y,\xi') + 2\norme[1] U \le \ell,
	\end{align*}
	i.e.\ $\xi \dpe \SetADFalt{x}{y}{\ell}$. This implies the first inequality in~\eqref{eqn : app/lattice/ergodicity/encadrement}. The second one is similar.

	We now prove Assumption~\ref{ass : intro/main/Moment}. Let $\ell \ge 1$. By~\eqref{eqn : app/lattice/reduction/UB},
	\begin{align}
		\MassAUFpen{\ell}\infty\cro\pro %
			&= \MassADFpen00{\ell}\infty\cro\pro \nonumber\\
			&\le \MassADFpen00{\ceil\ell}\infty\cro\pro \nonumber\\
			&= \MassLADF00{\ceil \ell +1}\nonumber\\
			&= \MassLAUF{\ceil \ell +1}.
		\intertext{Applying~\eqref{eqn : intro/special/Zd/moment} yields}
		\sup_{\ell \ge 1} \frac{\E{\MassAUFpen{\ell}\infty\cro\pro} }{\ell} &<\infty.
		\intertext{Consequently, by~\eqref{eqn : intro/framework/easy_inequality},}
		\sup_{\ell \ge 1} \frac{\E{\MassAUF{\ell}\cro\pro} }{\ell} &<\infty.
	\end{align}
	The first inequality in~\eqref{eqn : app/lattice/ergodicity/encadrement} gives the analogous bound for $\pro'$, i.e.\ Assumption~\ref{ass : intro/main/Moment}. \end{proof}


\subsection{Greedy animals and paths and Determinantal point processes}
In this section we prove Corollary~\ref{cor : intro/special/DPP}. We fix a stationary marked simple point process $\pro$ on $\R^d \times \intervalleoo0\infty$, with mean measure $\Leb \otimes \nu$. We assume that $\pro$ is a good DPP with kernel $K$ and background measure $\mu= \Leb\otimes\nu$, such that $\nu$ satisfies~\eqref{eqn : intro/context/greedy_condition}. By stationarity, for all $z \in \R^d$, for $(\Leb\otimes \nu)^{\otimes k}$-almost all $\p{ (x_1,t_1), \dots, (x_k, t_k) }\in \p{\R^d \times \intervalleoo0\infty}^k$,
\begin{equation}
	\label{eqn : app/DPP/stationarity/det}
	\det\p{K\p{(x_i + z, t_i), (x_j +z, t_j)} }_{1\le i,j \le k} = \det\p{K\p{(x_i , t_i), (x_j , t_j)} }_{1\le i,j \le k}.
\end{equation}
Moreover, $\d x \nu (\d t) = \MM[\pro](\d x, \d t) = K\p{(x,t) , (x,t) } \d x \nu(\d t)$, thus for $(\Leb \otimes \nu)$-almost all $(x,t)\in\R^d\times \intervalleoo0\infty$,
\begin{equation}
	\label{eqn : app/DPP/stationarity/1}
	K\p{(x,t) , (x,t)} = 1.
\end{equation}
In particular, by Proposition~5.1.20 in \cite{Bac20}, for all $k\ge 1$ and Borel subsets $B_1,\dots,B_k \subseteq \R^d\times \intervalleoo0\infty$,
\begin{equation}
	\label{eqn : app/DPP/fmm}
	\FMM{k}\p{\prod_{i=1}^k B_i} \le \prod_{i=1}^k \MM(B_i),
\end{equation}
i.e.~\eqref{eqn : intro/suff_condition/fmm} with $C=1$. In particular, by Proposition~\ref{prop : intro/suff_condition/fmm} Assumption~\ref{ass : intro/main/Moment} holds.

To show Assumption~\ref{ass : intro/main/Ergodic_Stationary}, it is sufficient to prove Lemma~\ref{lem : app/DPP/mixing_vp}.
\begin{Lemma}
	\label{lem : app/DPP/mixing_vp}
	For all compact subsets $A,A'\subseteq \R^d$ and $B,B'\subseteq \intervalleoo0\infty$,
	\begin{equation}
		\label{eqn : app/DPP/mixing_vp}
		\Pb{\pro\p{A\times B}=0, \pro\p{(A'+z)\times B'}=0 } \xrightarrow[\norme z \to \infty]{} \Pb{\pro\p{A\times B}=0} \Pb{\pro\p{A'\times B'}=0}.
	\end{equation}
\end{Lemma} 
Indeed, since $\pro$ is simple, by Rényi's theorem (see e.g.\ \cite[Theorem~2.1.10]{Bac20}) the void probabilities $\Pb{\pro\p{A\times B}=0 }$ with compact subsets $A\subseteq \R^d$ and $B\subseteq \intervalleoo0\infty$ characterize its distribution. Consequently, Lemma~12.3.II in Daly, Vere-Jones (2007) \cite{Dal07} and Lemma~\ref{lem : app/DPP/mixing_vp} imply that $\pro$ is mixing, thus ergodic. \qed

\begin{proof}[Proof of Lemma~\ref{lem : app/DPP/mixing_vp}]
	By Corollary~5.1.19 in \cite{Bac20}, for all compact subset $D\subseteq \R^d\times \intervalleoo0\infty$,
	\begin{equation}
		\label{eqn : app/DPP/mixing_vp/expression_vp}
		\Pb{\pro(D)=0}%
			=1 + \sum_{k=1}^\infty \frac{(-1)^k}{k!} \int_{D^k} \det \p{K\p{(x_i,t_i), (x_j, t_j)} }_{1\le i,j\le k}\d x_1 \dots \d x_k \nu(\d t_1)\dots \nu(\d t_k).
	\end{equation}
	For all integers $n,m,k\ge 0$ such that $n+m=k$, $x_{1:k}\dpe (x_1,\dots, x_k)\in (\R^d)^k$, $t_{1:k} \dpe (t_1,\dots, t_k)\in \intervalleoo0\infty^k $ and $z\in \R^d$, define
	\begin{equation*}
		f(n,m,k, x_{1:k}, t_{1:k})[z] \dpe \det\p{K\p{y_i,y_j} }_{1\le i,j \le k},
	\end{equation*}
	where
	\begin{equation*}
		y_i \dpe%
		\begin{cases}
			(x_i, t_i) &\text{ if } 1\le i \le n,\\
			(x_i+z, t_i) &\text{ if } n+1 \le j \le k.
		\end{cases}
	\end{equation*}

	Fix $n,m,k$. We claim that for $\Leb^{\otimes k} \otimes \nu^{\otimes k}$-almost all $(x_{1:k} , t_{1:k})$,
	\begin{equation}
		\label{eqn : app/DPP/mixing_vp/pointwise_cvg}
		\lim_{\norme z \to \infty} f(n,m,k, x_{1:k}, t_{1:k})[z] =\det\p{K\p{(x_i,t_i), (x_j,t_j)} }_{1\le i,j \le n} \cdot \det\p{K\p{(x_i,t_i), (x_j,t_j)} }_{n+1\le i,j \le k}.
	\end{equation}
	Indeed, consider the subgroup $\Sym[1]{\intint1k}$ of $\Sym{\intint1k}$ consisting of permutations that fix $\intint1n$ and $\intint{n+1}{k}$, and $\Sym[2]{\intint1k} \dpe \Sym{\intint1k} \setminus \Sym[1]{\intint1k}$. We have
	\begin{align}
		&f(n,m,k, x_{1:k}, t_{1:k})[z]\\%
			&\quad= \sum_{\sigma \in \Sym{\intint1k} } \eps(\sigma) \prod_{i=1}^k K\p{y_i, y_{\sigma(i)} }\eol
			&\quad= \sum_{\sigma \in \Sym[1]{\intint1k} } \eps(\sigma) \prod_{i=1}^k K\p{y_i, y_{\sigma(i)} }
				+ \sum_{\sigma \in \Sym[2]{\intint1k} } \eps(\sigma) \prod_{i=1}^k K\p{y_i, y_{\sigma(i)} }\eol
			&\quad= \p{\sum_{\sigma_1 \in \Sym{\intint1n} } \eps(\sigma_1) \prod_{i=1}^n K\p{y_i, y_{\sigma_1(i)} }} \p{\sum_{\sigma_2 \in \Sym{\intint{n+1}{m} } } \eps(\sigma_2) \prod_{i=n+1}^k K\p{y_i, y_{\sigma_2(i)} }} \eol%
				&\qquad + \sum_{\sigma \in \Sym[2]{\intint1k} } \eps(\sigma) \prod_{i=1}^k K\p{y_i, y_{\sigma(i)} }.\nonumber
	\end{align}
	By~\eqref{eqn : app/DPP/stationarity/det}, for $\Leb^{\otimes k} \otimes \nu^{\otimes k}$-almost all $(x_{1:k} , t_{1:k})$,
	\begin{align}
		f(n,m,k, x_{1:k}, t_{1:k})[z] &= \det\p{K\p{(x_i,t_i), (x_j,t_j)} }_{1\le i,j \le n} \cdot \det\p{K\p{(x_i,t_i), (x_j,t_j)} }_{n+1\le i,j \le k} \eol
		&\quad + \sum_{\sigma \in \Sym[2]{\intint1k} } \eps(\sigma) \prod_{i=1}^k K\p{y_i, y_{\sigma(i)} }.\nonumber
	\end{align}
	Each term in the sum is the product of factors that are either constant with respect to $z$ or converges to $0$ when $\norme z \to \infty$, and at least one of them belongs to the second category, thus applying~\eqref{eqn : intro/special/DPP/vanish_infty} implies~\eqref{eqn : app/DPP/mixing_vp/pointwise_cvg}.

	Let $A,A'\subseteq \R^d$ and $B,B'\subseteq \intervalleoo0\infty$ be compact subsets and $z\in \R^d$. Assume that $\norme z$ is large enough so that $A\cap (A'+z) = \emptyset$. Then by~\eqref{eqn : app/DPP/mixing_vp/expression_vp},
	\begin{align}
		&\Pb{\pro\p{A\times B}=0,\quad \pro\p{(A'+z)\times B'}=0}\eol
			&\quad= \Pb{ \pro\p{ \p{A\times B} \cup \p{(A'+z)\times B'}  } = 0 }\eol
			&\quad=1 + \sum_{k=1}^\infty \frac{(-1)^k}{k!} \int_{\p{ \p{A\times B} \cup \p{(A'+z)\times B'}  }^k} \det \p{K\p{(x_i,t_i), (x_j, t_j)} }_{1\le i,j\le k}\d x_1 \dots \d x_k \nu(\d t_1)\dots \nu(\d t_k)\eol
			&\quad=1 + \sum_{k=1}^\infty \frac{(-1)^k}{k!} \sum_{n+m=k}\Biggl[%
			\frac{k!}{n!m!} \eol
			&\qquad\qquad \cdot\int_{ \p{A\times B}^n\times \p{(A'+z)\times B'}^m } \det \p{K\p{(x_i,t_i), (x_j, t_j)} }_{1\le i,j\le k}\d x_1 \dots \d x_k \nu(\d t_1)\dots \nu(\d t_k) \Biggr]\eol
			&\quad=1 + \sum_{k=1}^\infty \sum_{n+m=k}\cro{%
			\frac{ (-1)^k }{n!m!} \int_{ \p{A\times B}^n\times \p{A'\times B'}^m } f(n,m,k, x_{1:k}, t_{1:k})[z] \d x_1 \dots \d x_k \nu(\d t_1)\dots \nu(\d t_k) }.
		\label{eqn : app/DPP/coupled}
	\end{align}
	We will apply the dominated convergence theorem to the integral in~\eqref{eqn : app/DPP/coupled}. By Hadamard's inequality (see e.g.\ \cite[Theorem~7.8.1]{Hor85}) and~\eqref{eqn : app/DPP/stationarity/1},
	\begin{equation*}
		0\le  f(n,m,k, x_{1:k}, t_{1:k})[z]  \le \prod_{i=1}^k K\p{y_i, y_i} = 1,
	\end{equation*}
	thus
	\begin{align}
		&1+\sum_{k=1}^\infty \sum_{n+m=k}\cro{%
			\frac{ 1 }{n!m!} \int_{ \p{A\times B}^n\times \p{A'\times B'}^m } f(n,m,k, x_{1:k}, t_{1:k})[z] \d x_1 \dots \d x_k \nu(\d t_1)\dots \nu(\d t_k) }\eol
		&\quad \le 1+\sum_{k=1}^\infty \sum_{n+m=k}%
			\frac{  \Leb(A)^n\nu(B)^n \Leb(A')^m\nu(B')^m }{n!m!} \eol
		&\quad = \p{1+ \sum_{n=1}^\infty \frac{\Leb(A)^n \nu(B)^n}{n!} } \cdot \p{1+ \sum_{m=1}^\infty \frac{\Leb(A')^m \nu(B')^m}{m!} } < \infty.\nonumber
	\end{align}
	Consequently,~\eqref{eqn : app/DPP/mixing_vp/pointwise_cvg} and~\eqref{eqn : app/DPP/coupled} yield
	\begin{equation}
	\begin{split}
		&\lim_{\norme z \to \infty }\Pb{\pro\p{A\times B}=0, \pro\p{(A'+z)\times B'}=0} \\
		&\quad= 1 + \sum_{k=1}^\infty \sum_{n+m=k}\Biggl[%
			\frac{ (-1)^k }{n!m!} \int_{ \p{A\times B}^n\times \p{A'\times B'}^m }\eol
		&\qquad \det\p{K\p{(x_i,t_i), (x_j,t_j)} }_{1\le i,j \le n} \cdot \det\p{K\p{(x_i,t_i), (x_j,t_j)} }_{n+1\le i,j \le k} \d x_1 \dots \d x_k \nu(\d t_1)\dots \nu(\d t_k) \Biggl],
	\end{split}
	\end{equation}
	thus applying~\eqref{eqn : app/DPP/mixing_vp/expression_vp} once again gives~\eqref{eqn : app/DPP/mixing_vp}. 
\end{proof}

\appendix
\section{A superadditive ergodic theorem}
\label{appsec : AK81}
In this section we prove Theorem~\ref{thm : intro/outline/AK81}. We follow Akcoglu and Krengel (1981) \cite{AK81}. For all processes $X$ as in the theorem, we define
\begin{equation}
	\label{eqn : AK81/time_constant}
	\TC{X} \dpe \sup_{t>0} \frac{\E{X(0,t)}}{t} = \sup_{t\ge 1} \frac{\E{X(0,t)}}{t} = \lim_{t\to\infty} \frac{\E{X(0,t)}}{t},
\end{equation}
the second and third equalities being consequences of Fekete's lemma. Our main argument is the maximal inequality given by Lemma~\ref{lem : AK81/max_inequality}.
\begin{Lemma}
	\label{lem : AK81/max_inequality}
	Let $\p{ \hat X(s,t) }_{\substack{s<t \\ s,t\in \Z} }$ be a stationary, nonnegative, superadditive, discrete process. Then for all $\alpha>0$,
	\begin{equation}
		\label{eqn : AK81/max_inequality}
		\Pb{\sup_{n \ge 1} \frac{\hat X(-n,n)}{2n}> \alpha} \le \frac{3\TC{\hat X} }{\alpha}.
	\end{equation}
\end{Lemma}
\begin{proof}
	Let $N\in \N^*$. Corollary~4.5 in Akcoglu and Krengel (1981) \cite{AK81} with the $1$-regular family $\p{\intint{-n}{n}}_{1\le n \le N}$ gives
	\begin{equation*}
		\Pb{\sup_{1\le n \le N} \frac{\hat X(-n,n)}{2n}> \alpha} \le \frac{3\TC{\hat X} }{\alpha}.
	\end{equation*}
	Letting $N\to \infty$ yields~\eqref{eqn : AK81/max_inequality}.
\end{proof}

\begin{proof}[Proof of Theorem~\ref{thm : intro/outline/AK81}]
Let $a\le b$. We define
\begin{equation}
	\Yinf(a,b) \dpe \liminf_{\ell \to \infty} \frac{X(a\ell, b\ell)}{\ell}%
	\text{ and }%
	\Ysup(a,b) \dpe \limsup_{\ell \to \infty} \frac{X(a\ell, b\ell)}{\ell}.
\end{equation}
For all $t>0$ and $n,m \in \Z$ such that $n\le m$, we define
\begin{equation}
	S_t(n,m) \dpe \sum_{k=n}^{m-1}X\p{kt, (k+1)t}.
\end{equation}
By Birkhoff's ergodic theorem (see e.g.\ \cite[Theorem~1.14]{Wal00}), for all $t>0$, for all integers $a\le b$, the limit
\begin{equation}
	\label{eqn : AL81/Birkhoff}
	\lim_{n\to \infty} \frac1n S_t(an,bn)
\end{equation}
exists a.s. and in $\rL^1$. In particular, the inequality
\begin{equation*}
	X(a\ell, b\ell) \ge \sum_{k = \ceil{\frac{\ell a}{t} } }^{\floor{\frac{\ell b}{t}-1}} X\p{kt, (k+1)t}
\end{equation*}
implies
\begin{equation*}
	\E{\Yinf(a,b)} \ge (b-a)\TC{X}.
\end{equation*}
Fatou's lemma yields the converse inequality, thus
\begin{equation}
	\label{eqn : AK81/expectaction_limit}
 	\E{\Yinf(a,b)} = (b-a)\TC{X}.
\end{equation}
To prove the almost sure convergence, we treat three cases differently, according to the values of $a$ and $b$.

\emph{Case 1: Assume that $a$ and $b$ are rational and satisfy $a\le 0 \le b$.} Define $c\dpe \module a \vee \module b$. Let $\alpha>0$. Since $\ell \mapsto X(a\ell, b\ell)$ is nondecreasing, for all $L>0$,
\begin{equation}
	\label{eqn : AK81/Yinf_Ysup_tranches}
	\Yinf(a,b) = \liminf_{k \to \infty} \frac{X(akL, bkL)}{kL}%
	\text{ and }%
	\Ysup(a,b) = \limsup_{k \to \infty} \frac{X(akL, bkL)}{kL}.
\end{equation}
In particular, without loss of generality $a$ and $b$ may be assumed to be integers. By~\eqref{eqn : AK81/time_constant} there exists $t>0$ such that
\begin{equation}
	\label{eqn : AK81/large_t_TC}
	\frac{\E{X(0,t)} }{t} \ge \TC{X} -\eps.
\end{equation}
Consider the discrete process $\p{\hat X_t(n,m)}_{\substack{n<m \\ n,m\in \Z} }$ defined by
\begin{equation}
	\hat X_t(n,m) \dpe X(nt, mt) - S_t(n,m).
\end{equation}
Then by~\eqref{eqn : AL81/Birkhoff} and~\eqref{eqn : AK81/Yinf_Ysup_tranches}, almost surely,
\begin{equation*}
	\Ysup(a,b) - \Yinf(a,b)%
		= \limsup_{k\to\infty} \frac{\hat X_t(ka, kb)}{kt}  - \liminf_{k\to\infty} \frac{\hat X_t(ka, kb)}{kt}.
\end{equation*}
Moreover, since $\hat X_t$ is nonnegative and superadditive, for all $k\ge 1$,
\begin{equation*}
	0 \le \frac{\hat X_t(ka, kb)}{kt} \le \frac{\hat X_t(-kc, kc)}{kt}.
\end{equation*}
Consequently,
\begin{equation*}
	\Ysup(a,b) - \Yinf(a,b)%
		\le \sup_{k\ge 1} \frac{\hat X_t\p{-kc,kc} }{kt}.
\end{equation*}
Besides, $\hat X_t$ satisfies the hypothesis of Lemma~\ref{lem : AK81/max_inequality}, thus for all $\alpha>0$,
\begin{align}
	\Pb{\Ysup(a,b) - \Yinf(a,b) > \alpha} %
		&\le \Pb{\sup_{k\ge 1} \frac{\hat X_t\p{-kc,kc} }{2kc} > \frac{t\alpha}{2c} } \nonumber\\%
		&\le \frac{6c\TC{\hat X_t} }{ t\alpha}.\nonumber%
	\intertext{Moreover, by~\eqref{eqn : AK81/large_t_TC}, $\TC{\hat X_t} = t\p{ \TC{X} - \TC{S_t} } \le \eps$, therefore}
	\Pb{\Ysup(a,b) - \Yinf(a,b) > \alpha} %
		&\le \frac{6c\eps }{ \alpha}.\nonumber
\end{align}
Consequently, almost surely, $\Yinf(a,b) = \Ysup(a,b)$, i.e.\ the limit
\begin{equation}
	\label{eqn : AK81/as_cvg}
	Y(a,b) \dpe \lim_{\ell \to \infty} \frac{X(a\ell, b\ell)}{\ell}
\end{equation}
exists.

\emph{Case 2: Assume that $a$ and $b$ are rational numbers with the same sign.} We only treat the subcase $0\le a \le b$, the other one being similar. Using Case 1, by superaddivity,
\begin{equation}
	\Yinf(a,b) \le \Ysup(a,b) \le Y(0,b) - Y(0,a).
\end{equation}
By~\eqref{eqn : AK81/expectaction_limit} the lefthand side and the righthand side have the same expectation, thus almost surely, 
\begin{equation}
	Y(a,b) \dpe \lim_{\ell \to \infty} \frac{X(a\ell, b\ell)}{\ell} = Y(0,b) - Y(0,a).
\end{equation}

\emph{General case.} Consider two sequences of rational numbers $(a_n)$ and $(b_n)$ such that
\begin{equation}
	\inclim{n\to \infty} a_n = a\text{ and } \declim{n\to \infty} b_n =b.
\end{equation}
Using Case 1 or 2, depending on the sign of $a_n$ and $b_n$, by monotone convergence, almost surely,
\begin{equation}
	\Yinf(a,b) \le \Ysup(a,b) \le \declim{n\to \infty} Y(a_n ,b_n).
\end{equation}
By~\eqref{eqn : AK81/expectaction_limit} the leftmost and rightmost terms have the same expectation, thus almost surely,
\begin{equation*}
	Y(a,b) \dpe \lim_{\ell \to \infty} \frac{X(a\ell, b\ell)}{\ell} =  \declim{n\to \infty} Y(a_n ,b_n),
\end{equation*}
i.e.\ the almost sure part in~\eqref{eqn : intro/outline/AK81/convergence}, and~\eqref{eqn : intro/outline/AK81/convergence_ab} hold. 

We now turn to the $\rL^1$ convergence in~\eqref{eqn : intro/outline/AK81/convergence}. By~\eqref{eqn : AK81/expectaction_limit}, $Y(a,b)$ is integrable, hence by dominated convergence
\begin{equation}
	\label{eqn : AK81/cvg_negative_part}
	\lim_{\ell \to \infty} \E{\p{\frac{X(a\ell, b\ell)}{\ell} - Y(a,b)}^- } =0.
\end{equation}
Furthermore, using~\eqref{eqn : AK81/expectaction_limit} again yields
\begin{equation}
	\label{eqn : AK81/cvg_expect}
	\lim_{\ell \to \infty} \E{\frac{X(a\ell, b\ell)}{\ell} - Y(a,b) } =0.
\end{equation}
Combining~\eqref{eqn : AK81/cvg_negative_part} and~\eqref{eqn : AK81/cvg_expect}, we obtain the $\rL^1$ convergence.
\end{proof}



\end{document}